\documentclass[11pt,a4paper,reqno]{amsart}
\pdfoutput=1
\usepackage{amsmath,amsthm,latexsym,amssymb,wasysym}
\usepackage{numprint}
\usepackage[margin=1.0in,tmargin=0.9in,bmargin=0.80in]{geometry}
\usepackage{hyperref}
\usepackage{bbm}
\usepackage{xargs}
\usepackage{enumerate}
\usepackage{enumitem}
\usepackage{subcaption}
\setlist[enumerate]{label={\upshape(\roman*)}}
\usepackage[textsize=footnotesize]{todonotes}
\usepackage[flushleft]{threeparttable}
\usepackage{multirow}
\usepackage{longtable}
\usepackage{times}
\usepackage{hyperref}
\usepackage{microtype}
\usepackage[numbers]{natbib}
\usepackage[normalem]{ulem}

\npthousandsep{,}
\definecolor{bred}{rgb}{0.8,0,0}
\hypersetup{colorlinks,linkcolor={blue},citecolor={bred},urlcolor={blue}}

\newtheorem{theorem}{Theorem}[section]
\newtheorem{proposition}[theorem]{Proposition}
\newtheorem{lemma}[theorem]{Lemma}
\newtheorem{corollary}[theorem]{Corollary}
\newtheorem{remark}[theorem]{Remark}

\newtheorem{assumption}{Assumption}

\def\rmd{\mathrm{d}}
\newcommand{\E}{\mathbb{E}}
\newcommand{\N}{\mathbb{N}}
\newcommand{\R}{\mathbb{R}}
\newcommand{\1}{\mathbbm{1}}
\newcommand{\lfrf}[1]{\left\lfloor #1 \right\rfloor}
\newcommand{\lcrc}[1]{\left\lceil #1 \right\rceil}

\DeclareMathOperator{\sech}{sech}

\newcommand{\cF}{\mathcal{F}}
\newcommand{\cS}{\mathcal{S}}
\newcommand{\0}{\mathbf{0}}
\newcommand{\cU}{\mathcal{U}}
\newcommand{\cG}{\mathcal{G}}
\newcommand{\cNN}{\mathcal{NN}}

\newcommand{\bz}{\mathbf z}

\allowdisplaybreaks

\begin{document}

\title[]{Langevin dynamics based algorithm e-TH$\varepsilon$O POULA for stochastic optimization problems with discontinuous stochastic gradient}

\author[D.-Y. Lim]{Dong-Young Lim}
\author[A. Neufeld]{Ariel Neufeld}
\author[S. Sabanis]{Sotirios Sabanis}
\author[Y. Zhang]{Ying Zhang}

\address{Department of Industrial Engineering, UNIST, 
Ulsan, South Korea}
\email{dlim@unist.ac.kr}

\address{Division of Mathematical Sciences, Nanyang Technological University, 
Singapore}
\email{ariel.neufeld@ntu.edu.sg}

\address{School of Mathematics, The University of Edinburgh, 
Edinburgh, 
UK
\& The Alan Turing Institute, 
London, 
UK \& National Technical University of Athens, Athens, 
Greece}
\email{s.sabanis@ed.ac.uk}

\address{Financial Technology Thrust, Society Hub, The Hong Kong University of Science and Technology (Guangzhou), 
Guangzhou, China}
\email{yingzhang@hkust-gz.edu.cn}

\date{}
\thanks{
Financial supports by The Alan Turing Institute, London under the EPSRC grant EP/N510129/1, the MOE AcRF Tier~2 Grant \textit{MOE-T2EP20222-0013}, 
the European Union’s Horizon 2020 research and innovation programme under the Marie Skłodowska-Curie grant agreement No.\ 801215, the University of Edinburgh Data-Driven Innovation programme, part of the Edinburgh and South East Scotland City Region Deal, Institute of Information \& communications Technology Planning \& Evaluation (IITP) grant funded by the Korea government (MSIT) (No.\ 2020-0-01336, Artificial Intelligence Graduate School Program (UNIST)),  National Research Foundation of Korea (NRF) grant funded by the Korea government (MSIT) (No.\ RS-2023-00253002), and the Guangzhou-HKUST(GZ) Joint Funding Program (No. 2024A03J0630) are gratefully acknowledged.}

\begin{abstract}
We introduce a new Langevin dynamics based algorithm, called e-TH$\varepsilon$O POULA, to solve optimization problems with discontinuous stochastic gradients which naturally appear in real-world applications such as quantile estimation, vector quantization, CVaR minimization, and regularized optimization problems involving ReLU neural networks. We demonstrate both theoretically and numerically the applicability of the e-TH$\varepsilon$O POULA algorithm. More precisely, under the conditions that the stochastic gradient is locally Lipschitz \textit{in average} and satisfies a certain convexity at infinity condition, we establish non-asymptotic error bounds for e-TH$\varepsilon$O POULA in Wasserstein distances and provide a non-asymptotic estimate for the expected excess risk, which can be controlled to be arbitrarily small. Three key applications in finance and insurance are provided, namely, multi-period portfolio optimization, transfer learning in multi-period portfolio optimization, and insurance claim prediction, which involve neural networks with (Leaky)-ReLU activation functions. Numerical experiments conducted using real-world datasets illustrate the superior empirical performance of e-TH$\varepsilon$O POULA compared to SGLD, TUSLA, ADAM, and AMSGrad in terms of model accuracy. 
\end{abstract}
\maketitle

\section{Introduction}\label{sec:intro}
A wide range of problems in economics, finance, and quantitative risk management can be represented as stochastic optimization problems. Traditional approaches to solve such problems typically face the curse of dimensionality in practical settings, which motivates researchers and practitioners to apply machine learning approaches to obtain approximated solutions. Consequently, deep learning have been widely adopted to almost all aspects in, e.g., financial applications including option pricing, implied volatility, prediction, hedging, and portfolio optimization \citep{becker:19,boudabsa2022machine,buehler:19,chen:21,fernandez2020machine,han:18,imajo:21,liu:19,neufeld2021deep,neufeld2022detecting,sirig:18,20:tsang}, and applications in insurance \cite{chen2020managing,frey2022deep,groll2022churn,guelman:12,jin:21,kaushik2022machine,kshirsagar2021accurate,matthews2022machine,perla2021time,wang2021neighbouring,wuthrich2020bias,yang:18}. While the aforementioned results justify the use of deep neural networks through the universal approximation theorem, it is not a trivial problem to train a deep neural network, which is equivalent to minimizing an associated loss function, using efficient optimization algorithms. Stochastic gradient descent (SGD) and its variants are popular methods to solve such non-convex and large scale optimization problems. However, it is well known that SGD methods are only proven to converge to a stationary point in non-convex settings. Despite the lack of theoretical guarantees for the SGD methods, the literature on deep learning in finance, insurance, and their related fields heavily rely on popular optimization methods such as SGD and its variants including, e.g., ADAM \cite{kingma:15} and AMSGrad \cite{j.2018on}. In \cite{jai:21}, the author explicitly highlights the importance of research on stochastic optimization methods for problems in finance:
`The choice of optimisation engine in deep learning is vitally important in obtaining sensible results, but a topic rarely discussed (at least within the financial mathematics community)'. The aim of this paper is thus to bridge the theoretical gap and to extend the empirical understanding of training deep learning models in applications in finance and insurance. We achieve these by investigating the properties of a newly proposed algorithm, i.e., the extended Tamed Hybrid $\varepsilon$-Order POlygonal Unadjusted Langevin Algorithm (e-TH$\varepsilon$O POULA), which can be applied to optimization problems with discontinuous stochastic gradients including quantile estimation, vector quantization, CVaR minimization, and regularized optimization problems involving ReLU neural networks, see, e.g., \cite{4,fort2016,lim2021nonasymptotic,rockafellar2000optimization}.

We consider the following optimization problem:
\begin{equation}\label{eq:obju}
\text{minimize} \quad \R^d \ni \theta \mapsto u(\theta) := \E[U(\theta, X)],
\end{equation}
where $U: \R^d \times \R^m \rightarrow \R$ is a measurable function, and $X$ is a given $\R^m$-valued random variable with probability law $\mathcal{L}(X)$. To obtain approximate minimizers of \eqref{eq:obju}, one of the approaches is to apply the stochastic gradient Langevin dynamics (SGLD) algorithm introduced in \cite{wt}, which can be viewed as a variant of the Euler discretization of the Langevin SDE defined on $t \in [0, \infty)$ given by
\begin{equation} \label{sdeintro}
\mathrm{d} Z_t=-h\left(Z_t\right) \mathrm{d} t+ \sqrt{2\beta^{-1}} \mathrm{d} B_t, \quad Z_0 = \theta_0,
\end{equation}
where $\theta_0$ is an $\R^d$-valued random variable, $h:= \nabla u$, $\beta>0$ is the inverse temperature parameter, and $(B_t)_{t \geq 0}$ is a $d$-dimensional Brownian motion. The associated stochastic gradient of the SGLD algorithm is defined as a measurable function $H:\R^d \times \R^m \to \R^d$ which satisfies $h(\theta) = \E[H(\theta, X)]$ for all $\theta \in \R^d$. One notes that, under mild conditions, the Langevin SDE \eqref{sdeintro} admits a unique invariant measure $\pi_{\beta}(\rmd \theta) \wasypropto \exp(-\beta u(\theta))\rmd \theta$ with $\beta>0$. It has been shown in \cite{hwang} that $\pi_{\beta}$ concentrates around the minimizers of $u$ when $\beta$ takes sufficiently large values. Therefore, minimizing \eqref{eq:obju} is equivalent to sampling from $\pi_{\beta}$ with large $\beta$. The convergence properties of the SGLD algorithm to $\pi_{\beta}$ in suitable distances have been well studied in the literature, under the conditions that the (stochastic) gradient of $u$ is globally Lipschitz continuous and satisfies a (local) dissipativity or convexity at infinity condition, see, e.g., \cite{nonconvex,berkeley,raginsky,xu,sgldloc} and references therein. Recent research focuses on the relaxation of the global Lipschitz condition imposed on the (stochastic) gradient of $u$ so as to accommodate optimization problems involving neural networks. However, the SGLD algorithm is unstable when applying to objective functions with highly non-linear (stochastic) gradients, and the absolute moments of the approximations generated by the SGLD algorithm could diverge to infinity at a finite time point, see \cite{hutzenthaler2011}. To address this issue, \cite{lovas2023taming} proposed a tamed unadjusted stochastic Langevin algorithm (TUSLA), which is obtained by applying the taming technique, developed in, e.g., \cite{tula,hutzenthaler2012,eulerscheme,SabanisAoAP}, to the SGLD algorithm. Convergence results of TUSLA are provided in \cite{lovas2023taming} under the condition that the stochastic gradient of $u$ is polynomially Lipschitz growing. In \cite{lim2021nonasymptotic}, the applicability of TUSLA is further extended to the case where the stochastic gradient of $u$ is discontinuous, and the polynomial Lipschitz condition is replaced by a more relaxed locally Lipschitz in average condition. The latter condition is similar to \cite[Eqn. (6)]{4} and \cite[H4]{fort2016}, which well accommodates optimization problems with ReLU neural networks. One may also refer to \cite{4,durmus2018efficient,durmus2019analysis,fort2016,luu2021sampling} for convergence results of the Langevin dynamics based algorithms with discontinuous (stochastic) gradients.

Despite their established theoretical guarantees, TUSLA and other Langevin dynamics based algorithms are less popular in practice, especially when training deep learning models, compared to adaptive learning rate methods including ADAM and AMSGrad. This is due to the superior empirical performance of the latter group of algorithms in terms of the test accuracy and training speed. In \cite{lim2021polygonal}, a new class of Langevin dynamics based algorithms, namely TH$\varepsilon$O POULA, is proposed based on the advances of polygonal Euler approximations, see \cite{krylov1985extremal,krylov1991simple}. More precisely, the design of TH$\varepsilon$O POULA relies on a combination of a componentwise taming function and a componentwise boosting function, which simultaneously address the exploding and vanishing gradient problems. Furthermore, such a design allows TH$\varepsilon$O POULA to convert from an adaptive learning rate method to a Langevin dynamics based algorithm when approaching an optimal point, preserving the feature of a fast training speed of the former and the feature of a good generalization of the latter. In addition, \cite{lim2021polygonal} provides a convergence analysis of TH$\varepsilon$O POULA for non-convex regularized optimization problems. Under the condition that the (stochastic) gradient is locally Lipschitz continuous, non-asymptotic error bounds for TH$\varepsilon$O POULA in Wasserstein distances are established, and a non-asymptotic estimate for the expected excess risk is provided. However, the local Lipschitz condition fails to accommodate optimization problems with discontinuous stochastic gradients.

In this paper, we propose the algorithm e-TH$\varepsilon$O POULA, which combines the advantages of utilizing Euler’s polygonal approximations of TH$\varepsilon$O POULA \cite{lim2021polygonal} resulting in its superior empirical performance, together with a relaxed condition on its stochastic gradient as explained below. We aim to demonstrate both theoretically and numerically the applicability of e-TH$\varepsilon$O POULA for optimization problems with discontinuous stochastic gradients. 
From a theoretical point of view, our goal is to provide theoretical guarantees for e-TH$\varepsilon$O POULA to find approximate minimizers of $u$ with discontinuous stochastic gradient. More concretely, we aim to relax the local Lipschitz condition, and replace it with a local Lipschitz \textit{in average} condition, see Assumption \ref{asm:AG}. 
In addition, \cite{lim2021polygonal} considers regularized optimization problems which assume a certain structure of the stochastic gradients of the corresponding objective functions. More precisely, \cite{lim2021polygonal} assumes that $u(\theta) := g(\theta)+\eta|\theta|^{2r+1}/(2r+1)$, $\theta \in \R^d$, where $g: \R^d \to \R$, $\eta>0$, and $r>0$. The second term on the RHS of $u$ is the regularization term, and the stochastic gradient of $u$, denoted by $H:\R^d \times \R^m \to \R^d$, is given by $H(\theta,x) = G(\theta,x) +\eta \theta |\theta|^{2r}$ where $\nabla_{\theta}g(\theta) = \E[G(\theta, X)]$. We aim to generalize the structure of $H$ by replacing $\eta \theta |\theta|^{2r}$ with any arbitrary function $F:\R^d\times \R^m \to \R^d$ which satisfies a local Lipschitz condition and a convexity at infinity condition, see \eqref{eq:expressiontGF} and Assumptions~\ref{asm:AF} and~\ref{asm:AC}. In our setting, the gradient of the regularization term is a particular feasible example for the choice of $F$. In addition to the aforementioned assumptions, by further imposing conditions on the initial value of e-TH$\varepsilon$O POULA and on the second argument of $H$, see Assumption \ref{asm:AI}, we establish non-asymptotic error bounds of e-TH$\varepsilon$O POULA in Wasserstein distances and a non-asymptotic upper estimate of the expected excess risk given by  $\E[u(\hat{\theta})] - \inf_{\theta \in \R^d} u(\theta)$ with $\hat{\theta}$ denoting an estimator generated by e-TH$\varepsilon$O POULA, which can be controlled to be arbitrarily small. From a numerical point of view, we illustrate the powerful empirical performance of e-TH$\varepsilon$O POULA by providing key examples in finance and insurance using real-world datasets, i.e., the multi-period portfolio optimization, transfer learning in the multi-period portfolio optimization, and the insurance claim prediction via neural network-based non-linear regression. Numerical experiments show that e-TH$\varepsilon$O POULA outperforms SGLD, TUSLA, ADAM, and AMSGrad in most cases\footnote{ while it performs as good as the best alternative method in the remaining cases.} with regard to test accuracy.

We conclude this section by introducing some notation. 
For $a,b \in \R$, denote by $a \wedge b = \min\{a,b\}$ and $a \vee b = \max\{a,b\}$. Let $(\Omega,\mathcal{F},P)$ be a probability space. We denote by $\E[Z]$  the expectation of a random variable $Z$. For $1\leq p<\infty$, $L^p$ is used to denote the usual space of $p$-integrable real-valued random variables. Fix integers $d, m \geq 1$. For an $\R^d$-valued random variable $Z$, its law on $\mathcal{B}(\R^d)$, i.e.\ the Borel sigma-algebra of $\R^d$, is denoted by $\mathcal{L}(Z)$. For a positive real number $a$, we denote by $\lfrf{a}$ its integer part, and $\lcrc{a} := \lfrf{a}+1 $. The Euclidean scalar product is denoted
by $\langle \cdot,\cdot\rangle$, with $|\cdot|$ standing for the corresponding norm (where the dimension of the space may vary depending on the context). For any integer $q \geq 1$, let $\mathcal{P}(\R^q)$ denote the set of probability measures on $\mathcal{B}(\R^q)$. For $\mu\in\mathcal{P}(\R^d)$ and for a $\mu$-integrable function $f:\R^d\to\R$, the notation $\mu(f):=\int_{\R^d} f(\theta)\mu(\rmd \theta)$ is used. For $\mu,\nu\in\mathcal{P}(\R^d)$, let $\mathcal{C}(\mu,\nu)$ denote the set of probability measures $\zeta$ on $\mathcal{B}(\R^{2d})$ such that its respective marginals are $\mu,\nu$. For two Borel probability measures $\mu$ and $\nu$ defined on $\R^d$ with finite $p$-th moments, the Wasserstein distance of order $p \geq 1$ is defined as
\[
{W}_p(\mu,\nu):=
\left(\inf_{\zeta\in\mathcal{C}(\mu,\nu)}\int_{\R^d}\int_{\R^d}|\theta-\bar{\theta}|^p\zeta(\rmd \theta, \rmd \bar{\theta})\right)^{1/p}.
\]

\section{e-TH$\varepsilon$O POULA: Setting and definition}
\subsection{Setting}\label{sec:theopoulasting}
Let $U: \R^d \times \R^m \rightarrow \R$ be a Borel measurable function, and let $X$ be an $\R^m$-valued random variable defined on the probability space $(\Omega,\mathcal{F},P)$ with probability law $\mathcal{L}(X)$ satisfying $\E[|U(\theta, X)|]<\infty$ for all $\theta \in \R^d$. We assume that $u: \R^d \rightarrow \R$ defined by $u(\theta) := \E[U(\theta, X)]$, $\theta \in \R^d$, is a continuously differentiable function, and denote by $h:=\nabla u$ its gradient. In addition, for any $\beta>0$, we define
\begin{equation}\label{eq:pibetaexp}
\pi_{\beta}(A) := \frac{\int_A e^{-\beta u(\theta)} \, \rmd \theta}{\int_{\R^d} e^{-\beta u(\theta)} \, \rmd \theta}, \quad A \in \mathcal{B}(\R^d),
\end{equation}
where we assume $\int_{\R^d} e^{-\beta u(\theta)} \, \rmd \theta <\infty$.

Denote by $(\mathcal{G}_n)_{n\in\N_0}$ a given filtration representing the flow of past information, and denote by $\mathcal{G}_{\infty} := \sigma(\bigcup_{n \in \N_0} \mathcal{G}_n)$. Moreover, let $(X_n)_{n\in\N_0}$ be a $(\mathcal{G}_n)$-adapted process such that $(X_n)_{n\in\N_0}$ is a sequence of i.i.d. $\R^m$-valued random variables with probability law $\mathcal{L}(X)$. In addition, let $(\xi_{n})_{n\in\N_0}$ be  a sequence of independent standard $d$-dimensional Gaussian random variables.  We assume throughout the paper that the $\R^d$-valued random variable $\theta_0$ (initial condition), $\mathcal{G}_{\infty}$, and $(\xi_{n})_{n\in\N_0}$ are independent.

Let $H: \R^d \times \R^m \rightarrow \R^d$ be an unbiased estimator of $h$, i.e., $h(\theta) = \E[H(\theta, X_0)]$, for all $\theta \in \R^d$, which takes the following form: for all $\theta \in \R^d, x \in \R^m$,
\begin{equation}\label{eq:expH}
H(\theta,x) := G(\theta,x)+F(\theta,x),
\end{equation}
where $G = (G^{(1)}, \dots, G^{(d)}):  \R^d \times \R^m \rightarrow \R^d$ is Borel measurable and $F = (F^{(1)}, \dots, F^{(d)}):  \R^d \times \R^m \rightarrow \R^d$ is continuous.

\begin{remark}\label{rmk:sturctureHtoyexamples} We consider $H$ taking the form of \eqref{eq:expH} with $G$ containing discontinuities and $F$ being locally Lipschitz continuous (see also Assumptions \ref{asm:AG} and \ref{asm:AF} in Section \ref{sec:main}) as it is satisfied by a wide range of real-world applications including quantile estimation, vector quantization, CVaR minimization,
and regularized optimization problems involving ReLU neural networks, see, e.g., \cite{4,fort2016,lim2021nonasymptotic,rockafellar2000optimization}. For illustrative purposes, we provide concrete examples for each of the applications mentioned above:
\begin{enumerate}
\item \label{item:examplei}For quantile estimation, we aim to identify the $\mathsf{q}$-th quantile of a given distribution $\mathcal{L}(X)$. To this end, we consider the following regularized optimization problem:
\[
\text{minimize} \quad \R \ni \theta \mapsto u(\theta):=  \mathbb{E}\left[l_{\mathsf{q}}(X-\theta)\right] +\frac{\eta}{2(r+1)}|\theta|^{2(r+1)},
\]
where $0<\mathsf{q}<1$, $\eta>0, r\geq 0$ are regularization and growth constants, respectively, and
\[l_{\mathsf{q}}(z) =	\begin{cases}
     			 \mathsf{q}z, & z \geq 0, \\
      			 (\mathsf{q}-1)z,  & z<0.
  			 \end{cases}
\]
Then, we have that $H(\theta,x) := G(\theta,x)+F(\theta,x)$ with $\theta \in \R, x  \in \R$,
\[
F(\theta,x):=\eta\theta|\theta|^{2r}, \quad G(\theta,x):= -\mathsf{q} +\mathbbm{1}_{\{x < \theta\}}.
\]
\item For vector quantization, our aim is to optimally quantize a given $\R^d$-valued random vector $X$ by an $\R^d$-valued random vector taking at most $N\in \N$ values. For the ease of notation, we consider the case $d=1$. For any $\theta=(\theta^{(1)},\dots,\theta^{(N)})\in \R^N$ we define the associated Voronoi cells as
\[
\mathcal{V}^{(i)}(\theta) :=\left\{x\in\R: |x-\theta^{(i)}| = \min_{j\in \{1,\dots,N\}}|x-\theta^{(j)}|\right\},\quad i = 1,\dots,N.
\]
Then, we quantize the values of $X$ in $\mathcal{V}^{(i)}(\theta)$ to $\theta^{(i)}$ in the following way. We consider minimizing the mean squared quantization error:
\[
\text{minimize} \quad \R^N \ni \theta \mapsto u(\theta):=  \sum_{i=1}^N\mathbb{E}\left[|X-\theta^{(i)}|^2\mathbbm{1}_{\mathcal{V}^{(i)}(\theta)}(X)\right] +\frac{\eta}{2(r+1)}|\theta|^{2(r+1)},
\]
where $\eta>0, r\geq 0$. This implies that $H(\theta,x) := G(\theta,x)+F(\theta,x)$ with $\theta \in \R^N, x \in \R$,
\[
F (\theta, x) := \eta\theta|\theta|^{2r}, \quad G(\theta,x) := (G^{(1)}(\theta,x),\dots, G^{(N)}(\theta,x)),
\]
where, for $i = 1, \dots, N$,
\[
G^{(i)}(\theta,x)= -2(x-\theta^{(i)})\mathbbm{1}_{\mathcal{V}^{(i)}(\theta)}(x).
\]
We note that, in the case where $X\sim \text{Uniform}[0,1]$ and $N=2$, Voronoi cells take the form $\mathcal{V}^{(1)}(\theta) = [0,(\theta^{(1)}+\theta^{(2)})/2]$ and $\mathcal{V}^{(2)}(\theta) = [(\theta^{(1)}+\theta^{(2)})/2,1]$.
\item \label{item:exampleiii}For CVaR minimization, we consider the problem of obtaining VaR and obtaining optimal weights which minimize CVaR of a given portfolio consisting of $N\in\N$ assets, i.e., we consider
\[
\text{minimize} \quad \R^{N+1} \ni \theta \mapsto u(\theta):=  \mathbb{E}\left[ \frac{1}{1-\mathsf{q}}\left(\sum_{i=1}^{N}g_i(w)X^{(i)}-\overline{\theta}\right)_+ + \overline{\theta} \right]+\frac{\eta}{2(r+1)}|\theta |^{2(r+1)},
\]
where $\theta:=(\overline{\theta},w)=(\overline{\theta}, w^{(1)}, \dots,w^{(N)})\in \R^{N+1}$, for each $i = 1, \dots, N$, $X^{(i)} \in \R$ denotes the loss of the $i$-th asset, $g_i:\R^N \rightarrow \R$ denotes the (parameterized) weight of the $i$-th asset with $g_i(w) := \frac{e^{w^{(i)}}}{\sum_{j=1}^{N}e^{w^{(j)}}} \in (0,\,1)$, $0<\mathsf{q}<1$, $(x)_+:=\max\{0,x\}$ for $x\in \R$, $\eta>0$, and $r\geq 0$. Then, we have that $H(\theta,x) := G(\theta,x)+F(\theta,x)$ with $\theta \in \R^{N+1}, x \in \R^N$,
\[
F (\theta, x) := \eta\theta|\theta|^{2r}, \quad G(\theta,x) := (G_{\overline{\theta}}(\theta,x), G_{w^{(1)}}(\theta,x),\dots,G_{w^{(N)}}(\theta,x)),
\]
where for $i = 1, \dots, N$,
\begin{align*}
G_{\overline{\theta}}(\theta,x)&:= 1- \frac{1}{1-\mathsf{q}}\1_{\{\sum_{i=1}^{N}g_i(w)x^{(i)}\geq \overline{\theta}\}} , \\
G_{w^{(j)}}(\theta,x)&:= \frac{1}{1-\mathsf{q}}\sum_{i=1}^N\partial_{w^{(j)}} g_i(w)x^{(i)}\1_{\{\sum_{i=1}^{N}g_i(w)x^{(i)} \geq \overline{\theta}\}}.
\end{align*}
\item \label{item:exampleiv}For the regularized optimization problems involving ReLU neural networks, we consider an example of identifying the best regularized mean-square estimator\footnote{For the ease of presentation, we consider the case where the input and target variables are both one dimensional. For the multi-dimensional version, we refer to Section \ref{sub:transfer_learning} and the corresponding Proposition \ref{prop:optim_tl_reg_vasm}.}. We consider the following regularized optimization problem:
\begin{equation*}
\text{minimize} \quad \R^2 \ni \theta \mapsto u(\theta) := \E[(Y-\mathfrak{N}(\theta, Z))^2]+\frac{\eta}{2(r+1)}|\theta|^{2(r+1)},
\end{equation*}2.
where $\mathfrak{N}: \R^2 \times \R \to \R$ is the neural network given by
\[
\mathfrak{N}(\theta, z) := \mathsf{K}_1\sigma_1(\mathsf{c}_0z+\mathsf{b}_0),
\]
with $\mathsf{K}_1$ the weight parameter, $\sigma_1(y) = \max\{0, y\}$, $y \in \R$, the ReLU activation function, $\mathsf{c}_0$ the fixed (pre-trained non-zero) input weight, $z$ the input data, $\mathsf{b}_0$ the bias parameter, and where $\theta =(\mathsf{K}_1, \mathsf{b}_0) \in \R^2$ is the parameter of the optimization problem, $Y$ is the $\R$-valued target random variable, $Z$ is the $\R$-valued input random variable, and $\eta, r>0$. Then, we have that $H(\theta,x) := G(\theta,x)+F(\theta,x)$ with $\theta \in \R^2, x = (y, z) \in \R^2$,
\[
F (\theta, x) := \eta\theta|\theta|^{2r}, \quad G(\theta,x) := (G_{\mathsf{K_1}}(\theta,x), G_{\mathsf{b}_0}(\theta,x))
\]
where
\begin{align*}
G_{\mathsf{K}_1}(\theta,x)
&= -2(y-\mathfrak{N}(\theta, z))\sigma_1(\mathsf{c}_0z+\mathsf{b}_0),\\
G_{\mathsf{b}_0}(\theta,x)
& = -2(y-\mathfrak{N}(\theta, z))\mathsf{K}_1\mathbbm{1}_{\{z\geq -\mathsf{b}_0/\mathsf{c}_0\}}.
\end{align*}
\end{enumerate}
We note that all the examples \ref{item:examplei}-\ref{item:exampleiv} satisfy Assumptions \ref{asm:AI}-\ref{asm:AC} in Section \ref{sec:assumption}, see, e.g., \cite{4,fort2016,lim2021nonasymptotic,sglddiscont} for detailed proofs, and hence can be solved using e-TH$\varepsilon$O POULA with its performance backed by theoretical results presented in Section \ref{sec:mr}. While examples \ref{item:examplei}-\ref{item:exampleiii} are presented to illustrate the wide applicability of e-TH$\varepsilon$O POULA, we focus in this paper on a general case of example \ref{item:exampleiv} in Section \ref{sub:transfer_learning} and demonstrate the superior empirical performance of e-TH$\varepsilon$O POULA in Section \ref{sec:numapp} compared to other alternatives including SGLD, TUSLA, ADAM, and AMSGrad.
\end{remark}

\subsection{Algorithm} We define the extended Tamed Hybrid $\varepsilon$-Order POlygonal Unadjusted Langevin Algorithm (e-TH$\varepsilon$O POULA) by
\begin{equation}\label{eq:theopoula}
\theta^{\lambda}_0 :=\theta_0, \quad \theta^{\lambda}_{n+1}:=\theta^{\lambda}_n-\lambda H_\lambda(\theta^{\lambda}_n,X_{n+1})+ \sqrt{2\lambda\beta^{-1}} \xi_{n+1},\quad  n\in\N_0,
\end{equation}
where $\lambda>0$ is the stepsize, $\beta>0$ is the inverse temperature parameter, and where 
$H_\lambda(\theta,x) $ is defined, for all $\theta \in \R^d, x \in \R^m$, by
\begin{equation} \label{eq:expressiontH}
H_\lambda(\theta,x):=G_\lambda(\theta,x)+F_\lambda(\theta,x),
\end{equation}
with $G_\lambda(\theta,x)= (G_\lambda^{(1)}(\theta,x), \dots, G_\lambda^{(d)}(\theta,x))$ and $F_\lambda(\theta,x) =  (F_\lambda^{(1)}(\theta,x), \dots, F_\lambda^{(d)}(\theta,x))$ given by
\begin{align} \label{eq:expressiontGF}
\begin{split}
G_\lambda^{(i)}(\theta,x)&:=\frac{G^{(i)}(\theta,x)}{1+\sqrt{\lambda}|G^{(i)}(\theta,x)|}\left(1+\frac{\sqrt{\lambda}}{\varepsilon+|G^{(i)}(\theta,x)|}\right), \quad F_\lambda^{(i)}(\theta, x): = \frac{F^{(i)}(\theta,x)}{1+\sqrt{\lambda}|\theta|^{2r}},
\end{split}
\end{align}
for any $i = 1, \dots, d$ with fixed $0<\varepsilon <1$, $r>0$.
\begin{remark}\label{rmk:theopoulacmt} Recall that the general form of the stochastic gradient Langevin dynamics (SGLD) algorithm is given by
\begin{equation}\label{eqn:sgld}
\theta^{\mathsf{SGLD}}_0 :=\theta_0, \quad \theta^{\mathsf{SGLD}}_{n+1}:=\theta^{\mathsf{SGLD}}_n-\lambda H(\theta^{\mathsf{SGLD}}_n,X_{n+1})+ \sqrt{2\lambda\beta^{-1}} \xi_{n+1},\quad  n\in\N_0.
\end{equation}
Therefore, e-TH$\varepsilon$O POULA is obtained by replacing $H$ in the SGLD algorithm with $H_\lambda$ given in \eqref{eq:expressiontH}-\eqref{eq:expressiontGF}. More precisely, one part of $H_\lambda$, i.e., $F_\lambda$, is obtained by multiplying $F$ with the taming factor $1+\sqrt{\lambda}|\theta|^{2r}$, while the other part of $H_\lambda$, i.e., $G_\lambda$, is defined by dividing $G$ componentwise with the taming factor $1+\sqrt{\lambda}|G^{(i)}(\theta,x)|$ and, importantly, with the boosting function $1+\frac{\sqrt{\lambda}}{\varepsilon+|G^{(i)}(\theta,x)|}$. One observes that, when $|G^{(i)}(\theta,x)|$ is small, the boosting function takes a large value, which, in turn, contributes to the step-size and helps prevent the vanishing gradient problem which occurs when the stochastic gradient is extremely small resulting in insignificant updates of the algorithm before reaching an optimal point, while the boosting function is close to one when $|G^{(i)}(\theta,x)|$ is large.  Moreover, the design of $H_{\lambda}$ is motivated by the regularized optimization problems. In such a setting, $F$ corresponds to the gradient of the regularization term, and $G$ corresponds to the gradient of the original (non-regularized) objective function of a given optimization problem. The boosting function, together with the componentwise design of $G_\lambda$, thus significantly improve the training efficiency of e-TH$\varepsilon$O POULA as demonstrated numerically in Section~\ref{sec:numapp}.
\end{remark}


\section{Numerical Experiments}\label{sec:numapp}

This section demonstrates the performance of e-TH$\varepsilon$O POULA by applying it to real-world applications arising in finance and insurance. In Section~\ref{sub:port_op}, we apply e-TH$\varepsilon$O POULA to approximately solve the problem of portfolio selection studied in \cite{20:tsang} using neural networks, where e-TH$\varepsilon$O POULA is used for the training of the neural networks. Data sets are generated from popular models in finance such as Black-Scholes and Autoregressive models. Then, Section~\ref{sub:transfer_learning} discusses a transfer learning setting, based on the dynamic programming principle, in the context of portfolio selection with theoretical guarantees for the convergence of our proposed algorithm. Next, in Section~\ref{sub:gamma}, we consider a neural network-based non-linear regression to predict insurance claims where French auto insurance claim data is used. Finally, in Section~\ref{sub:simulationsummary}, we provide a summary of our numerical results as well as a brief discussion on the optimal choice of optimization algorithms. Source code for all the experiments can be found at \url{https://github.com/DongyoungLim/eTHEOPOULA}.

\subsection{Multi-period portfolio optimization}\label{sub:port_op}

This subsection discusses a deep learning approach proposed in \cite{20:tsang} to solve multi-period portfolio optimization problems. The idea of the approach is to view a given portfolio optimization problem as a Markov Decision Process (MDP), and then approximate the optimal policy function of the MDP by means of neural networks. We train the corresponding neural networks using e-TH$\varepsilon$O POULA and showcase its performance also in comparison with other popular optimization algorithms for the training of neural networks. 


Fix $K>0$. Assume that the financial market is defined on a filtered probability space $(\Omega, \{\cF_k\}_{k=0}^K, P)$, where $\cF_0 = \{\emptyset, \Omega\}$, with finite time horizon $[0,K]$ where assets can be traded at discrete time points, $k=0, 1, 2, \dots, K-1$. For each time point $k$, denote by $R_k \in \R^p$ the excess return vector of $p$ risky assets between the period $[k. k+1)$, whereas the risk free return is denoted by $R_f$. Moreover, denote by $W_k\in\R$ the wealth of the portfolio at time point $k$. For positive integers $d_s$ and $p$, denote by $\cS \subseteq \R^{d_s}$ the set of possible states and $D\subseteq \R^p$ the set of possible actions representing the proportion of current wealth invested in each risky asset. Then,  for any $k=0, 1,\ldots, K-1$, the evolution of the wealth between the time points $k$ and $k+1$ is given by
\begin{equation*}
W_{k+1} = W_k ( \langle g_k(s_k), R_k\rangle +R_f),
\end{equation*}
where $g_k(\cdot) :\cS \rightarrow D$ is the investment control policy function on $p$ risky assets at time point $k$ and $s_k \in \cS$ is the state at time point $k$ which is $\cF_{k}$-measurable. Moreover, we denote by $\cU$ the set of admissible\footnote{The functions in $\cU$ may satisfy certain bounding constraints, e.g., $D=\Pi_{i=1}^p [l_i, u_i]$ for some lower bounds $l=(l_1, \ldots, l_p)$ and upper bounds $u=(u_1,\ldots,u_p)$. } control functions.

In this setting, we are interested in finding the optimal portfolio selection of $p$ risky assets which maximizes the expected utility function of the terminal wealth $W_K$. The expected utility maximization problem can be written as an MDP problem as follows:
\begin{eqnarray}
\mathcal{V}_K(s_0)&=&\max_{g_0, \ldots, g_{K-1} \in \cU} \E[\Psi(s_K) ] \label{eq:optim_mdp}\\
\mbox{s.t.} &\quad& s_{k+1} = \bar{h}(s_k, g_k(s_k), \eta_k),\quad k=0, 1, \ldots, K-1, \nonumber 
\end{eqnarray}
where $\Psi(\cdot):\cS \rightarrow \R$ is the objective function, $ \bar{h}: \cS \times D \times \R^{\overline{m}} \rightarrow \cS$ with $\overline{m}>0$ is the transition function, and $\eta_k $ is an $\R^{\overline{m}}$-valued $\cF_{k+1}$-measurable random variable. We assume that each of the $\R^{d_s}$-valued state variable $s_k$ contains (in one component) the wealth $W_k$, see, e.g., the autoregressive (AR) (1) model below. Furthermore, we set the quadratic utility function as the objective function, which is given by $\Psi(s_K):= U(W_K)=-(W_K- \frac{\gamma}{2})^2$ for some fixed $\gamma>0$. 

We solve the MDP problem \eqref{eq:optim_mdp} via the deep learning approach proposed in \cite{20:tsang}. We briefly introduce the approach to make our paper self-contained. Denote by $\cG_\nu $ the set of standard feedforward neural networks with two hidden layers, which is given explicitly by
\begin{align}\label{def:tlfn}
\begin{split}
\cG_\nu &= \{f: \R^{d_s} \rightarrow \R^p | f(x) = \tanh(K_3 z + b_3),  z=\sigma(K_2 y + b_2), \\
 &\qquad y=\sigma(K_1 x + b_1), K_1 \in \R^{\nu \times {d_s}}, K_2 \in \R^{\nu \times \nu}, K_3 \in \R^{p \times \nu}, b_1,b_2 \in \R^{\nu}, b_3 \in \R^{p} \},
\end{split}
\end{align}
where $\nu$ denotes the number of neurons on each layer of the neural network, $\tanh(x)$, for any $x\in \R^p$, is the hyperbolic tangent function at $x$ applied componentwise, and $\sigma(y) = \max\{0, y\}$, $y \in \R^{\nu}$, is the ReLU activation function at $y$ applied componentwise.

For any matrix $M \in \R^{a_M\times b_M}$ with $a_M, b_M>0$, denote by $[M]$ the vector of all elements in $M$. Moreover, for any $k = 0,1, \dots, K-1$, denote by $g_k(\cdot; \theta_k):\R^{d_s} \rightarrow \R^p$ the approximated policy function at time $k$ using a neural network with its structure defined in with \eqref{def:tlfn}, where $\theta_k = (b_1, b_2, b_3, [K_1], [K_2], [K_3]) \in \R^{\nu(d_s+\nu+p+2)+p}$ denotes the parameter of the neural network. Then\footnote{Note that $\mathcal{V}_K(s_0)$ and $V_K(s_0)$ only differ by the sign, hence, up to the approximation error, solving \eqref{eq:optim_mdp} is equivalent to solving \eqref{eq:optim_nn}.}, the MDP problem~\eqref{eq:optim_mdp} can be approximated by restricting\footnote{We refer to \cite{20:tsang} for the verification of the approximation.} $g_k(\cdot; \theta_k) \in\cG_\nu $:
\begin{align}\label{eq:optim_nn}
\begin{split}
-\mathcal{V}_K(s_0) =:V_K(s_0) \approx V_K^*(s_0) 
&= \min_{\theta}\E[-\Psi(s_K^\cNN(\zeta;\theta))],
\end{split}
\end{align}
where $\zeta := (s_0, \eta_0, \ldots, \eta_{K-1})$ denotes the vector of the initial state variable and all the random variables throughout the trading time horizon $[0,K]$, and where $s_K^\cNN(\zeta;\theta)$ is recursively defined, for $k=0, 1, \dots, K-1$, by
\begin{equation}\label{eq:propsed_nn}
 s_{k+1} =  \bar{h}(s_k, g_k(s_k;\theta_k), \eta_k), \quad g_k(\cdot;\theta_k)\in \cG_\nu
\end{equation}
with $ s_K^{\cNN}(\zeta; \theta): = s_K $, $\theta = (\theta_0, \ldots, \theta_{K-1})\in \R^{d}$ being the parameter for the neural networks, $d:=K(\nu(d_s +\nu +p +2)+ p)$, and $\cG_\nu$ given\footnote{In the implementation stage, one might need to perform suitable scalar addition and multiplication for the neural networks in $\cG_\nu$ so that they also satisfy the bounding constraints specified for functions in $\cU$.} in \eqref{def:tlfn}. 

We test the performance of e-TH$\varepsilon$O POULA in comparison with other popular stochastic optimization algorithms such as SGLD defined in \eqref{eqn:sgld}, ADAM, and AMSGrad by solving the optimization problem~\eqref{eq:optim_nn} under two different asset return models: the (discrete-time version of) Black-Scholes model and the AR($1$) model. 
Moreover, we provide extensive numerical experiments with different market parameters and different sizes of neurons to demonstrate the efficiency of our algorithm.

\paragraph{\textbf{Black-Scholes model.}}
For any matrix $M \in \R^{p\times p}$, denote by $\mbox{diag}(M):=(M_{11}, \ldots, M_{pp})$ the vector of the diagonal elements of $M$, and denote by $M^\top$ its transpose. Denote by $I_p$ the $p\times p$ identity matrix. For any $k = 0, \dots, K-1$, we consider the following (discrete-time) Black-Scholes model analyzed in \cite{20:tsang} for the excess return $R_k$:
\begin{equation}\label{eqn:bser}
R_k = \exp\left(\left(\widetilde{r} \1 + \Sigma \widetilde{\lambda} -\frac{1}{2}\mbox{diag}(\Sigma\Sigma^\top)\right)\Delta + \sqrt{\Delta} \Sigma \epsilon_k\right) - R_f \1,
\end{equation}
where $\widetilde{r} \in \R$, $\1 = (1,\ldots ,1) \in \R^p$, $\Sigma \in \R^{p\times p}$, $\widetilde{\lambda} \in \R^p$, $\Delta>0$ is a constant rebalancing time period, $\epsilon_k$, $k=0, \dots, K-1$, are i.i.d.\ $p$-dimensional Gaussian vectors with mean $\0$ and covariance $ I_p$, i.e., $\epsilon_k \sim N_p(\0, I_p)$, and $R_f := \exp(\widetilde{r}\Delta)$ denotes the risk free return. In this setting, the excess returns $\{R_k\}_{k=0}^{K-1}$ are i.i.d.. Then, the equivalent optimization problem to the MDP problem \eqref{eq:optim_mdp} in the Black-Scholes model can be written as follows:
\begin{eqnarray}
V_K(s_0)&=&\min_{g_0, \ldots, g_{K-1} \in \cU} \E[-U(W_K)  ]   = \min_{g_0, \ldots, g_{K-1} \in \cU} \E[(W_K- \gamma/2)^2 ] \label{eq:optim_bs} \\
\mbox{s.t.} &\quad& W_{k+1} = W_k ( \langle g_k(s_k), R_k\rangle +R_f),\quad k=0, 1, \ldots, K-1. \nonumber
\end{eqnarray}
where $s_k := W_k$, $ {d_s} := 1$, $\eta_k:=R_k$, $\overline{m} := p$, and $\bar{h}(s_k, g_k(s_k), \eta_k) := W_k ( \langle g_k(s_k), R_k \rangle +R_f)$.

We approximate the optimization problem \eqref{eq:optim_bs} using the deep learning approach \eqref{eq:optim_nn}  where we train $K$ neural networks involved in \eqref{eq:propsed_nn} with each of the neural networks defined explicitly in \eqref{def:tlfn}. In other words, we have
\begin{equation}\label{eqn:bsfl}
V_K(s_0) \approx V_K^*(s_0)= \min_{\theta}\E\left[\left(s_K^\cNN(\zeta;\theta)-\gamma/2\right)^2 \right],
\end{equation}
where $\zeta := (W_0, R_0, \ldots, R_{K-1})$ and where $s_K^\cNN(\zeta;\theta)$ is defined recursively as in \eqref{eq:propsed_nn} with $s_K^{\cNN}(\zeta; \theta): = s_K =W_K$, $s_k := W_k$, $k=0, 1, \ldots, K-1$, and $\bar{h}(s_k, g_k(s_k;\theta_k), \eta_k) := W_k ( \langle g_k(s_k;\theta_k), R_k \rangle +R_f)$. Three different simulation settings of the Black-Scholes model are summarized in Table~\ref{tab:iid}.  Similar to \cite{20:tsang}, we run our models for $200$ steps\footnote{Following \cite{20:tsang}, we use here the term ``step'' to indicate ``epoch''.} with batch size of $128$. For each step, $\numprint{20000}$ training samples are generated and $157$ iterations ($=\lceil \numprint{20000}/128 \rceil$) are performed to train the models. Then, the test score is computed using $\numprint{50000}$ test samples. In addition, three different numbers of $\nu$ are tested for each experimental setting: for $p=5$: $\nu = \{1,5, 10\}$; for $p=50$: $\nu = \{1, 5, 20\}$; and for $p=100$: $\nu = \{1, 5, 20\}$.

\begin{table}[t]
\centering
\begin{tabular}{c | c  c  c }
 \hline
$p$ & $5$ & $50$ & $100$ \\ \hline \hline
$\widetilde{r}$ &$ 0.03$ & $0.03$ &$ 0.03 $\\ \hline
$\Delta$ & $1/40 $& $1/40 $& $1/30$\\ \hline
K &$ 40 $ & $ 40 $ & $30$ \\  \hline
$W_0$ & $1$ &$1$ &$1$ \\ \hline
$\gamma$ & $4$& $5$&$6$ \\  \hline
$D$ & $[0, 1.5]^p$ & $[0, 1.5]^p$ & $[0, 0.5]^p$ \\ \hline
\multirow{2}{*}{$\widetilde{\lambda}$} & $\widetilde{\lambda}_i=0.1$ for $i=1,2$ & $\widetilde{\lambda}_i=0.01$ for $i=1,\ldots, 25$ & $\widetilde{\lambda}_i=0.01$ for $i=1,\ldots, 50$ \\
 & $\widetilde{\lambda}_i=0.2$ for $i=3,4,5$ & $\widetilde{\lambda}_i=0.05$ for $i=26,\ldots, 50$ & $\widetilde{\lambda}_i=0.05$ for $i=51,\ldots, 100$ \\  \hline
\multirow{2}{*}{$\Sigma$} & $\Sigma_{ii}=0.15$ & $\Sigma_{ii}=0.15$  & $\Sigma_{ii}=0.15$  \\
 & $\Sigma_{ij}=0.01$  for $i\neq j$ & $\Sigma_{ij}=0.005$  for $i\neq j$& $\Sigma_{ij}=0.0025$  for $i\neq j$\\  \hline
\end{tabular}
\caption{Parameters for optimization problem \eqref{eq:optim_bs}.}
\label{tab:iid}
\end{table}

\begin{table}[t]
\centering
\begin{tabular}{c | c  c  c | c  c  c}
 \hline
 & \multicolumn{3}{c}{test score}  & \multicolumn{3}{|c}{training speed} \\ \hline
 & \multicolumn{3}{c}{$p=5$} & \multicolumn{3}{|c}{$p=5$}  \\ \hline
$\nu$ & $1$ & $5$ & $10$ & $1$ & $5$ & $10$ \\ \hline
SGLD & $0.845$ & $0.836$ & $0.835$& NA $(\numprint{1284})$ & NA $(\numprint{1284})$ & NA ($\numprint{1297}$)\\
TUSLA & $0.852$ & $0.84$ & $0.839$ & NA ($\numprint{1445}$)& NA ($\numprint{1465}$)& NA ($\numprint{1470}$) \\
ADAM & $0.832$ & $0.825$ & $0.822$ & 14 ($\numprint{1375}$)  & 69 ($\numprint{1379}$) & 69 ($\numprint{1385}$) \\
AMSGrad & $0.833$ & $0.825$ & $0.822$ & 14 ($\numprint{1418}$) & 71 ($\numprint{1417}$) & 71 ($\numprint{1420}$) \\
{\color{blue} e-TH$\varepsilon$O POULA}& ${\color{blue}0.832}$ & ${\color{blue}0.824}$ & ${\color{blue}0.822}$ &  {\color{blue} 32 ($\numprint{1606}$)}  & {\color{blue}104 ($\numprint{1598}$)} & {\color{blue}88 ($\numprint{1603}$)}\\
HJB solution (benchmark) & $0.821$ & $0.821$ & $0.821$ & - & - & -\\  \hline
 & \multicolumn{3}{c}{$p=50$} & \multicolumn{3}{|c}{$p=50$}\\  \hline
$\nu$ & $1$ & $5$ & $20$ & $1$ & $5$ & $20$ \\ \hline
SGLD & $2.176$ & $2.079$ & $2.056$ & NA ($\numprint{1582}$)  & NA ($\numprint{1587}$)& NA ($\numprint{1590}$) \\
TUSLA & $2.401$ & $2.207$ & $2.097$ & NA ($\numprint{1776}$)  & NA ($\numprint{1779}$)  & NA ($\numprint{1772}$) \\
ADAM & $2.048$ & $2.039$ & $2.038$ & 67 ($\numprint{1666}$) & 135 ($\numprint{1682}$) & 142 ($\numprint{1674}$) \\
AMSGrad & $2.049$ & $2.040$ & $2.039$ & 68 ($\numprint{1709}$)  & 129 ($\numprint{1718}$) & 173 ($\numprint{1727}$)  \\
{\color{blue} e-TH$\varepsilon$O POULA} & ${\color{blue}2.049}$ &$ {\color{blue}2.042}$ & ${\color{blue}2.041}$ & {\color{blue}  76 ($\numprint{1903}$)} & {\color{blue} 96 ($\numprint{1915}$) } & {\color{blue} 154 ($\numprint{1930}$) } \\
HJB solution (benchmark) & $2.032$ & $2.032$ & $2.032$ & - & - & - \\  \hline
 & \multicolumn{3}{c}{$p=100$} & \multicolumn{3}{|c}{$p=100$}\\  \hline
$\nu$ & $1$ &$ 5$ & $20$ & $1$ & $5$ & $20$ \\ \hline
SGLD & $3.690$ & $3.581$ & $3.527$ & NA ($\numprint{1334}$) & NA ($\numprint{1339}$) & NA ($\numprint{1338}$) \\
TUSLA & $4.693$ & $3.852$ & $3.636$ & NA ($\numprint{1523}$) & NA ($\numprint{1544}$) & NA ($\numprint{1535}$) \\
ADAM & $3.541$ & $3.491$ & $3.487$ & 147 ($\numprint{1403}$)  & 170 ($\numprint{1416}$)  & 204 ($\numprint{1409}$) \\
AMSGrad & $3.556$ & $3.496$ &$ 3.489$ & 217 ($\numprint{1444}$)  & 197 ($\numprint{1459}$) & 167 ($\numprint{1451}$) \\
{\color{blue} e-TH$\varepsilon$O POULA}& ${\color{blue}3.539}$ &${\color{blue}3.500 }$ &$ {\color{blue}3.496}$ & {\color{blue} 95 ($\numprint{1585}$) } & {\color{blue} 153 ($\numprint{1611}$) } & {\color{blue} 183 ($\numprint{1593}$) } \\
HJB solution (benchmark) & $3.460$ &$ 3.460$ & $3.460$ & - & - & - \\
\hline
\end{tabular}
\caption{Test score $V_K^*(s_0)$ and two metrics for training speed under the Black-Scholes model. In the `training speed' column, we report the first time (measured in seconds) when each optimizer reaches a score within a $1\%$ difference from the lowest best score over all the optimizers for each experiment. `NA' means that the optimizer does not achieve a difference of less than $1\%$ from the lowest best score even after $200$ epochs. In addition, the number in each parenthesis indicates the time (measured in seconds) required to train the model for $200$ epochs.
}
\label{tab:bs}
\end{table}


For e-TH$\varepsilon$O POULA, we find the best hyperparameters among the following choices: $\lambda = \{0.1, 0.05,\\ 0.01\}$, $\epsilon = \{10^{-2}, 10^{-4}, 10^{-8}, 10^{-12}\}$, and $\beta=10^{12}$. For SGLD and TUSLA, we use the following hyperparameters: $\lambda=\{0.5, 0.1, 0.05, 0.01\}$ and $\beta=10^{12}$. For ADAM and AMSGrad, the best learning rate is chosen among $\lambda=\{0.1, 0.01, 0.001\}$ with other hyperparameters $\epsilon =10^{-8}$, $\beta_1 = 0.9$, and $\beta_2 = 0.999$ being fixed. The learning rate is decayed by $10$ after $50$ steps for all the optimization algorithms. 

In \cite{20:tsang}, the authors have approximately solved the optimization problem \eqref{eqn:bsfl} using ADAM. Following \cite{20:tsang}, we have also included in Table~\ref{tab:bs} the values of the solution of the Hamilton-Jacobi-Bellman (HJB) equation which were calculated in \cite{20:tsang}. As highlighted in \cite{20:tsang}, the solution of the HJB equation can be interpreted as the continuous-time analog of our discrete-time optimization problem and provides values which are lower than the ones of the discrete-time optimization problem. However, since there is no benchmark algorithm for the discrete-time optimization problem we are considering, we follow \cite{20:tsang} and still include the values obtained from the HJB solution.

Figure~\ref{fig:iid} plots learning curves of all the optimization algorithms for different configurations of $(p, \nu)$. Table~\ref{tab:bs} shows the best test score $V_K^*(s_0)$, defined in \eqref{eq:optim_nn}, of each optimization algorithm where $V_K^*(s_0) \approx V_K(s_0)$ with $ V_K(s_0)$ defined in \eqref{eq:optim_bs}. As shown in Figure~\ref{fig:iid} and Table~\ref{tab:bs}, SGLD performs worst across all the experiments. On the other hand, e-TH$\varepsilon$O POULA achieves similar test scores as ADAM and AMSGrad.

We also compare the training speed of the optimization algorithms using two different metrics. First, we report the first time (measured in seconds) when each optimizer reaches a score within a 1\% difference from the lowest best score over \textit{all} optimizers, i.e., SGLD, TUSLA, ADAM, AMSGrad, and e-TH$\varepsilon$O POULA. For example, in the case of $p=100$, $\nu=1$, the lowest best score is $3.539$ attained by e-TH$\varepsilon$O POULA. SGLD and TUSLA do not get close to within 1\% of the lowest best score for the 200 epochs. ADAM and AMSGrad achieves values within 1\% of the lowest best score after $147$ and $217$ seconds have elapsed, respectively. Second, we report the time (measured in seconds) it takes for each optimizer to train the neural network for $200$ epochs. These two metrics for training speed are summarized in Table~\ref{tab:bs}. Although training the model for $200$ epochs with e-TH$\varepsilon$O POULA takes approximately 15\% longer compared to ADAM, it reaches the best score faster or as fast as the other optimizers.

\begin{figure}[t]
    \centering
    \begin{subfigure}[b]{0.32\textwidth}
        \includegraphics[width=\textwidth]{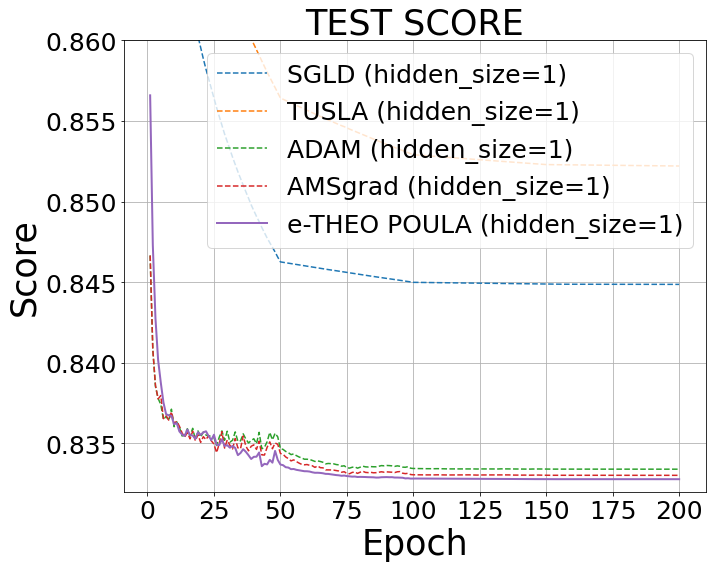}
        \caption{$p=5$ and $\nu=1$}
    \end{subfigure}
    \begin{subfigure}[b]{0.32\textwidth}
        \includegraphics[width=\textwidth]{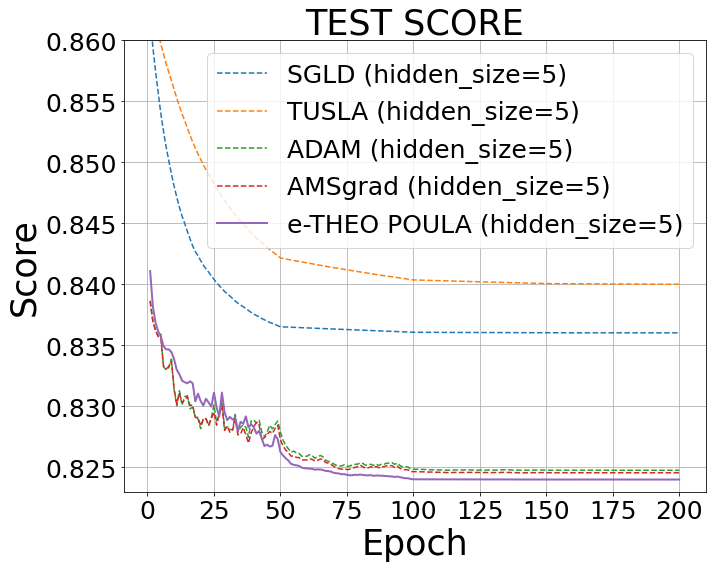}
        \caption{$p=5$ and $\nu=5$}
    \end{subfigure}
    \begin{subfigure}[b]{0.32\textwidth}
        \includegraphics[width=\textwidth]{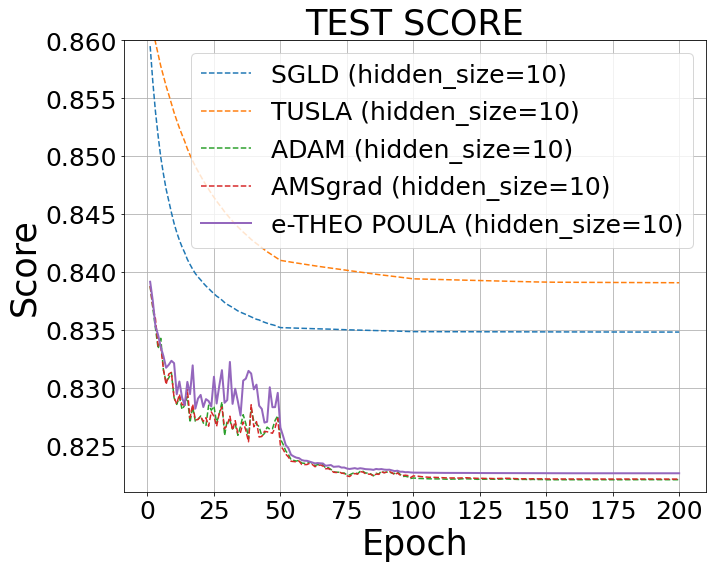}
        \caption{$p=5$ and $\nu=10$}
    \end{subfigure}

    \begin{subfigure}[b]{0.32\textwidth}
        \includegraphics[width=\textwidth]{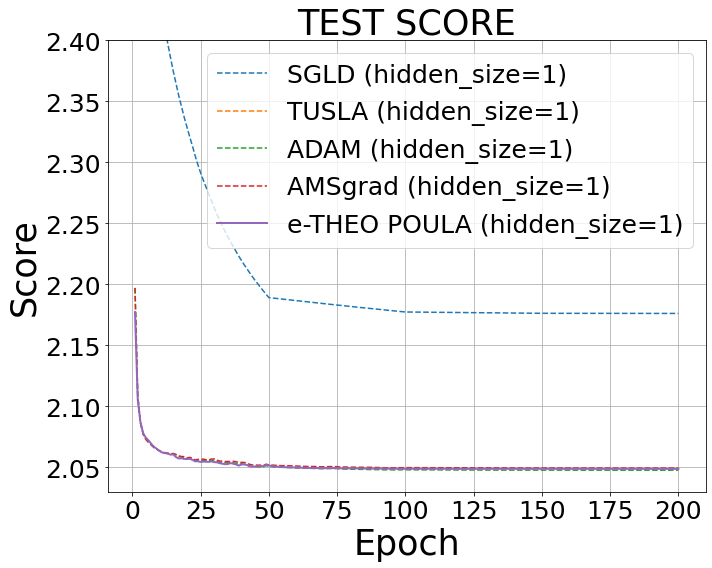}
        \caption{$p=50$ and $\nu=1$}
    \end{subfigure}
    \begin{subfigure}[b]{0.32\textwidth}
        \includegraphics[width=\textwidth]{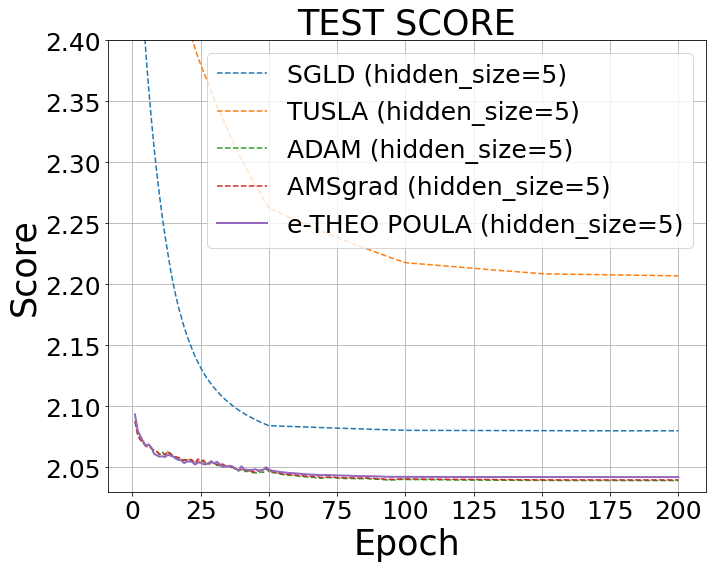}
        \caption{$p=50$ and $\nu=5$}
    \end{subfigure}
    \begin{subfigure}[b]{0.32\textwidth}
        \includegraphics[width=\textwidth]{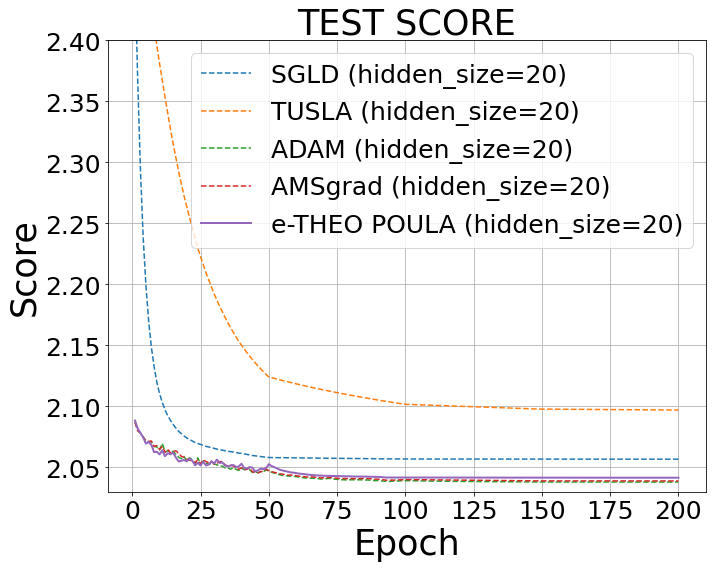}
        \caption{$p=50$ and $\nu=20$}
    \end{subfigure}

    \begin{subfigure}[b]{0.32\textwidth}
        \includegraphics[width=\textwidth]{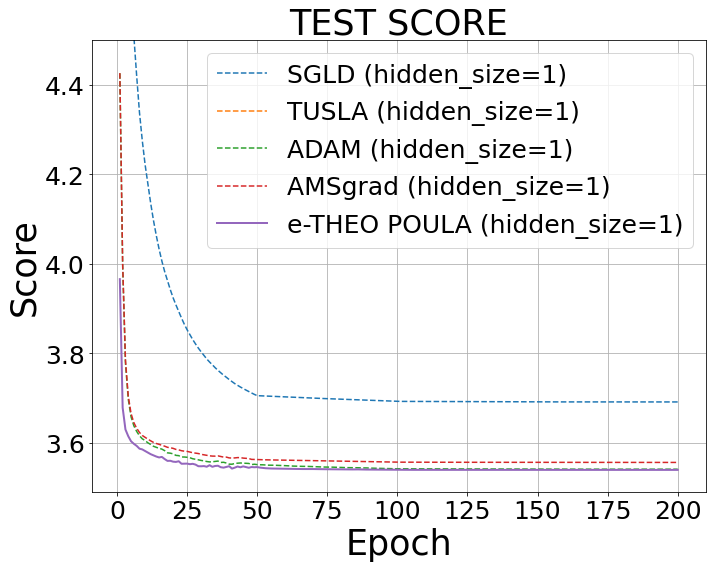}
        \caption{$p=100$ and $\nu=1$}
    \end{subfigure}
    \begin{subfigure}[b]{0.32\textwidth}
        \includegraphics[width=\textwidth]{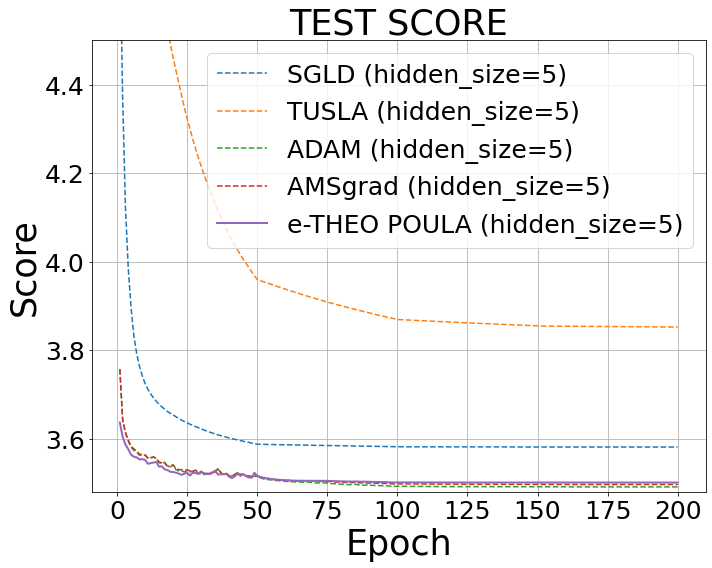}
        \caption{$p=100$ and $\nu=5$}
    \end{subfigure}
    \begin{subfigure}[b]{0.32\textwidth}
        \includegraphics[width=\textwidth]{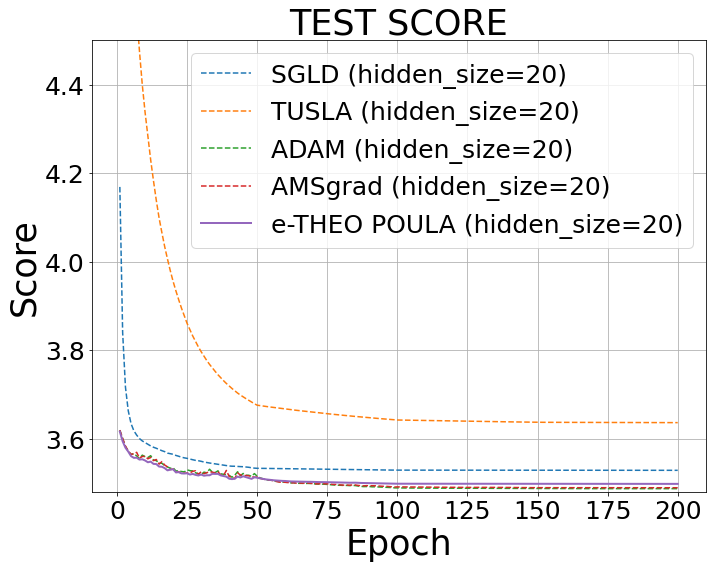}
        \caption{$p=100$ and $\nu=20$}
    \end{subfigure}

    \caption{Test score $V_K^*(s_0)$ of each optimizer for different number of assets under the Black-Scholes model. The parameter settings are summarized in Table~\ref{tab:iid}.}
    \label{fig:iid}
\end{figure}

\paragraph{\textbf{AR($1$) model.}}
We consider the following AR($1$) model:
\begin{equation}\label{eq:ar}
R_k =  \alpha + AR_{k-1} + \overline{\epsilon}_k, \quad k=0,1,\ldots, K-1,
\end{equation}
where $\alpha \in \R^p$, $A\in \R^{p\times p}$, and $\overline{\epsilon}_k \sim N_p(\0, \overline{\Sigma})$, $k = 0, \dots, K-1$, are i.i.d.\ with $\overline{\Sigma}\in \R^{p\times p}$. One observes that, in this setting, the excess returns $\{R_k\}_{k=0}^{K-1}$ are serially dependent. Thus, under the AR($1$) model \eqref{eq:ar}, the MDP problem \eqref{eq:optim_mdp} is reformulated as follows:
\begin{eqnarray}\label{eq:mdp_ar}
&&V_K(s_0)=\min_{g_0, \ldots, g_{K-1} \in \cU} \E[-U(W_K) ]  \\
\mbox{s.t.} &\quad& s_{k+1} = (W_k(\langle g_k(s_k), R_k\rangle + R_f), R_k),\quad k=0, 1, \ldots, K-1. \nonumber
\end{eqnarray}
where $s_k := (W_k, R_{k-1})$ is the augmented state variable such that the state transition is Markovian, ${d_s}:=p+1$, $\eta_k := \epsilon_k$, and $\overline{m} := p$.

We aim to approximate the optimization problem \eqref{eq:mdp_ar} using the deep learning approach \eqref{eq:optim_nn}. More precisely, for numerical experiments, we consider $\{R_k\}_{k=0}^{K-1}$ satisfying the AR($1$) model \eqref{eq:ar} with $p=30$, $K=10$, $\alpha=(0.015, \ldots, 0.015)\in \R^p$, $R_{-1}=\frac{\alpha}{1.15}$, $A_{ii}=-0.15$, $A_{ij}=0$ for $i\neq j$, and $\Sigma_{ii}=0.0238$, $\Sigma_{ij}=0.0027$ for $i\neq j$. We fix $W_0=1$, $R_f=1.03$, $\gamma = 15$ and $D=[0,1]^p$. Moreover, the training scheme is similar to that in the case of the Black-Schole model, but $\numprint{40000}$ training samples (instead of $ \numprint{20000}$) are used for each step. 
For the AR($1$) model \eqref{eq:ar}, numerical or analytical benchmark values are not available. Therefore, we only report and compare the test scores obtained from the following four different optimization algorithms: e-TH$\varepsilon$O POULA, SGLD, ADAM, and AMSGrad.

We use the same hyperparameters as that in the case of the Black-Scholes model for tuning the optimization algorithms, and then record the best test score among all the combinations of hyperparameters for each algorithm. Furthermore, we use three different values, i.e., $5, 20$, and $50$, for the number of neurons $\nu$ in the neural networks. In Figure~\ref{fig:ar}, we show the test scores of the different algorithms for each value of $\nu$. The best test score is reported in Table~\ref{tab:ar}, which shows that e-TH$\varepsilon$O POULA attains the lowest scores compared to SGLD, ADAM, and AMSGrad, as desired.

As in the Black-Scholes model, we provide the training speed of all optimization algorithms. Table~\ref{tab:ar} shows that while the total training time of e-TH$\varepsilon$O POULA takes approximately 20\% longer compared to ADAM, \textit{other optimizers failed to approach values within 1\% of the lowest best score achieved by e-TH$\varepsilon$O POULA} throughout the 200 epochs. This demonstrates in a relevant example that e-TH$\varepsilon$O POULA outperforms the other algorithms under consideration in terms of test accuracy.

\begin{figure}
    \centering
    \begin{subfigure}[b]{0.32\textwidth}
        \includegraphics[width=\textwidth]{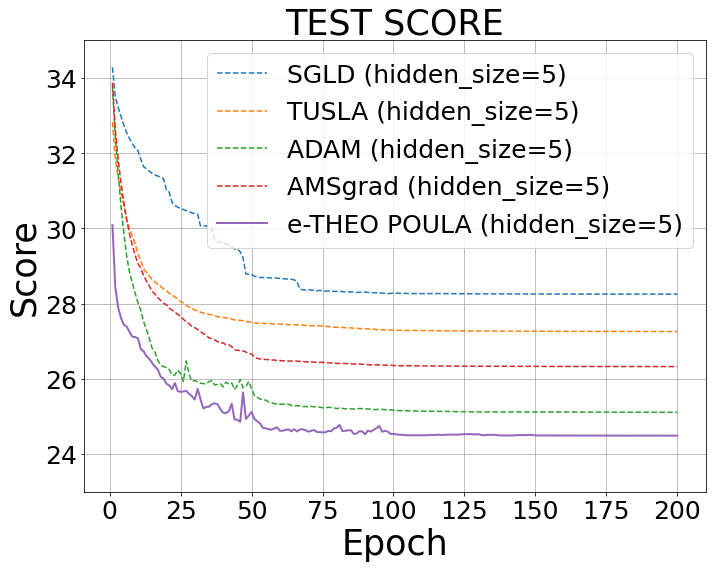}
        \caption{$\nu=5$}
    \end{subfigure}
    \begin{subfigure}[b]{0.32\textwidth}
        \includegraphics[width=\textwidth]{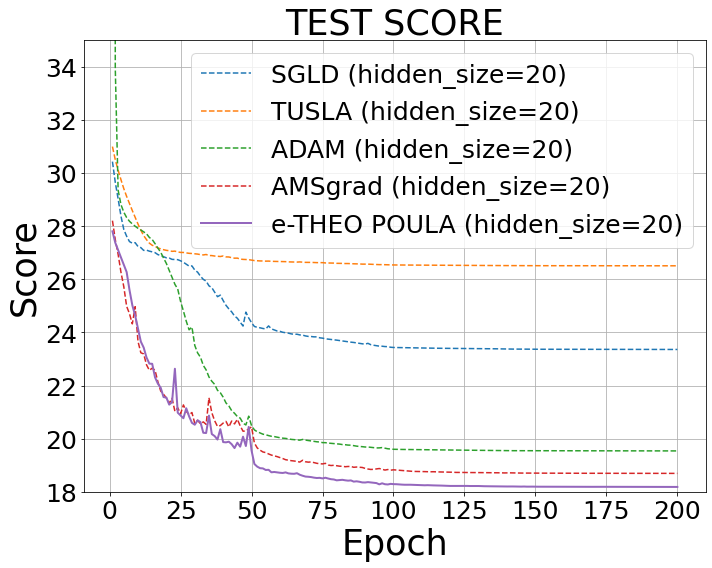}
        \caption{$\nu=20$}
    \end{subfigure}
    \begin{subfigure}[b]{0.32\textwidth}
        \includegraphics[width=\textwidth]{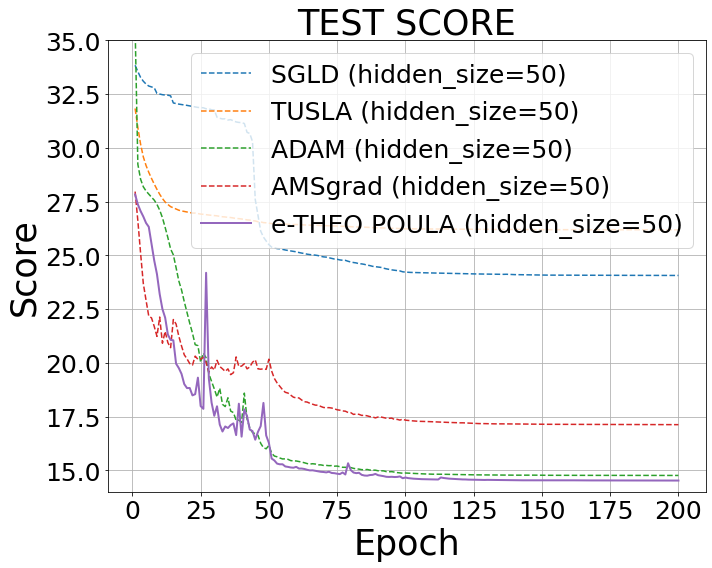}
        \caption{$\nu=50$}
    \end{subfigure}
    \caption{Test score $V_K^*(s_0)$ of each optimizer for different values of $\nu$ under the AR($1$) model.}
    \label{fig:ar}
\end{figure}


\begin{table}[t]
\centering
\begin{tabular}{c | c  c  c | c  c  c}
 \hline
 & \multicolumn{3}{c}{test score}  & \multicolumn{3}{|c}{training speed} \\ \hline
$\nu$ & $5$ & $20$ & $50$ & $5$ & $20$ & $50$ \\ \hline
SGLD & $28.255$ & $23.36$ & $24.06$& NA ($688$)  & NA ($686$)  & NA ($685$) \\
TUSLA & $27.26$ & $26.511$ & $26.189$ & NA ($795$)  & NA ($788$)  & NA ($787$) \\
ADAM & $25.113$ & $19.54$ & $14.76$ & NA ($711$) & NA ($729$)  & NA ($736$) \\
AMSGrad & $26.329$ & $18.693$ & $17.124$ & NA ($758$)  & NA ($754$) & NA ($756$) \\
{\color{blue} e-TH$\varepsilon$O POULA}& ${\color{blue}24.492}$ & ${\color{blue}18.183}$ & ${\color{blue}14.522}$ & {\color{blue}233 (852)}  & {\color{blue}373 (846)} & {\color{blue}431 (855)} \\ \hline
\end{tabular}
\caption{Test score $V_K^*(s_0)$ and two metrics for training speed under the AR(1) model. In the `training speed' column, we report the first time (measured in seconds) when each optimizer reaches a score within a $1\%$ difference from the lowest best score over all the optimizers for each experiment. `NA' means that the optimizer does not achieve a difference of less than $1\%$ from the lowest best score even after $200$ epochs. In addition, the number in each parenthesis indicates the time (measured in seconds) required to train the model for $200$ epochs.  }
\label{tab:ar}
\end{table}

\subsection{Transfer learning in the multi-period portfolio optimization}\label{sub:transfer_learning}

Transfer learning is a machine learning technique where knowledge gained from one task is reused to a related task by leveraging pre-trained models, which allows to save training time and often to achieve better performance \cite{jason14}. On the other hand, dynamic programming is a typical method to solve MDP problems which involve optimal decision making over multiple time steps, allowing to decompose the optimal decision problem over the entire time horizon into several simpler one-time-step optimization problems. This subsection discusses an interesting connection between the dynamic programming principle (DPP) and transfer learning, which allows us to present an example relevant in practice that can be solved using e-TH$\varepsilon$O POULA \eqref{eq:theopoula}-\eqref{eq:expressiontGF} with full theoretical guarantees ensuring its performance. More precisely, by considering an MDP problem as described in Section~\ref{sub:port_op} in a transfer learning setting, we show in Proposition \ref{prop:optim_tl_reg_vasm} that Theorem \ref{thm:opeer} and Corollary \ref{corollary:eerepsilon} can be applied to ensure the convergence of e-TH$\varepsilon$O POULA to the optimal solution of the aforementioned problem. The setting of the problem is given explicitly as follows.

\paragraph{\bf Transfer learning setting} Consider the single-hidden-layer feedforward network (SLFN) $\mathfrak{N}: \R^{\widetilde{d} } \times \R^{d_s} \rightarrow \R^{p}$ with its $i$-th element given by
\begin{equation}\label{eq:slfn}
\mathfrak{N}^i(\widetilde \theta, z) = \tanh\left(\sum_{j = 1}^\nu \widetilde{K}^{ij}_1 \sigma_1(\langle \overline{c}^{j\cdot}, z\rangle + \widetilde{b}_0^j) \right), \quad i = 1, \dots, p,
\end{equation}
where $z \in \R^{d_s}$ is the input vector, $\overline{c} \in \R^{\nu \times {d_s}}$ is the fixed (i.e.\ not trained) weight matrix, $\widetilde K_1 \in \R^{p \times \nu}$ is the weight parameter, $\widetilde b_0 \in \R^\nu$ is the bias parameter, $\widetilde \theta = ([\widetilde{K}_1], \widetilde{b}_0) \in \R^{\widetilde{d} }$ is the parameter of SLFN \eqref{eq:slfn} with $\widetilde{d}  = \nu(p+1)$, and $\sigma_1(y) = \max\{0, y\}$, $y \in \R$, is the ReLU activation function. In our numerical experiments, each element in $\bar c$ is generated by a standard normal distribution. We refer to \cite{cuchiero2020deep}, \cite{gonon2020approximation}, and \cite{neufeld2022chaotic} for the universal approximation property of neural networks with a randomly generated weight matrix. In addition, consider the set of two-hidden-layer feedforward network (TLFN) given by
\begin{align}\label{def:tlfn_sigmoid}
\begin{split}
\overline{\cG}_\nu &= \{f: \R^{d_s} \rightarrow \R^p | f(x) = \tanh(K_3 z + b_3),  z=\sigma_2(K_2 y + b_2), \\
&\qquad y=\sigma_2(K_1 x + b_1), K_1 \in \R^{\nu \times {d_s}}, K_2 \in \R^{\nu \times \nu}, K_3 \in \R^{p \times \nu}, b_1,b_2 \in \R^{\nu}, b_3 \in \R^{p} \}.
\end{split}
\end{align}
We note that TLFN \eqref{def:tlfn_sigmoid} has the same structure as that of TLFN \eqref{def:tlfn}, however, we use here the sigmoid activation function for TLFN \eqref{def:tlfn_sigmoid}, i.e., $\sigma_2(y) = 1/(1+e^{-y})$, $y \in \R^{\nu}$, which is applied componentwise, instead of the ReLU activation function for TLFN \eqref{def:tlfn}.

Fix $K>0$. We consider the case where the asset excess returns follow the Black-Scholes model in \eqref{eqn:bser}, which implies that $\{R_k\}_{k=0}^{K-1}$ are i.i.d.. Then, consider the time-indexed optimization problem of \eqref{eq:optim_bs}:
\begin{eqnarray}
V(t, K,W_t)&=&\min_{g_t, \ldots, g_{K-1} \in \cU} \E[(W_K-\gamma/2)^2 |\cF_t] \label{eq:optim_bsvar}\\
\mbox{s.t.} &\quad& W_{k+1} = W_k ( \langle g_k(W_k), R_k \rangle +R_f),\quad k=t, t+1, \ldots, K-1,\nonumber
\end{eqnarray}
where $t=0,\dots, K-1$. We note that $V(0,K, W_0)$ is the solution of the original problem \eqref{eq:optim_bs}, i.e., $V(0,K, W_0) = V_K(W_0)$, where \eqref{eq:optim_bs} is a special case of the MDP problem \eqref{eq:optim_mdp} with $s_k:=W_k$ for $k=0,1,\ldots, K-1$, $ {d_s} := 1$, $\eta_k:=R_k$, $\overline{m} := p$, and $\bar{h}(s_k, g_k(s_k), \eta_k) := W_k ( \langle g_k(s_k), R_k \rangle +R_f)$.

Denote by $\overline{V}_K^*(\cdot)$ the neural-network-based approximated solution of the MDP problem \eqref{eq:optim_bs} obtained using \eqref{eq:optim_nn}, where $s_K^\cNN(\zeta;\theta)$ recursively defined in \eqref{eq:propsed_nn} is replaced with  $\overline{s}_K^\cNN(\overline{\zeta};\theta)$ which is defined, for $k=0, 1, \dots, K-1$, by
\begin{equation}\label{eq:propsed_nntl}
 \overline{s}_{k+1} =  \bar{h}(\overline{s}_k, \overline{g}_k(\overline{s}_k;\theta_k), R_k), \quad \overline{g}_k(\cdot;\theta_k)\in \overline{\cG}_\nu
\end{equation}
with $\overline{s}_K^{\cNN}(\overline{\zeta}; \theta): = \overline{s}_K$,  $\overline{s}_k:=W_k$, $\overline{\zeta} := (W_0, R_0, \ldots, R_{K-1})$, $\theta = (\theta_0, \ldots, \theta_{K-1})\in \R^{d}$ being the parameter for the neural networks, $d:=K(\nu(d_s +\nu +p +2)+ p)$, and $\overline \cG_\nu$ given in \eqref{def:tlfn_sigmoid}. This implies that
\begin{align}\label{eq:optim_nntl}
V_K(W_0) =V(0,K, W_0)  \approx \overline{V}_K^*(W_0) = \min_{\theta}\E\left[\left(\overline{s}_K^\cNN(\overline{\zeta};\theta)-\gamma/2\right)^2 \right].
\end{align}
We consider the following transfer learning problem: we aim to compute $V(0, K+1, W_0)$ using $\overline{V}_K^*(\cdot)$, where the corresponding $K$ neural networks for $\overline{V}_K^*(\cdot)$ have been already trained to approximate $V_K(\cdot)$ as described in \eqref{eq:optim_nntl}. More precisely, we use the $K$ neural networks which have been already trained to obtain $\overline{V}^*_K(\cdot)\approx V(0, K,\cdot)$ in order to first approximate $V(1, K+1, W_1^{g_0})$ where $W_1^{g_0} = W_0 ( \langle g_0(W_0), R_0\rangle +R_f)$. This together with the DPP and the time-homogeneity of the MDP (see below for details) reduces our task to the training of only one SLFN \eqref{eq:slfn}, instead of $(K+1)$ neural networks involved in $V(0, K+1, W_0)$ as explained in Section~\ref{sub:port_op}. As a consequence, the training time is reduced significantly as illustrated in Table \ref{tab:tl_time}.

To concretely formulate the aforementioned procedures in transfer learning, we utilize two key ideas, i.e., DPP and the time-homogeneity property of MDP\footnote{See \cite{book:bert} and \cite{book:seier} for an overview of stochastic optimal control in discrete-time.}, which can be described explicitly as follows:
\begin{eqnarray}
 \mbox{(Dynamic programming principle) }& V(0, K, W_0) = \min_{g_0 \in \cU} \E[V(1, K, W_1^{g_0}) ], \label{eq:dpp}
\end{eqnarray}
where $W_1^{g_0} = W_0 ( \langle g_0(W_0), R_0\rangle +R_f)$, and
\begin{eqnarray}
 \mbox{(Time-homogeneity) }& V(t_1, t_2, s) =  V(0, t_2-t_1, s), \label{eq:th}
\end{eqnarray}
where $0\leq t_1 \leq t_2$ and $s \in \cS$. By using the time-homogeneity property \eqref{eq:th}, we obtain $V(0, K, s)= V(1, K+1, s)$, for any $s\in \cS$, which implies that $\overline{V}_K^*(\cdot)$ is also an approximated solution of $V(1, K+1, \cdot)$. Then, using the DPP in \eqref{eq:dpp}, $V(0, K+1, W_0)$ can be rewritten as follows:
\begin{align}\label{eq:optim_tl}
V(0, K+1, W_0) &= \min_{g_{0} \in \cU} \E[V(1, K+1, W_1^{g_{0}}) ] \nonumber \\
&\approx \min_{g_{0} \in \cU} \E[\overline{V}^*_K(W_1^{g_{0}}) ] \nonumber \\
&\approx \min_{\widetilde \theta}\E\left[ \overline{V}^*_K\left(W_1^{\mathfrak{N}(\widetilde \theta, W_0)}\right)  \right]=: V_{K+1}^{*,\mathsf{tl}}(W_0),
\end{align}
where $W_1^{\mathfrak{N}(\widetilde \theta, W_0)}:= W_{0}(\langle\mathfrak{N}(\widetilde \theta, W_{0}), R_{0}\rangle +R_f)$ with $\mathfrak{N}$ defined in \eqref{eq:slfn}, and where $\overline{V}_K^*(\cdot)$  is the approximated solution of $V_K(\cdot)$ as described in \eqref{eq:optim_nntl} with $\overline{s}_K^\cNN(\overline{\zeta};\theta)$ specified in \eqref{eq:propsed_nntl}. We note that since $\overline{V}_K^*(\cdot)$ is deterministic, our task of approximately solving the MDP problem \eqref{eq:optim_tl} is equivalent to optimizing the parameters $\widetilde \theta \in \R^{\widetilde{d} }$ of SLFN \eqref{eq:slfn} with $\widetilde{d}  = \nu(p+1)$.

The next proposition shows that our theoretical convergence results for e-TH$\varepsilon$O POULA, provided in Section \ref{sec:main}, can be applied to a regularized version of \eqref{eq:optim_tl}. More precisely, we consider the following regularized optimization problem:
\begin{align} \label{eq:optim_tl_reg}
V(0, K+1, W_0)
& \approx V_{K+1}^{*,\mathsf{tlreg}}(W_0):=\min_{\widetilde{\theta}}\left(\E\left[ \overline{V}^*_K\left(W_1^{\mathfrak{N}(\widetilde{\theta}, W_0)}\right)  \right]+ \frac{\eta}{2(r+1)}|\widetilde{\theta}|^{2(r+1)}\right)
\end{align}
where $r\geq1/2$ and $\eta>0$. 
\begin{proposition}\label{prop:optim_tl_reg_vasm}The optimization problem \eqref{eq:optim_tl_reg} satisfies Assumptions \ref{asm:AI}-\ref{asm:AC} in Section \ref{sec:main}.
\end{proposition}
\begin{proof} See Appendix \ref{sec:optim_tl_reg_proof}.
\end{proof}
Thus, by using Proposition \ref{prop:optim_tl_reg_vasm}, Theorem \ref{thm:opeer} and Corollary \ref{corollary:eerepsilon} can be applied to the optimization problem \eqref{eq:optim_tl_reg}, which provide theoretical guarantees for e-TH$\varepsilon$O POULA \eqref{eq:theopoula}-\eqref{eq:expressiontGF} to find approximate minimizers of \eqref{eq:optim_tl_reg}. We refer to Section \ref{sec:main} for the precise non-asymptotic convergence bounds for e-TH$\varepsilon$O POULA.
\begin{remark} We would like to comment on the regularization term $\eta|\widetilde{\theta}|^{2(r+1)}/(2(r+1))$, $\widetilde{\theta}\in \R^{\widetilde{d}}$, added in the optimization problem \eqref{eq:optim_tl_reg}. Theoretically, by adding this term, (part of) the stochastic gradient of the objective function \eqref{eq:optim_tl_reg}, i.e., $F$ defined  in \eqref{eq:fixedfgexp1}, satisfies Assumption \ref{asm:AC}. This is crucial in obtaining an upper estimate for the expected excess risk as provided in Theorem \ref{thm:opeer}. Numerically, adding the aforementioned regularization term does not affect essentially the simulation results due to the smallness of the regularization parameter $\eta$ (e.g., $\eta$ is set to be $10^{-6}$ in the numerical experiments). This can also be seen from the numerical results in Table~\ref{tab:tl}, where we obtain similar results compared to those obtained by solving the original (unregularized) problem \eqref{eq:optim_tl} with $\eta=0$. 
\end{remark}

\paragraph{\bf Comparison with full learning setting} We compare the performance between the two training methods, i.e., \textit{full training} and \textit{transfer learning}.  More precisely, as discussed in Section~\ref{sub:port_op}, full training refers to approximating $V(0, K+1, W_0)$ using $(K+1)$ neural networks defined in \eqref{eq:propsed_nntl} with the structure of each neural network specified in \eqref{def:tlfn_sigmoid}, whereas transfer learning refers to approximately solving $V(0, K+1, W_0)$ using \eqref{eq:optim_tl} with $\overline{V}_K^*(\cdot)$ obtained using \eqref{eq:propsed_nntl}. 
It is worth emphasizing that the dimension $\widetilde{d}:=\nu(p+1)$ of the parameters in transfer learning described in \eqref{eq:optim_tl} is significantly smaller than that of the parameters in full learning, i.e., $d:=(K+1)(\nu(d_s +\nu +p +2)+ p)$.

To generate sample paths under the Black-Scholes model, in both full training and transfer learning settings, we use identical parameters as in Table~\ref{tab:iid} except that $W_0$ is uniformly distributed on $[0.99, 1.01]$, i.e., $W_0\sim \text{Uniform}([0.99, 1.01])$. Then, we compute $V(0, K+1, W_0) \equiv V_{K+1}(W_0)$ defined in \eqref{eq:optim_bs} with $K=40$ for $p=\{5, 50\}$ and $K=30$ for $p=100$. The hidden size $\nu$ for the SLFN and TLFN is specified as follows: $\nu=\{1, 5, 10\}$ for $p=5$ and $\nu=\{1,5,20\}$ for $p=\{50, 100\}$. Moreover, we set $r=1$ and $\eta = 10^{-6}$. All the models are trained by e-TH$\varepsilon$O POULA with the same hyperparameters as in Section~\ref{sub:port_op} and batch size of $128$. Then, the test score is computed using  $\numprint{50000}$ test samples. Table~\ref{tab:tl} summarizes the test scores, i.e., the approximated values $V_{K+1}^*(W_0)$, $V_{K+1}^{*,\mathsf{tl}}(W_0)$, and $V_{K+1}^{*,\mathsf{tlreg}}(W_0)$ of $V(0,K+1, W_0)$, computed from the full training, the transfer learning, and the transfer learning with regularization, respectively, while Table~\ref{tab:tl_dim} shows the number of parameters to be determined for each experiment. The results show that transfer learning yields a similar test score in comparison with that of full training while the dimension of parameters of the transfer learning is significantly lower than that of full training. In addition, we measure the training time for transfer learning and full training to demonstrate the computational efficiency of the former approach. Table~\ref{tab:tl_time} shows that the training time of transfer learning is at least three times faster than that of full learning.

\begin{table}[t]
\centering
\begin{tabular}{c | c  c  c  | c c c | c c c}
 \hline
 & \multicolumn{3}{c|}{$p=5$, $K=40$}  & \multicolumn{3}{c|}{$p=50$, $K=40$} & \multicolumn{3}{c}{$p=100$, $K=30$} \\  \hline
$\nu$ & $1$ & $5$ & $10$ & $1$ & $5$ & $20$  & $1 $& $5$ & $20$ \\ \hline
$V_{K+1}^*(W_0)$ & $0.830$ & $0.820$ & $0.818$  & $2.044$ & $2.037$ & $2.036$ & $3.530$ & $3.483$ & $3.481$ \\
$V_{K+1}^{*,\mathsf{tlreg}}(W_0)$ & $0.830$ & $0.820$ & $0.819$ & $2.047$ &$2.042$ &$2.041$ & $3.530 $& $3.491$ & $3.489$  \\
$V_{K+1}^{*,\mathsf{tl}}(W_0)$  & $0.830$ & $0.820$ & $0.819$ & $2.047$ & $2.042$ & $2.041$ & $3.530$ & $3.491$ & $3.489$ \\ \hline
\end{tabular}
\caption{Test scores for the full training, the transfer learning, and the transfer learning with regularization under the Black-Scholes model.}
\label{tab:tl}
\end{table}


\begin{table}[t]
\centering
\begin{tabular}{c | c  c  c | c c c | c c c }
 \hline
 & \multicolumn{3}{c|}{$p=5$, $K=40$} & \multicolumn{3}{c|}{$p=50$, $K=40$} & \multicolumn{3}{c}{$p=100$, $K=30$} \\  \hline
$\nu$ & $1 $& $5$ & $10$ & $1$ & $5$ & $20$ & $1$ & $5$ & $20$ \\ \hline
Full Training & $574$ & $\numprint{2870}$ & $\numprint{7585}$ & $\numprint{4264}$ &  $\numprint{13940}$ &  $\numprint{61910}$ &  $\numprint{6324}$ &  $\numprint{19840}$ &  $\numprint{79360}$\\
TL with reg& $6$ & $30$ & $60$ & $51$ & $255$ &  $\numprint{1020}$ & $101$ & $505$ &  $\numprint{2020}$ \\
TL & $6$ & $30$ & $60$ & $51$ & $255$ &  $\numprint{1020}$ & $101$ & $505$ &  $\numprint{2020}$ \\ \hline

\end{tabular}
\caption{Number of parameters for the transfer learning (with regularization) and full training. `TL with reg.' and `TL' stand for 'Transfer Learning with regularization' and 'Transfer Learning,' respectively.}
\label{tab:tl_dim}
\end{table}

\begin{table}[t]
\centering
\begin{tabular}{c | c  c  c | c c c | c c c}
 \hline
 & \multicolumn{3}{c|}{$p=5$, $K=40$}  & \multicolumn{3}{c|}{$p=50$, $K=40$} & \multicolumn{3}{c}{$p=100$, $K=30$}\\  \hline
$\nu$ & $1 $& $5$ & $10$ & $1$ & $5$ & $20$ & $1$ & $5$ & $20$\\ \hline
$V_{K+1}^*(W_0)$ & $593.6$ & $596.7$ & $594.6$ & $599.4$ & $594.1$ & $606.5$ & $440.5$ & $465.1 $& $424.5$ \\
$V_{K+1}^{*,\mathsf{tlreg}}(W_0)$ & $173.0$ & $173.2$ & $173.4$ & $171.3 $& $170.4$ & $176.8$  & $143.1$ & $156.0$ & $158.2 $\\
$V_{K+1}^{*,\mathsf{tl}}(W_0)$& $178.7$ & $178.9 $ & $182.0$ & $180.2 $& $184.8$ &$ 186.3$ & $123.3$ &$ 131.4$ & $160.5$\\ \hline
\end{tabular}
\caption{Training time (measured in seconds) for the full training, the transfer learning, and the transfer learning with regularization.}
\label{tab:tl_time}
\end{table}


\subsection{Non-linear Gamma regression}\label{sub:gamma}
In this subsection, we consider optimization problems involving Gamma regression models. 
We are interested in non-linear Gamma regression problems which extends the linear Gamma regression model by replacing its linear regressor function with a neural network in order to incorporate non-linear relations of the input variables. This approach is widely used in insurance business to predict insurance claim sizes, see, e.g., \cite{frees:14,garrido:16,renshaw:94,yang:18}.


Here, we provide an example of a non-linear Gamma regression model based on neural networks which can be used to predict a target variable $Y\in (0, \infty)$ given an input variable $ Z\in \R^{m}$. Under the assumption that $Y$ follows a certain Gamma distribution, its logarithmic mean function can be estimated by minimizing the negative log-likelihood (NLL) function associated with its density function \cite{garrido:16}. We then train a neural network to approximately solve this minimization problem. More precisely, in this setting, we assume that $Y$ follows the Gamma distribution with mean $\mu \in  (0, \infty)$ and log-dispersion $\phi \in \R$. Denote by $f_Y: (0, \infty) \to  (0, \infty)$ the probability density function of $Y$ given explicitly by
\[
f_Y(y; \mu, \phi) \equiv f_Y(y):= \frac{1}{y\Gamma(\exp(-\phi))}\left(\frac{y}{\mu\exp(\phi)}\right)^{\exp(-\phi)} e^{-\frac{y\exp(-\phi)}{\mu}}, \quad y \in  (0, \infty),
\]
where $\Gamma(\exp(-\phi))$ denotes the gamma function evaluated at $\exp(-\phi)$. Moreover, we consider the following TLFN $\widehat{\mathfrak{N}}: \R^d\times \R^m \to \R$:
\begin{equation}\label{eq:TLFNgr}
\widehat{\mathfrak{N}}(\theta,  z) :=   \widehat{K}_3\sigma_3\left( \widehat{K}_2\sigma_3\left( \widehat{K}_1 z + \widehat{b}_1\right)+\widehat{b}_2\right)+\widehat{b}_3,
\end{equation}
where $z \in \R^m$ is the input vector, $\theta=( \widehat{K}_1,  \widehat{K}_2,  \widehat{K}_3,  \widehat{b}_1, \widehat{b}_2, \widehat{b}_3)\in \R^d$ is the parameter with $d=d_1(m+d_2+1) + 2d_2+1$, $ \widehat{K}_1 \in \R^{d_1 \times m}$, $ \widehat{K}_2 \in \R^{d_2 \times d_1}$, $ \widehat{K}_3 \in \R^{1 \times d_2}$,  $\widehat{b}_1 \in \R^{d_1}$, $\widehat{b}_2 \in \R^{d_2}$, $\widehat{b}_3 \in \R$, and $\sigma_3(x):=\max\{0,x\} + 0.01\min\{0,x\}$, $x \in \R$, is the Leaky-ReLU activation function applied componentwise. Then, we model the logarithmic mean function $\widehat{\mu}:\R^d \times \R^m \to  (0, \infty)$ of $Y$ by
\begin{equation}\label{eq:sev}
\log \widehat{\mu}(\theta, z) = \log \mathbb E[Y|Z=z,\theta]:= \widehat{\mathfrak{N}}(\theta, {z}),
\end{equation}
or, equivalently, $\widehat{\mu}(\theta, z) := \exp(\widehat{\mathfrak{N}}(\theta, z))$ where $\widehat{\mathfrak{N}}$ is a TLFN defined in \eqref{eq:TLFNgr}. Denote by $\Theta :=(\theta, \phi)\in \R^{d+1}$. The mean function $\widehat{\mu}$ defined in \eqref{eq:sev} can be estimated by minimizing the NLL function $\ell:  (0, \infty) \times \R^{m} \times \R^{d+1} \rightarrow \R$ given by
\begin{align}\label{eq:nll}
\begin{split}
&\ell(y, z, \Theta) := -\log f_Y(y;\widehat{\mu}(\theta, z), \phi)   \\
&=\log  y +\log \Gamma\left(\frac{1}{\exp(\phi)}\right) -\frac{1}{\exp(\phi)}\left(\log\left(\frac{ y}{\exp(\phi)}\right) - \widehat{\mathfrak{N}}(\theta,z) \right)   +  \frac{y}{\exp(\phi)}\exp{(-\widehat{\mathfrak{N}}(\theta, z) )},
\end{split}
\end{align}
see, e.g., \cite{garrido:16}, and the associated regularized optimization problem (for some $r, \eta>0$) is given as follows:
\begin{equation}\label{eqn:optimization_real}
\text{minimize} \quad \R^{d+1} \ni \Theta \mapsto u(\Theta) :=  \E[\ell(Y, Z, \Theta)]+\frac{\eta}{2(r+1)}|\theta|^{2(r+1)}.
\end{equation}

For the numerical experiments, we consider the auto-insurance claim data from ``freMTPL2sev'' in the R package ``CASdatasets'' \cite{casdataset}, which contains $\widetilde{N}=\numprint{24944}$ observations. Its $i$-th observation, $i = 1, \dots, \widetilde{N}$, consists of a target variable, denoted by $y_i\in  (0, \infty)$, indicating the average claim size for one year and an input vector, denoted by $\bz_i \in \mathbb R^m$, containing relevant quantities including, e.g., driver's age, vehicle's age, and region. More precisely, in this case, for each $i$, the input vector ${\bz}_i \in \R^m$ with $m=65$ contains $4$ continuous variables and $7$ categorical variables. We refer to `freqMTPL' in \url{http://cas.uqam.ca/pub/web/CASdatasets-manual.pdf} for the precise description of the attributes in each input feature. We note that the dimension for each input vector ${\bz}_i$ is $65$. This is due to the fact that each categorical variable contains certain number of classes, for example, the variable ``VehPower'' contains $6$ classes, and each element in ${\bz}_i$ represents one class. Hence, each ${\bz}_i\in \R^{65}$ contains $61$ classes and $4$ continuous variables as its elements.

For the training and testing purposes, we split the dataset such that the training set contains $70\%$ of the observations and the test set contains $30\%$ of the observations.


We employ e-TH$\varepsilon$O POULA, SGLD defined in \eqref{eqn:sgld}, ADAM, and AMSGrad to solve the optimization problem \eqref{eq:nll}-\eqref{eqn:optimization_real} using the aforementioned dataset. Set $r=0$, $\eta=0.0005$. We note that the optimization problem \eqref{eq:nll}-\eqref{eqn:optimization_real} is then equivalent to an $\ell_2$ regularized optimization problem. We search the hyperparameters for e-TH$\varepsilon$O POULA: $\lambda=\{0.1, 0.01, 0.001\}$, $\epsilon=\{10^{-2}, 10^{-4}, 10^{-8}\}$, and $\beta=10^{12}$. For SGLD, we use the following hyperparameters: $\lambda=\{0.1, 0.01, 0.001, 0.0001\}$ and $\beta=10^{12}$. For ADAM and AMSGrad, the hyperparameters are chosen among $\lambda=\{0.1, 0.01, 0.001\}$ where $\epsilon=10^{-8}$, $\beta_1=0.9$, $\beta_2=0.999$ are fixed. Moreover, we decay the learning rate by $10$ after $25$ epochs to all the optimization algorithms. In addition, TLFN \eqref{eq:TLFNgr} with 100 neruons on each layer is trained for $50$ epochs with $128$ batch size. Each experiment is run three times to compute the mean and standard deviation of NLL on the test set.

Figure~\ref{fig:nll} shows the learning curves for the NLL on both the training and test set for each optimizer. Table~\ref{tab:nll} displays the mean and standard deviation of the NLL on the test set. As shown in Table~\ref{tab:nll}, the performance of SGLD is significantly inferior to that of ADAM, AMSGrad, and e-TH$\varepsilon$O POULA, and its learning curves are highly unstable. On the contrary, ADAM, AMSGrad, and e-TH$\varepsilon$O POULA produce very stable learning processes. Moreover, e-TH$\varepsilon$O POULA achieves the lowest test NLL, implying the model trained by e-TH$\varepsilon$O POULA generalizes better than the models found by other optimization algorithms.

\begin{figure}[t]
    \centering
    \includegraphics[width=0.47\textwidth]{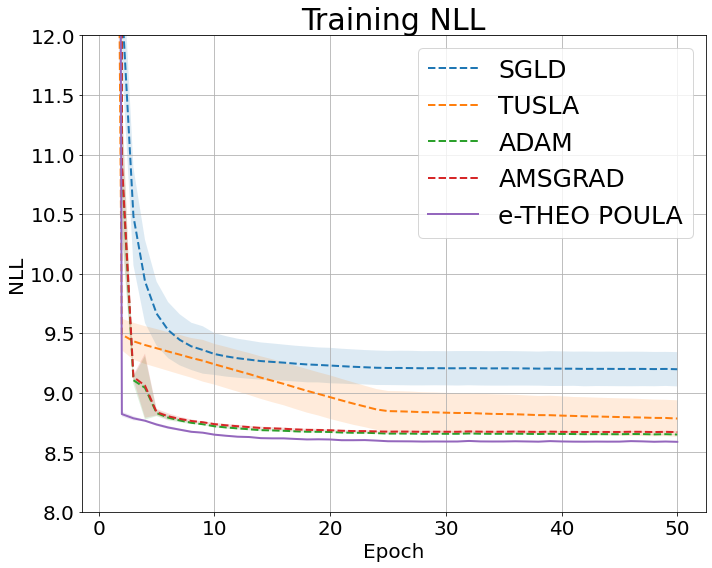}
    \includegraphics[width=0.47\textwidth]{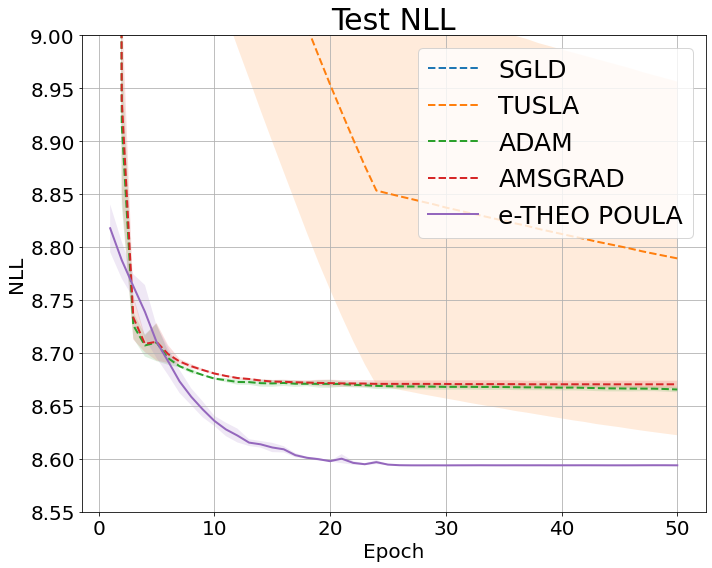}
    \caption{Negative likelihood curve on training and test set. The colored area corresponds to the mean $\pm$ standard deviation for each algorithm.}
    \label{fig:nll}
\end{figure}

\begin{table}[t]
\centering
\begin{tabular}{c|c| c | c | c | c}
 \hline
& SGLD & TUSLA &ADAM  & AMSGrad & {\color{blue} e-TH$\varepsilon$O POULA}   \\  \hline
test NLL &9.93 (0.439) & 8.78 (0.167)& 8.66 (0.002) & 8.66 (0.003) & {\color{blue} 8.59 (0.002)} \\\hline
training time & 129 & 129 & 129 & 128 & {\color{blue} 130} \\\hline
\end{tabular}
\caption{The best NLL evaluated on the test set and training time (measured in seconds) for the non-linear Gamma regression task. We report the mean and standard deviation of the test NLL computed from three experiments with different random seeds where the numbers in parenthesises indicate the standard deviations.}
\label{tab:nll}
\end{table}

\subsection{Conclusion of numerical experiments}\label{sub:simulationsummary} 
For the experiments under the Black-Scholes model as described in Section~\ref{sub:port_op}, ADAM and AMSGrad demonstrate comparable test scores to e-TH$\varepsilon$O POULA and offer advantages in training speed, whereas in the AR($1$) model discussed in Section~\ref{sub:port_op}, 
e-TH$\varepsilon$O POULA outperforms ADAM and AMSGrad in terms of test scores. Furthermore, in terms of training speed in the AR($1$) model,  ADAM and AMSGrad do not
achieve a difference of less than 1\% from the lowest best score obtained by e-TH$\varepsilon$O POULA even after 200 epochs. In addition, in the non-linear gamma regression discussed in Section~\ref{sub:gamma}, e-TH$\varepsilon$O POULA outperforms ADAM and AMSGrad in terms of test scores.

Based on our experiments, we observe that e-TH$\varepsilon$O POULA outperforms ADAM-type optimizers when training neural networks with a larger number of neurons and hidden layers, and when approximating more complex target functions. Therefore, we suggest that ADAM-type optimizers are viable options when a reasonably accurate solution is required quickly. However, when theoretical guarantees on convergence are critical and complex deep learning architectures are involved, we recommend e-TH$\varepsilon$O POULA over ADAM-type optimizers.


\section{Non-asymptotic convergence bounds for e-TH$\varepsilon$O POULA} \label{sec:main}
In this section, we provide non-asymptotic error estimates for e-TH$\varepsilon$O POULA \eqref{eq:theopoula}-\eqref{eq:expressiontGF}, which are established based on the assumptions provided below.
\subsection{Assumptions} \label{sec:assumption} Let the conditions imposed in Section \ref{sec:theopoulasting} be fulfilled, and let $q \in [1, \infty), r \in [q/2, \infty) \cap \N, \rho \in [1, \infty)$ be fixed. 

In the first assumption, we impose moment requirements for the initial value $\theta_0$ of e-TH$\varepsilon$O POULA \eqref{eq:theopoula}-\eqref{eq:expressiontGF} and for the data process $(X_n)_{n\in\N_0}$.
\begin{assumption} \label{asm:AI}
The initial condition $\theta_0$ has a finite $(8r+4)$-th moment, i.e., $\E[|\theta_0|^{8r+4}]<\infty$. The process $(X_n)_{n \in \N_0}$ is a sequence of i.i.d.\ random variables with $\mathcal{L}(X_n)=\mathcal{L}(X)$ for each $n \in \N_0$ and has a finite $(8r+4)\rho$-th moment, i.e., $\E[|X_0|^{(8r+4)\rho}]<\infty$.
\end{assumption}

Recall the definition of $H$ given in \eqref{eq:expH}, which is the sum of $G$ and $F$. In the following assumption, we assume that $G$ satisfies a ``continuity in average'' condition and a growth condition.
\begin{assumption} \label{asm:AG}
There exists a constant $L_G>0$ such that, for all $\theta, \bar{\theta} \in \R^d$,
\[
\E[|G(\theta, X_0)-G(\bar{\theta}, X_0)|] \leq L_G(1+|\theta|+|\bar{\theta}|)^{q-1}|\theta-\bar{\theta}|.
\]
In addition, there exists a constant $K_G >1$, such that for all $\theta \in \R^d, x \in \R^m$,
\[
|G(\theta,x) | \leq K_G(1+|x|)^{\rho}(1+|\theta|)^q.
\]
\end{assumption}

Then, we assume that $F$ is locally Lipschitz continuous. Furthermore, we impose a growth condition on each component of $F$, which enables us to obtain more relaxed step-size restrictions.
\begin{assumption} \label{asm:AF}
There exists a constant $L_F>0$ such that, for all $\theta, \bar{\theta} \in \R^d, x,\bar{x} \in \R^m$,
\[
|F(\theta, x)-F(\bar{\theta}, \bar{x})| \leq L_F(1+|x|+|\bar{x}|)^{\rho-1}(1+|\theta|+|\bar{\theta}|)^{2r}(|\theta-\bar{\theta}|+|x-\bar{x}|).
\]
Furthermore, there exists a constant $K_F>0$ such that for all $\theta \in \R^d, x \in \R^m$, $i = 1, \dots, d$,
\[
|F^{(i)}(\theta,x) | \leq K_F(1+|x|)^{\rho}(1+|\theta^{(i)}|)(1+|\theta|^{2r}).
\]
\end{assumption}
Under Assumptions \ref{asm:AI}-\ref{asm:AF}, one obtains a growth condition for $H$, and a local Lipschitz condition for $h$. We refer to Appendix \ref{sec:remarkproofs} for the proofs of the statements in the following remark.
\begin{remark}\label{rem:growthHlliph} By Assumptions \ref{asm:AI}, \ref{asm:AG}, and \ref{asm:AF}, for any $\theta \in \R^d, x \in \R^m$, one obtains that
\begin{equation*}
|H(\theta,x) | \leq K_H(1+|x|)^{\rho}(1+|\theta|^{2r+1}),
\end{equation*}
where $K_H:=2^{2r}K_G+3\sqrt{2d}K_F$. Moreover, by Assumptions \ref{asm:AI}, \ref{asm:AG}, and \ref{asm:AF}, we have, for any $\theta, \bar{\theta} \in \R^d$,
\begin{equation*}
|h(\theta) - h(\bar{\theta}) | \leq L_h(1+|\theta|+|\bar{\theta}|)^{2r}|\theta-\bar{\theta}|,
\end{equation*}
where $L_h:=\max\{L_G+L_F\E[(1+2|X_0|)^{\rho-1}],1\}$.
\end{remark}

Next, we impose a (local) convexity at infinity condition on $F$.
\begin{assumption} \label{asm:AC}
There exist Borel measurable functions $A:\R^m\to\R^{d\times d}, B:\R^m\to\R^{d\times d}$, and $0\leq \bar{r}<2r$ such that the following holds:
\begin{enumerate}[leftmargin=*]
\item For any $ x \in \R^m, y \in \R^d$, $\langle y, A(x) y\rangle \geq 0, \quad \langle y, B(x) y\rangle \geq 0$.
\item For all $\theta, \theta' \in \R^d$ and $x\in\R^m$,
\begin{align}\label{cailoclip}
\begin{split}
\langle  \theta- \theta' , F(\theta,x)- F(\theta',x)\rangle
&\geq \langle\theta- \theta', A(x) (\theta- \theta')\rangle(|\theta|^{2r} +|\theta'|^{2r}) \\
&\quad -\langle\theta- \theta', B(x) (\theta- \theta')\rangle(|\theta|^{\bar{r}} +|\theta'|^{\bar{r}}).
\end{split}
\end{align}
\item The smallest eigenvalue of $\E[A(X_0)]$ is a positive real number $a$, and the largest eigenvalue of $\E[B(X_0)]$ is a nonnegative real number $b$.
\end{enumerate}
\end{assumption}
Under Assumptions \ref{asm:AI}-\ref{asm:AC}, one can show that $F$ and $h$ satisfy certain dissipativity conditions. Moreover, $h$ further satisfies a one-sided Lipschitz condition. The explicit statements are provided below, and the proofs follow the same ideas as in the proofs of \cite[Remark 2.5, 2.6]{lim2021nonasymptotic}, which can be found in \cite[Appendix 1]{lim2021nonasymptotic}. 
\begin{remark}\label{rem:Fhdissiposl} By Assumptions \ref{asm:AI}, \ref{asm:AG}, \ref{asm:AF}, and \ref{asm:AC}, one obtains, for any $\theta \in \R^d$, that
\begin{equation}\label{eq:Fdissip}
\langle \theta, \E[F(\theta, X_0)]\rangle \geq a_F|\theta|^{2r+2} - b_F,
\end{equation}
where $a_F := a/2$ and $b_F := (a/2+b)R_F^{\bar{r}+2}+dK_F^2\E[(1+|X_0|)^{2\rho}]/(2a)$ with
\[
R_F := \max\{(4b/a)^{1/(2r-\bar{r})}, 2^{1/(2r)}\}.
\]

Furthermore, for any $\theta \in \R^d$, we have
\begin{equation}\label{eq:hdissip}
\langle \theta, h(\theta)\rangle \geq a_h|\theta|^2 - b_h,
\end{equation}
where $a_h := 2^qK_G\E[(1+|X_0|)^{\rho}]$, $b_h := 3(2^{q+1}K_G\E[(1+|X_0|)^{\rho}]/\min{\{1, a_F\}})^{q+2}+b_F$. One notes that due to \cite[Eqn. (25), (26)]{JARNER2000341} and \cite[Theorem 2.32]{beck2014introduction}, \eqref{eq:hdissip} implies that $u$ has a minimum $\theta^* \in \R^d$.

In addition, for any $\theta, \bar{\theta} \in \R^d$, one obtains
\begin{equation}\label{eq:hosl}
\langle \theta-\bar{\theta},h(\theta)-h(\bar{\theta})\rangle\geq -L_R |\theta-\bar{\theta}|^2,
\end{equation}
where $L_R :=  L_h(1+2R)^{2r}>0$ with $R :=\max\{1, (3^{q-1}L_G/a)^{1/(2r-q+1)}, (2b/a)^{1/(2r - \bar{r})} \}$.
\end{remark}

\begin{remark} We provide further justifications of our Assumptions \ref{asm:AI}-\ref{asm:AC}:
\begin{enumerate}
\item In Assumption \ref{asm:AI}, we impose moment requirements for the initial value $\theta_0$ of e-TH$\varepsilon$O POULA and for the data process $(X_n)_{n\in\N_0}$ as they are essential to obtain moment estimates of e-TH$\varepsilon$O POULA. For example, in \eqref{eq:convergencew2ub2} in our convergence analysis, we have the term $\E\left[|\bar{\theta}^{\lambda}_{\lfrf{s}}|^{8r+4}\right]$ which can be upper bounded by using Lemma \ref{lem:2ndpthmmt} as
\[
\E\left[|\bar{\theta}^{\lambda}_{\lfrf{s}}|^{8r+4}\right] \leq \E\left[|\theta_0|^{8r+4}\right]  +\mathring{c}_{4r+2}
\]
for some constant $\mathring{c}_{4r+2}$. Hence, we require $(8r+4)$-th moment of $\theta_0$ to be finite to make sense of $\E\left[|\theta_0|^{8r+4}\right]$. Similarly, by the definition of $\lambda_{p, \max} $ and $\lambda_{\max}$ in \eqref{eq:stepsizemax}, we require $(8r+4)\rho$-th moment of $(X_n)_{n \in \N_0}$ to be finite to make sense of $\E[(1+|X_0|)^{(8r+4)\rho }]$.
\item In Assumption \ref{asm:AG}, we impose a local Lipschitz in average condition and a growth condition on $G$, while in Assumption \ref{asm:AF}, we impose a local (or polynomial) Lipschitz continuous condition on $F$. Assumptions \ref{asm:AG} and \ref{asm:AF} can be viewed as extensions of a global Lipschitz condition, which are part of the assumptions required for the existence and uniqueness of the solution to the SDE defined in \eqref{sdeintro}, see, e.g., \eqref{eq:hdissip} and \eqref{eq:hosl} in Remark \ref{rem:Fhdissiposl}, and \cite[Theorem 1]{krylov1991simple}. Moreover, these conditions also play a crucial role in establishing convergence results for e-TH$\varepsilon$O POULA (see, e.g., \cite{tula,lim2021polygonal,lim2021nonasymptotic,lovas2023taming} for machine learning algorithms and \cite{hutzenthaler2012,eulerscheme,SabanisAoAP} for numerical schemes for SDEs with super-linearly growing coefficients which also impose certain local Lipschitz condition) similar to that of a Lipschitz continuity condition for the SGLD algorithm (see, e.g., \cite{dalalyan,pmlr-v65-dalalyan17a,dk,DM16} and references therein).
\item In Assumption \ref{asm:AC}, we impose a local convexity at infinity condition on $F$. To understand this condition, we first consider the following inequality:
\begin{equation}\label{eq:slconv}
\langle  \theta- \theta' , F(\theta,x)- F(\theta',x)\rangle \geq \langle\theta- \theta', A(x) (\theta- \theta')\rangle(|\theta|^{2r} +|\theta'|^{2r}).
\end{equation}
In the case $r=0$, $F$ is globally Lipschitz continuous in $\theta$ by Assumption \ref{asm:AF} and thus growing linearly, and \eqref{eq:slconv} above becomes a (local) convexity condition, see also \cite[Assumption 3.9]{convex}. Then, in the case $r>0$, \eqref{eq:slconv} can be viewed as a local convexity condition for a super-linearly growing $F$, where \eqref{eq:slconv} (and \eqref{cailoclip}) are referred to as ``local'' conditions as the RHS of these inequalities depends on the data stream $x$. Now, we observe that \eqref{eq:slconv} implies \eqref{cailoclip} in Assumption \ref{asm:AC}, which indicates that our assumption is weaker. We refer to \eqref{cailoclip} as a convexity ``at infinity'' condition as, while it cannot be seen as a (local) convexity condition due to the subtraction of the non-negative definite term involving $B(x)$, the first term on the RHS of \eqref{cailoclip} dominates when $|\theta|$ is sufficiently large, which results in a strong convexity condition (for $|\theta|$ sufficiently large). For illustrative purposes, we consider the case $F(\theta, x) = \theta^3 - \theta$ for all $\theta \in \R, x \in \R$. We see that $F$ does not satisfy \eqref{eq:slconv}, but it satisfies \eqref{cailoclip} in Assumption \ref{asm:AC} with $A(x)= 1/2, B(x)= 1, r =1 , \bar{r} =0$, i.e.,
\[
(\theta- \theta')(F(\theta, x) -F(\theta', x) )\geq(|\theta|^2+|\theta'|^2)|\theta- \theta'|^2/2 - |\theta- \theta'|^2,
\]
which implies that the following strong convexity condition holds for $|\theta|, |\theta'| \geq \sqrt{2}$:
\[
(\theta- \theta')(F(\theta, x) -F(\theta', x) ) \geq |\theta - \theta'|^2.
\]
We note that Assumption \ref{asm:AC} can be satisfied by a wide class of functions including, e.g., the regularization term in the regularized optimization problem, the double-well potential model, and the Ginzburg-Landau model.
	
Moreover, Assumption \ref{asm:AC} is crucial as it can be used to deduce a dissipativity condition of $h$, i.e., \eqref{eq:hdissip} in Remark \ref{rem:Fhdissiposl}, which is then used to deduce a Lyapunov drift condition in Lemma \ref{lem:Lyapunovdriftc}. This is a key assumption to obtain the convergence results in, e.g., Wasserstein distances, in non-convex optimization, see \cite[Assumption 2.2]{eberle2019quantitative}. In addition, we note that the SDE associated with the optimization problem \eqref{eq:obju} is the Langevin SDE \eqref{eq:sde} with a super-linearly growing drift coefficient, and e-TH$\varepsilon$O POULA can be viewed as its numerical approximation. Then, Assumption \ref{asm:AC} allows us to deduce a one-sided Lipschitz condition on $h$, i.e., \eqref{eq:hosl} in Remark \ref{rem:Fhdissiposl}, which is one of the standard assumptions required to ensure a unique solution of the Langevin SDE, see \cite[Theorem 1]{krylov1991simple}, and to establish the convergence results for numerical schemes of SDEs with super-linearly growing coefficients, see \cite[Assumption 3.1]{higham2002strong}.
\end{enumerate}
\end{remark}
\subsection{Main results} \label{sec:mr}
Define, for any $p \in \N$,
\begin{equation}\label{eq:stepsizemax}
\lambda_{p, \max} := \min\left\{1,  \frac{1}{a_F}, \frac{1}{a_F^2},  \frac{\min\{(a_F/K_F)^2, (a_F/K_F)^{2/(2p-1)}\}}{16K_F^2 p^2(2p-1)^2(\E[(1+|X_0|)^{2p\rho }])^2}  \right\}, \qquad \lambda_{\max} : = \lambda_{4r+2, \max},
\end{equation}
where $a_F: =a/2$.

Our first result provides a non-asymptotic error bound in Wasserstein-1 distance between the law of e-TH$\varepsilon$O POULA \eqref{eq:theopoula}-\eqref{eq:expressiontGF} and $\pi_{\beta}$.
\begin{theorem}\label{thm:convergencew1}
Let Assumptions \ref{asm:AI}, \ref{asm:AG}, \ref{asm:AF}, and \ref{asm:AC} hold.  Then, for any $\beta>0$ there exist constants $C_0, C_1,C_2>0$ such that, for any $0<\lambda\leq \lambda_{\max}$ with $\lambda_{\max}$ given in \eqref{eq:stepsizemax}, and $n \in \N_0$,
\[
W_1(\mathcal{L}(\theta^{\lambda}_n),\pi_{\beta}) \leq C_1 e^{-C_0 \lambda n}(\E[|\theta_0|^{4(2r+1)}]+1) +C_2\sqrt{\lambda},
\]
where $C_0, C_1,C_2$ are given explicitly in \eqref{eq:convergencew1consts}.
\end{theorem}
Then, we provide a non-asymptotic convergence result in Wasserstein-2 distance  between the law of e-TH$\varepsilon$O POULA \eqref{eq:theopoula}-\eqref{eq:expressiontGF} and $\pi_{\beta}$.
\begin{corollary}\label{thm:convergencew2}
Let Assumptions \ref{asm:AI}, \ref{asm:AG}, \ref{asm:AF}, and \ref{asm:AC} hold. Then, for any $\beta>0$ there exist constants $C_3, C_4,C_5>0$ such that, for any $0<\lambda\leq \lambda_{\max}$ with $\lambda_{\max}$ given in \eqref{eq:stepsizemax}, and $n \in \N_0$,
\[
W_2(\mathcal{L}(\theta^{\lambda}_n),\pi_{\beta})\leq C_4 e^{-C_3 \lambda n}(\E[|\theta_0|^{4(2r+1)}]+1)^{1/2} +C_5\lambda^{1/4},
\]
where $C_3, C_4,C_5$ are given explicitly in \eqref{eq:convergencew2consts}.
\end{corollary}
By applying Corollary \ref{thm:convergencew2}, one can obtain a non-asymptotic upper bound for the expected excess risk following the splitting approach adopted in \cite{raginsky}.
\begin{theorem}\label{thm:opeer}
Let Assumptions \ref{asm:AI}, \ref{asm:AG}, \ref{asm:AF}, and \ref{asm:AC} hold. Then, for any $\beta>0$ there exist constants $C_6, C_7,C_8, C_9>0$ such that, for any $0<\lambda\leq \lambda_{\max}$ with $\lambda_{\max}$ given in \eqref{eq:stepsizemax}, and $n \in \N_0$,
\[
\E[u( \theta_n^{\lambda})] - u^* \leq C_7 e^{-C_6\lambda n}+C_8\lambda^{1/4}+C_9/\beta,
\]
where $u^*:= \inf_{\theta \in \R^d} u(\theta)$, $C_6, C_7,C_8$ are given explicitly in \eqref{eq:opeerp1consts} while $C_9$ is given in \eqref{eq:opeerp2consts}.
\end{theorem}

\begin{corollary}\label{corollary:eerepsilon}
Let Assumptions \ref{asm:AI}, \ref{asm:AG}, \ref{asm:AF}, and \ref{asm:AC} hold, and let $C_6, C_7,C_8$, and $C_9$ be defined in Theorem \ref{thm:opeer}. For any $\delta>0$, if we first choose
\[
\beta \geq \max\left\{1,   \frac{9d^2}{\delta^2}, \left(\frac{3d}{\delta}  \log\left(\frac{L_h(1+4(\sqrt{b_h/a_h}+\sqrt{2d/ L_h}))^{2r}  e}{a_h d}\left(b_h+1\right)\left(d+1\right)\right) +\frac{\log 64}{\delta}\right)\right\},
\]
then choose $\lambda \leq  \min\{\lambda_{\max}, \delta^4/(81C_8^4)\}$, and finally choose $n \geq  \max\{(1/C_6\lambda_{\max})\log (3C_7/\delta), \allowbreak (81C_8^4/C_6 \delta^4)\log (3C_7/\delta)\}$, then, we have
\[
\E[u( \theta_n^{\lambda})] -  \inf_{\theta \in \R^d} u(\theta) \leq \delta.
\]
\end{corollary}
\section{Comparison to existing literature and our contributions}

\subsection{Related work and discussions} Langevin dynamics based algorithms are widely used methods for solving sampling and optimization problems. Convergence results provide theoretical justifications for the effectiveness and efficiency of the algorithms. Under the conditions that $u$ defined in \eqref{eq:obju} is (strongly) convex and the gradient of $u$ is Lipschitz continuous, convergence results in total variation and in Wasserstein-2 distance are established in \cite{dalalyan,pmlr-v65-dalalyan17a,dk,durmus2017nonasymptotic,DM16} for the unadjusted Langevin algorithm (which is also referred to as the Langevin Monte Carlo algorithm and which can be viewed as an Euler discretization of \eqref{sdeintro}). As the exact gradient of $u$ is often unavailable in practice, \cite{wt} proposed the SGLD algorithm given in \eqref{eqn:sgld}, which is a natural extension of the unadjusted Langevin algorithm with the exact gradient of $u$ replaced by its unbiased estimator. Under the same set of the aforementioned conditions together with i.i.d.\ data stream, \cite{ppbdm,dk} provide non-asymptotic convergence estimates for the SGLD algorithm in Wasserstein-2 distance. These results are then extended in \cite{convex} to the case of dependent data stream under a relaxed local convexity condition.

However, the gradient Lipschitzness and the strong convexity condition of $u$ are usually not satisfied by many practical applications, which motivates the relaxation of the assumptions in two directions. The first direction considers the generalization of the strong convexity condition of $u$. In \cite{raginsky}, the authors proposed a dissipativity condition which is weaker than the strong convexity condition, and obtained a convergence result in Wasserstein distance for the SGLD algorithm. This result is improved in \cite{nonconvex,xu} where the latter result is applicable also to the case of dependent data stream. A further extension of the dissipativity condition to the so-called local dissipativity condition is considered in \cite{sgldloc} which accommodates examples from variational inference and index tracking optimization. The other direction considers the generalization of the gradient Lipschitzness of $u$. One line of research in this direction focuses on the convergence analysis under a local Lipschitz (in average) condition of the (stochastic) gradient of $u$, which corresponds to the case where the (stochastic) gradient of $u$ is super-linearly growing. In \cite{tula}, the authors proposed the tamed unadjusted Langevin algorithm, and obtained convergence results in total variation and in Wasserstein-2 distance which can be applied to sampling problems involving the double-well model and the Ginzburg-Landau model. In \cite{lim2021polygonal,lim2021nonasymptotic,lovas2023taming}, the authors developed variants of the SGLD algorithm, i.e., TUSLA and TH$\varepsilon$O POULA , by applying the taming technique introduced in \cite{hutzenthaler2012,eulerscheme}, and obtained convergence results in Wasserstein distances in the context of non-convex optimization. More precisely, \cite{lovas2023taming} proposed the TUSLA algorithm, which can be viewed as a tamed SGLD algorithm. The authors suggest a specific form for $H$ with $H(\theta, x) := G(\theta,x)+\eta \theta|\theta|^{2r}$, $\eta>0, r >3/2$, which is a natural representation of the stochastic gradient of a given optimization problem with a high-order regularization term. Under the condition that $G$ is locally Lipschitz continuous together with certain moment requirements on the initial condition of the algorithm $\theta_0$ and $X$ (see \eqref{eq:obju}), non-asymptotic error bounds in Wasserstein distances are established for TUSLA, which is used to further deduce a non-asymptotic estimate for the expected excess risk. Then, \cite{lim2021nonasymptotic} extended the results obtained in \cite{lovas2023taming} which provides guarantees for TUSLA to solve a larger variety of applications. To improve the empirical performance of the SGLD type of algorithms including SGLD and TUSLA, \cite{lim2021polygonal} proposed a new algorithm TH$\varepsilon$O POULA which is developed based on the taming technique and the Euler’s polygonal approximations for SDEs:
\begin{enumerate}
\item In \cite{lim2021nonasymptotic}, the authors considered a general form of $H$ in the sense that $H = G+F$. It is assumed in \cite[Assumption 2]{lim2021nonasymptotic} that $G$ satisfies a continuity \textit{in average} condition, which is a relaxation of the local Lipschitz condition imposed in \cite{lovas2023taming} allowing $G$ to be discontinuous. However, unlike the local Lipschitz condition, the continuity in average condition does not necessarily imply a growth condition of $G$, hence, in \cite[Assumption 2]{lim2021nonasymptotic}, the authors imposed separately a growth condition of $G$. Moreover, in \cite[Assumption 3]{lim2021nonasymptotic}, it is assumed that $F$ satisfies a local Lipschitz condition, which covers the case that $F(\theta, x) = \eta \theta|\theta|^{2r}$ considered in \cite{lovas2023taming}. Then, in \cite[Assumption 4]{lim2021nonasymptotic}, the authors imposed a local convexity at infinity condition of $F$. This condition for TUSLA (applied to optimization problems with super-linearly growing stochastic gradient) plays the same role as the strong convexity condition for SGLD (applied to optimization problems with linearly growing stochastic gradient) in the sense that it is the key condition to establish the convergence results in Wasserstein distances. Moreover, \cite[Assumption 4]{lim2021nonasymptotic} can be satisfied by a wide range of functions including, e.g., the regularization function, the double well model, and the Ginzburg-Landau model. Under \cite[Assumption 2-4]{lim2021nonasymptotic} together with moment requirements on $\theta_0$ and $(X_n)_{n \in \N_0}$, i.e., \cite[Assumption 1]{lim2021nonasymptotic}, the authors established convergence results in Wasserstein distances and obtained non-asymptotic convergence bound for the expected excess risk. These results can be applied to optimization problems with discontinuous stochastic gradient which cannot be covered by the results in \cite{lovas2023taming}. We highlight that the assumptions considered in \cite{lim2021nonasymptotic} are the most relaxed conditions under which non-asymptotic convergence results in Wasserstein distances can be obtained for non-convex stochastic optimization problems, allowing for super-linearly growing and discontinuous stochastic gradient. However, despite its wide applicability and theoretical guarantees, TUSLA is outperformed by ADAM-type optimizers empirically in terms of test accuracy in many examples of the fine tuning of artificial neural networks.
\item In \cite{lim2021polygonal}, the authors proposed the TH$\varepsilon$O POULA algorithm which has superior empirical performance compared to TUSLA in terms of test accuracy, and which performs at least as good as ADAM-type optimizers. This is due to the design of TH$\varepsilon$O POULA, which combines the component-wise taming technique with a suitable boosting function to address the vanishing gradient problem. Numerical experiments for several examples are presented to confirm the superior performance of TH$\varepsilon$O POULA. Besides, by using the same structure of $H$ and under the same set of assumptions as in \cite{lovas2023taming}, the authors provided full theoretical guarantees for the convergence of TH$\varepsilon$O POULA. However, as the convergence results are obtained under a local Lipschitz condition of the stochastic gradient of $u$, it cannot accommodate applications with discontinuous stochastic gradient including, e.g., optimization problems involving ReLU neural networks.
\end{enumerate}
Another line of research in this direction focuses on the convergence analysis in the case where the (stochastic) gradient of $u$ is discontinuous and is linearly growing. In this setting, \cite{4,fort2016} provided convergence results for the SGD algorithm in the almost sure sense and in $L^1$, respectively, while \cite{sglddiscont} provided a non-asymptotic convergence estimate in Wasserstein distances for the SGLD algorithm. Moreover, in \cite{durmus2018efficient}, the authors proposed the Moreau-Yosida Unadjusted Langevin Algorithm (MYULA) and obtained a non-asymptotic convergence bound in total variation distance, while in \cite{durmus2019analysis}, the authors developed the Stochastic Proximal Gradient Langevin Dynamics (SPGLD) algorithm and obtained a convergence result between the Kullback-Leibler divergence from the target distribution $\pi (\rmd \theta) \wasypropto \exp(- u(\theta))\rmd \theta$ to the averaged distribution associated with the SPGLD algorithm. In addition, \cite{ruszczynski2020convergence} considered constrained optimization problems with a non-smooth and non-convex objective function, and obtained an almost sure convergence result for the proposed algorithm developed using a stochastic sub-gradient method with sub-gradient averaging. However, this result does not specify key constants on the convergence upper bound including, e.g., the rate of convergence.

In this paper, we introduce a new algorithm e-TH$\varepsilon$O POULA which combines the advantages of utilizing Euler’s polygonal approximations resulting in its superior empirical performance, together with a relaxed condition on its stochastic gradient, namely a local continuity in average condition, allowing for discontinuous stochastic gradient, resulting in its wide applicability. Furthermore, we derive non-asymptotic convergence bounds for e-TH$\varepsilon$O POULA with explicit constants.

Let us provide a detailed comparison of e-TH$\varepsilon$O POULA with the most related works in the literature \cite{lim2021polygonal,lim2021nonasymptotic,ruszczynski2020convergence}:

\begin{enumerate}
\item We first compare our results with those in \cite{lim2021nonasymptotic}. Under the same set of assumptions on the corresponding stochastic gradient as on the one in \cite{lim2021nonasymptotic}, we establish non-asymptotic convergence bounds in Wasserstein distance for e-TH$\varepsilon$O POULA (Theorem \ref{thm:convergencew1}, Corollary \ref{thm:convergencew2}), and then provide an optimization convergence bound for the expected excess risk (Theorem \ref{thm:opeer}). The convergence results in our paper are comparable to those in \cite{lim2021nonasymptotic} in the sense that the rates of convergence in Wasserstein distances for e-TH$\varepsilon$O POULA and for TUSLA are the same and the constants on the convergence upper bound are of the same magnitude. 
However, we highlight that the structure of e-TH$\varepsilon$O POULA significantly differs from TUSLA in \cite{lim2021nonasymptotic} leading to the superior empirical performance of e-TH$\varepsilon$O POULA compared to TUSLA. We illustrate this point by presenting examples from multi-period portfolio optimization and from non-linear Gamma regression. Table \ref{tab:bs}, \ref{tab:ar}, and \ref{tab:nll} clearly show that e-TH$\varepsilon$O POULA outperforms TUSLA in almost all experiments in terms of test accuracy and training speed. In particular, we consider a transfer learning setting in the multi-period portfolio optimization with numerical results presented in Table \ref{tab:tl}-\ref{tab:tl_time} and provide a proof to show that the example satisfies Assumptions \ref{asm:AI}-\ref{asm:AC}. Hence, we provide a concrete example illustrating the powerful empirical performance of e-TH$\varepsilon$O POULA backed by our convergence results.
\item Now let us compare with \cite{lim2021polygonal}. While e-TH$\varepsilon$O POULA keeps the advantages of utilizing Euler’s polygonal approximations as in TH$\varepsilon$O POULA \cite{lim2021polygonal}, it allows for a stochastic gradient with a more general structure of  the form $H:=G+F$, where $G$ might be discontinuous, compared to the local Lipschitz continuity requirement on $G$ in TH$\varepsilon$O POULA \cite{lim2021polygonal}, and where the continuous part $F$ can, but is not restricted to, be of the form $F= \eta \theta|\theta|^{2r}$. We establish convergence results for e-TH$\varepsilon$O POULA which can be applied to a large class of applications including, e.g., optimization problems with ReLU neural networks that cannot be covered by the results in \cite{lim2021polygonal}.
\item Finally, we compare our work with \cite{ruszczynski2020convergence}. \cite{ruszczynski2020convergence} considers constrained optimization problems with a non-convex and non-smooth objective function satisfying a so-called generalized differentiability property, while our work focuses on an unconstrained optimization problem \eqref{eq:obju} whose objective function is continuously differentiable with a discontinuous stochastic gradient. \cite[(A1)]{ruszczynski2020convergence} assumes that all iterates of the stochastic approximation algorithm $\{x^k\}$ belong to a compact set, while the iterates of e-TH$\varepsilon$O POULA $(\theta^{\lambda}_n)_{n \in \N_0}$ explore the whole $\R^d$. \cite[(A2)]{ruszczynski2020convergence} assumes a decreasing sequence of stepsizes $\{\tau_k\}$ satisfying $\tau_k \in (0, \min(1,1/a)]$ with $a>0$ being a constant parameter, for all $k$, and $\sum_{k}\tau_k = \infty$, while we assume a constant stepsize $\lambda \in (0, \lambda_{\max}]$ where for $p \in \N$,
	\begin{align*}
	\lambda_{p, \max} := \min\left\{1,  \frac{1}{a_F}, \frac{1}{a_F^2},  \frac{\min\{(a_F/K_F)^2, (a_F/K_F)^{2/(2p-1)}\}}{16K_F^2 p^2(2p-1)^2(\E[(1+|X_0|)^{2p\rho }])^2}  \right\}, \quad \lambda_{\max}: = \lambda_{4r+2, \max}.
	\end{align*}
	\cite[(A3)]{ruszczynski2020convergence} imposes conditions on the error term $\{r^k\}$ of the stochastic subgradient $\{g^k\}$, while we assume that the stochastic gradient $H$ is an unbiased estimator of the exact gradient $h$ of the objective function. \cite[(A4)]{ruszczynski2020convergence} imposes a condition on the objective function $f$ such that the set of Clarke stationary point of $f$ does not contain an interval of nonzero length, while we assume that the objective function $u$ is continuously differentiable satisfying certain convexity at infinity condition which ensures the existence of a minimizer. Besides, we further impose: 1) moment requirements for $\theta_0$ and $(X_n)_{n \in N_0}$; 2) $H = F+G$ with $F$ being locally Lipschitz continuous and $G$ satisfying a local Lipschitz in average condition and a growth condition. Under \cite[(A1)-(A4)]{ruszczynski2020convergence}, \cite[Theorem 4.2]{ruszczynski2020convergence} shows that the algorithm under consideration is convergent in the almost sure sense to a broader class of functions satisfying the property of generalized differentiability, which is achieved by proving a chain rule on a path for the aforementioned functions. However, such convergence result is provided without specifying the key constants including, e.g., the rate of convergence of the algorithm, while our results (Theorem \ref{thm:convergencew1} and Corollary \ref{thm:convergencew2}) provide non-asymptotic convergence bounds in Wasserstein distances with explicit constants, e.g., the rate of convergence of e-TH$\varepsilon$O POULA in Wasserstein-1 and Wasserstein-2 distances are $1/2$ and $1/4$, respectively. These results are then used to deduce a non-asymptotic error bound for the expected excess risk associated to \eqref{eq:obju} where explicit constants on the upper bound are provided. 
We note that our Assumptions \ref{asm:AI}-\ref{asm:AC} are minimal assumptions required to obtain non-asymptotic convergence results with explicit constants, which provide crucial information on the choice of key parameters including $\beta,\lambda,n$ (see Corollary \ref{corollary:eerepsilon}) for numerical experiments. Regarding numerical results, \cite{ruszczynski2020convergence} considers one concrete optimization problem with its objective function given by $\E[|W_2\max\{0,W_1X\}-Y|^2]/2$ where $W_1\in \R^{n\times n}$ and $W_2 \in \R^{1\times n}$ are parameters with $n=11$, and where $X$ and $Y$ are $\R^n$-valued input random variable and $\R$-valued target random variable, respectively. Experiments are conducted on a relatively small-scale dataset with $\numprint{6497}$ samples consisting of $\numprint{71467}$ data points and they show that the averaged stochastic subgradient method considered in \cite{ruszczynski2020convergence} outperforms stochastic subgradient method in terms of training accuracy as the former achieves lower loss. In our case, we present several examples relevant in practice in Remark \ref{rmk:sturctureHtoyexamples} and in Section \ref{sub:transfer_learning} which satisfy our Assumptions \ref{asm:AI}-\ref{asm:AC}. In particular, for the example considered in Section \ref{sub:transfer_learning}, we further show that e-TH$\varepsilon$O POULA outperforms other popular machine learning algorithms including, e.g., ADAM and AMSGrad, in terms of test accuracy indicating its great generalization ability. The superior empirical performance of e-TH$\varepsilon$O POULA are illustrated by applying it also to other real-world applications in as presented in Section \ref{sec:numapp}. We note that our experiments are conducted on large-scale datasets, for example, for the Black-Scholes model considered in Section \ref{sub:port_op},  we use the dataset of $\numprint{4000000}$ samples which consists of $\numprint{4000000}\times Kp$\footnote{We refer to Table
\ref{tab:iid} for the values of $K$ and $p$.} data points to train the model.
\end{enumerate}

We highlight that the proof of convergence for e-TH$\varepsilon$O POULA follows the line of all convergence results for the Langevin dynamics based algorithms. More precisely, for the convergence results in Wasserstein distances, the usual first step is to obtain moment estimates for all the processes involved in the convergence analysis (including the proposed algorithm) so as to make sense of the convergence in $L^p$, $p\geq 1$. Then, in the second step, we proceed to obtain upper estimates for Wasserstein distances between the proposed algorithm and the target distribution associated with $u$, which is achieved by adopting a splitting using appropriate auxiliary processes (see, e.g., the splitting in \eqref{eq:convergencew1split}) aiming to obtain optimal convergence results. We note that these two steps are taken in all the aforementioned papers considering the convergence in Wasserstein distances. In particular, in the context of non-convex optimization, to the best of the authors' knowledge, the state-of-the-art convergence bounds for the SGLD algorithm in Wasserstein distances are achieved using the framework established in \cite{nonconvex}. For our newly proposed algorithm e-TH$\varepsilon$O POULA, we follow the steps developed in \cite{nonconvex} but carefully adapt the proofs to obtain optimal convergence results. Therefore, from a technical point of view, even though our conditions on the stochastic gradient are the same as in \cite{lim2021nonasymptotic}, the structure of e-TH$\varepsilon$O POULA is significantly different from TUSLA and hence the proof of convergence cannot be derived from previous results (neither from \cite{lim2021nonasymptotic} nor from \cite{lim2021polygonal}).

\subsection{Our contributions} We summarize the main contributions of our paper as follows:
\begin{enumerate}
\item We propose a new algorithm called e-TH$\varepsilon$O POULA which combines the advantages of utilizing Euler’s polygonal approximations resulting in its superior empirical performance, together with a relaxed condition on its stochastic gradient, namely a local continuity in average condition, allowing for discontinuous stochastic gradient, resulting in its wide applicability.
\item We propose a transfer learning setting involving neural networks in the multi-period portfolio optimization (in Section \ref{sub:transfer_learning}) by applying the dynamic programming principle and the time-homogeneity property of the associated Markov decision process. We provide a proof to show that this example satisfies Assumptions \ref{asm:AI}-\ref{asm:AC} under which theoretical guarantees for the performance of e-TH$\varepsilon$O POULA are obtained. To the best of the authors’ knowledge, this is the first time theoretical convergence results with explicit constants are obtained for an optimization algorithm in the context of deep learning based multi-period portfolio optimization. Furthermore, we provide additional examples from multi-period portfolio optimization and from non-linear Gamma regression to illustrate its powerful empirical performance in terms of test accuracy and training speed.
\item We provide a theoretical framework that accommodates optimization problems with discontinuous stochastic gradient, where the example from multi-period portfolio optimization with transfer learning is a special case. Under Assumptions \ref{asm:AI}-\ref{asm:AC}, we provide non-asymptotic convergence bounds with explicit constants for the newly proposed e-TH$\varepsilon$O POULA algorithm. In particular, we show that under Assumptions \ref{asm:AI}-\ref{asm:AC}, e-TH$\varepsilon$O POULA minimizes the expected excess risk associated to \eqref{eq:obju}.
\end{enumerate}

\section{Proofs of main theoretical results}\label{sec:mtproofs}

\subsection{Moment estimates}
Consider the SDE $(Z_t)_{t \geq 0}$ given by
\begin{equation} \label{eq:sde}
\mathrm{d} Z_t=- h (Z_t)\mathrm{d} t+ \sqrt{2\beta^{-1}} \mathrm{d} B_t,
\end{equation}
with the initial condition $Z_0:= \theta_0$, where $(B_t)_{t \geq 0}$ is standard $d$-dimensional Brownian motion with its completed natural filtration denoted by $(\mathcal{F}_t)_{t\geq 0}$. We assume that $(\mathcal{F}_t)_{t\geq 0}$ is independent of $\mathcal{G}_{\infty} \vee \sigma(\theta_0)$. Under Assumptions \ref{asm:AI}, \ref{asm:AG}, \ref{asm:AF}, and \ref{asm:AC}, one notes that \eqref{eq:sde} has a unique solution adapted to $\mathcal{F}_t \vee \sigma(\theta_0)$, $t \geq 0$, due to Remark \ref{rem:Fhdissiposl}, see, e.g., \cite[Theorem 1]{krylov1991simple}. For any $p \in \N$, the $2p$-th moment of SDE \eqref{eq:sde} is finite, i.e., $\sup_{t \geq 0}\E[|Z_t|^{2p}]<\infty$, and its explicit upper bound can be obtained by using similar arguments as in the proof of \cite[Lemma A.1]{lim2021nonasymptotic}. Moreover, following the proof of  \cite[Proposition 1-(ii)]{DM16}, one can show that $\pi_{\beta}$ has a finite $2p$-th moment, for any $p \in \N$.

We introduce a time-changed version of SDE \eqref{eq:sde}, which is denoted by $(Z_t^{\lambda})_{t \geq 0}$ with $Z_t^{\lambda} = Z_{\lambda t}, t \geq 0$. For each $\lambda >0$, define $B^{\lambda}_t :=B_{\lambda t}/\sqrt{\lambda}$, for any $t \geq 0$, and denote by $(\mathcal{F}^{\lambda}_t)_{t \geq 0} := (\mathcal{F}_{\lambda t})_{t \geq 0}$ the completed natural filtration of $(B^{\lambda}_t)_{t \geq 0}$. One notes that $(\mathcal{F}^{\lambda}_t)_{t \geq 0}$ is independent of $\mathcal{G}_{\infty} \vee \sigma(\theta_0)$. Then, $(Z_t^{\lambda})_{t \geq 0}$ is defined explicitly by
\begin{equation} \label{eq:tcsde}
\mathrm{d} Z_t^{\lambda}=-  \lambda h (Z_t^{\lambda})\mathrm{d} t+ \sqrt{2\lambda\beta^{-1}} \mathrm{d} B_t^{\lambda},
\end{equation}
with the initial condition $Z_0^{\lambda}:= \theta_0$. We will use $(Z_t^{\lambda})_{t \geq 0}$ in the proofs of main theorems.

Next, we consider the continuous-time interpolation of e-TH$\varepsilon$O POULA \eqref{eq:theopoula}-\eqref{eq:expressiontGF}, denoted by $(\bar{\theta}^{\lambda}_t)_{t \geq 0}$, which is given by
\begin{equation}\label{eq:theopoulaproc}
\rmd \bar{\theta}^{\lambda}_t=-\lambda H_{\lambda}(\bar{\theta}^{\lambda}_{\lfrf{t}},X_{\lcrc{t}})\, \rmd t
+ \sqrt{2\lambda\beta^{-1}} \rmd B^{\lambda}_t
\end{equation}
with the initial condition $\bar{\theta}^{\lambda}_0:=\theta_0$. One notes that the law of the process \eqref{eq:theopoulaproc} coincides with the law of e-TH$\varepsilon$O POULA \eqref{eq:theopoula}-\eqref{eq:expressiontGF} at grid-points, i.e., $\mathcal{L}(\bar{\theta}^{\lambda}_n)=\mathcal{L}(\theta_n^{\lambda})$, for each $n\in\N_0$.

For any $p \in \N$, we establish $2p$-th moment estimates of e-TH$\varepsilon$O POULA \eqref{eq:theopoula}-\eqref{eq:expressiontGF} in the lemma below. In particular, we show that, in the special case where $F$ depends only on $\theta \in \R^d$,  finite $2p$-th moments of e-TH$\varepsilon$O POULA \eqref{eq:theopoula}-\eqref{eq:expressiontGF} can be obtained under a more relaxed step-size restriction $\hat{\lambda}_{\max}$ given in \eqref{eq:stepsizemaxrelaxed} (instead of $\lambda_{p, \max}$ given in \eqref{eq:stepsizemax}).
\begin{lemma}\label{lem:2ndpthmmt} Let Assumptions \ref{asm:AI}, \ref{asm:AG}, \ref{asm:AF}, and \ref{asm:AC} hold. Then, one obtains the following:
\begin{enumerate}[leftmargin=*]
\item \label{lem:2ndpthmmti}
For any $0<\lambda\leq \lambda_{1, \max}$ with $\lambda_{1, \max}$ given in \eqref{eq:stepsizemax}, $n \in \N_0$, and $t \in (n, n+1]$,
\[
\E\left[ |\bar{\theta}^{\lambda}_t|^2  \right]  \leq  (1-\lambda(t-n)a_F \kappa)(1-a_F \kappa \lambda)^n \E\left[|\theta_0|^2\right]  +\mathring{c}_0,
\]
where $\mathring{c}_0:=c_0(1+1/( a_F \kappa))$, $a_F:= a/2$, and the constants $ \kappa, c_0$ are given in \eqref{eq:2ndmmtconstc0}. In particular, the above inequality implies $\sup_{t\geq 0}\E\left[|\bar{\theta}^{\lambda}_t|^2\right]  \leq  \E\left[|\theta_0|^2\right] +\mathring{c}_0<\infty$.
\item \label{lem:2ndpthmmtii}
For any $p \in [2, \infty) \cap \N$, $0<\lambda\leq \lambda_{p, \max}$ with $\lambda_{p, \max}$ given in \eqref{eq:stepsizemax}, $n \in \N_0$, and $t \in (n, n+1]$,
\begin{equation}\label{eq:2pthmmtiiresult}
\E\left[ |\bar{\theta}^{\lambda}_t|^{2p}  \right]  \leq  (1-\lambda(t-n)  a_F \kappa^{\sharp}_2 )(1- \lambda  a_F \kappa^{\sharp}_2)^n \E\left[|\theta_0|^{2p}\right]  +\mathring{c}_p,
\end{equation}
where $\mathring{c}_p:=c_0^{\sharp}(p)(1+1/(  a_F \kappa^{\sharp}_2))$, $a_F:= a/2$, $\kappa^{\sharp}_2 := \min\{\bar{\kappa}(2), \tilde{\kappa}(2)\}$, $c_0^{\sharp}(p):= \max\{\bar{c}_0(p), \tilde{c}_0(p)\}$ with $\bar{\kappa}(2)$, $\bar{c}_0(p)$ and $\tilde{\kappa}(2)$, $\tilde{c}_0(p)$ given in \eqref{eq:2pthmmtconstbarc0} and \eqref{eq:2pthmmtconsttildec0}, respectively. In particular, the above result implies $\sup_{t\geq 0}\E\left[|\bar{\theta}^{\lambda}_t|^{2p}\right]  \leq  \E\left[|\theta_0|^{2p}\right] +\mathring{c}_p<\infty$.
\item \label{lem:2ndpthmmtiii} If there exists $\hat{F}:\R^d \rightarrow \R^d$, such that $F(\theta, x) = \hat{F}(\theta)$, for any $\theta \in \R^d$, $x \in \R^m$, then \eqref{eq:2pthmmtiiresult} holds for any $p \in [2, \infty) \cap \N$, $0<\lambda\leq \hat{\lambda}_{\max}$ with $\hat{\lambda}_{\max}$ given by
\begin{equation}\label{eq:stepsizemaxrelaxed}
\hat{\lambda}_{\max} := \min\left\{1,  \frac{1}{a_F}, \frac{1}{a_F^2},  \frac{a_F^2}{16K_F^4 (\E[(1+|X_0|)^{2\rho}])^2} \right\}.
\end{equation}
\end{enumerate}
\end{lemma}
\begin{proof} See Appendix \ref{proof2ndpthmmt}.
\end{proof}

\begin{remark} One observes that, for every $p \in [2, \infty) \cap \N$, $\lambda_{p, \max} \leq \hat{\lambda}_{\max} $ with $\lambda_{p, \max}$, $\hat{\lambda}_{\max} $ given in \eqref{eq:stepsizemax}, \eqref{eq:stepsizemaxrelaxed}, respectively. Hence, $\hat{\lambda}_{\max}$ is indeed a relaxation of the stepsize compared to $\lambda_{p, \max}$, $p \in [2, \infty) \cap \N$. More importantly, in the case where $F$ depends only on $\theta$, Theorem \ref{thm:convergencew1}, \ref{thm:convergencew2}, and \ref{thm:opeer} hold for $0<\lambda\leq \hat{\lambda}_{\max}$, which can be verified by using the same arguments as provided in Section \ref{sec:mtproof}.
\end{remark}
For each $\bar{p} \in [2, \infty)\cap {\N}$, we denote by $V_{\bar{p}}$ the Lyapunov function given by $V_{\bar{p}}(\theta): =  (1+|\theta|^2)^{\bar{p}/2}$, for all $\theta \in \R^d$. Moreover, define $\mathrm{v}_{\bar{p}}(\nu) = (1+\nu^2)^{\bar{p}/2}$ for all $\nu \geq 0$. One notes that $V_{\bar{p}}$ is twice continuously differentiable, and possess the following properties:
\begin{equation}\label{eq:Lyapucontrasm}
\sup_{\theta \in \R^d} (|\nabla V_{\bar{p}}(\theta)|/V_{\bar{p}}(\theta))  <\infty, \qquad \lim_{|\theta|\to\infty} (\nabla V_{\bar{p}}(\theta)/V_{\bar{p}}(\theta))=0.
\end{equation}
Furthermore, we denote by $\mathcal{P}_{V_{\bar{p}}}(\R^d)$ the set of probability measures $\mu \in \mathcal{P}(\R^d)$ that satisfy the condition $\int_{\R^d} V_{\bar{p}}(\theta)\, \mu(\rmd \theta) <\infty$.

In the following lemma, we show that the Lyapunov function $V_{\bar{p}}$, $\bar{p} \in [2, \infty)\cap {\N}$, satisfies a geometric drift condition.
\begin{lemma}\label{lem:Lyapunovdriftc} Let Assumptions \ref{asm:AI}, \ref{asm:AG}, \ref{asm:AF}, and \ref{asm:AC} hold. Then, for any $\theta \in \R^d$, $\bar{p} \in [2, \infty)\cap {\N}$, one obtains that
\[
 - \langle \nabla V_{\bar{p}}(\theta), h(\theta) \rangle +\Delta V_{\bar{p}}(\theta)/\beta \leq -c_{V,1}({\bar{p}}) V_{\bar{p}}(\theta) +c_{V,2}({\bar{p}}),
\]
where $c_{V,1}(\bar{p}) := a_h\bar{p}/4$, $c_{V,2}(\bar{p}) := (3/4)a_h\bar{p}\mathrm{v}_{\bar{p}}(M_V(\bar{p}) )$ with $M_V(\bar{p}) := (1/3+4b_h/(3a_h)+4d/(3a_h\beta)+4(\bar{p}-2)/(3a_h\beta))^{1/2}$.
\end{lemma}
\begin{proof} See \cite[Lemma 3.5]{nonconvex}.
\end{proof}

For every $s \geq 0$, we introduce an auxiliary process, denoted by $(\zeta^{s,v, \lambda}_t)_{t \geq s}$, which is crucial in establishing the convergence results. More precisely, the process $(\zeta^{s,v, \lambda}_t)_{t \geq s}$ is given by
\begin{equation}\label{eq:auxzetaproc}
\rmd\zeta^{s,v, \lambda}_t = -\lambda h(\zeta^{s,v, \lambda}_t)\, \rmd t +\sqrt{2\lambda\beta^{-1}}\,\rmd B^{\lambda}_t,
\end{equation}
with the initial condition $\zeta^{s,v, \lambda}_s := v \in \R^d$. Denote by $T\equiv T(\lambda) : = \lfrf{1/\lambda}$. For compact notation, for each fixed $\lambda >0$, $n \in \N_0$, define $\bar{\zeta}^{\lambda, n}_t := \zeta^{nT,\bar{\theta}^{\lambda}_{nT}, \lambda}_t$, $t \geq nT$. The process $\bar{\zeta}^{\lambda, n}_t$, $t \geq nT$ can be interpreted as a continuous-time process starting from the value of  e-TH$\varepsilon$O POULA \eqref{eq:theopoula}-\eqref{eq:expressiontGF} at time $nT$, i.e., $\bar{\theta}^{\lambda}_{nT}$, which evolves according to the Langevin SDE \eqref{eq:auxzetaproc} up to time $t \geq nT$.

We provide the second and the fourth moment estimate of the process $(\bar{\zeta}^{\lambda, n}_t)_{t\geq nT}$ in the following lemma.
\begin{lemma}\label{lem:zetaprocme} Let Assumptions \ref{asm:AI}, \ref{asm:AG}, \ref{asm:AF}, and \ref{asm:AC} hold. Then, one obtains the following:
\begin{enumerate}[leftmargin=*]
\item For any $0<\lambda\leq \lambda_{1,\max}$ with $\lambda_{1,\max}$ given in \eqref{eq:stepsizemax}, $n \in \N_0$, and $t\geq nT$, we have
\[
\E[V_2(\bar{\zeta}^{\lambda, n}_t)] \leq e^{-\min\{a_h /2,  a_F\kappa\}\lambda t}\E[V_2(\theta_0)]+\mathring{c}_0 +1+3\mathrm{v}_2(M_V(2)),
\]
where $\mathring{c}_0:=c_0(1+1/(  a_F \kappa))$, the constants $c_0, \kappa$ are given in \eqref{eq:2ndmmtconstc0}, and $M_V(2):= (1/3+4b_h/(3a_h)+4d/(3a_h\beta) )^{1/2}$.
\item For any $0<\lambda\leq \lambda_{2,\max}$ with $\lambda_{2,\max}$ given in \eqref{eq:stepsizemax}, $n \in \N_0$, and $t\geq nT$, we have
\[
\E[V_4(\bar{\zeta}^{\lambda, n}_t)] \leq 2e^{- \min\{a_h,  a_F \kappa^{\sharp}_2\}\lambda t}\E[V_4(\theta_0)]+2\mathring{c}_2+2+3\mathrm{v}_4(M_V(4)),
\]
where $\mathring{c}_2:=c_0^{\sharp}(2)(1+1/(  a_F \kappa^{\sharp}_2))$, $\kappa^{\sharp}_2 := \min\{\bar{\kappa}(2), \tilde{\kappa}(2)\}$, $c_0^{\sharp}(2):= \max\{\bar{c}_0(2), \tilde{c}_0(2)\}$ with $\bar{\kappa}(2)$, $\bar{c}_0(2)$ and $\tilde{\kappa}(2)$, $\tilde{c}_0(2)$ given in \eqref{eq:2pthmmtconstbarc0} and \eqref{eq:2pthmmtconsttildec0}, respectively, and where $M_V(4):=(1/3+4b_h/(3a_h)+4d/(3a_h\beta)+8/(3a_h\beta))^{1/2}$.
\end{enumerate}
\end{lemma}
\begin{proof} We follow the proof of \cite[Lemma 4.4]{lim2021nonasymptotic} where \cite[Lemma 4.2, Lemma 4.3]{lim2021nonasymptotic} are replaced by Lemma \ref{lem:2ndpthmmt}, \ref{lem:Lyapunovdriftc}, respectively, to obtain the explicit constants.
\end{proof}
\subsection{Proofs of main theorems}\label{sec:mtproof} In this section, we provide a proof overview of the main theoretical results in the setting of super-linearly growing $H$ in both variables. We first introduce a semimetric $w_{1,\hat{p}}$, which is defined as follows: for any $\hat{p}\geq 1$, $ \mu,\nu \in \mathcal{P}_{V_{\hat{p}}}(\R^d)$, let
\begin{equation}\label{eq:semimetricw1p}
w_{1,\hat{p}}(\mu,\nu):=\inf_{\zeta\in\mathcal{C}(\mu,\nu)}\int_{\mathbb{R}^d}\int_{\mathbb{R}^d} [1\wedge |\theta-\theta'|](1+V_{\hat{p}}(\theta)+V_{\hat{p}}(\theta'))\zeta(\rmd\theta, \rmd\theta').
\end{equation}
The analysis of the convergence results, i.e., Theorem \ref{thm:convergencew1} and \ref{thm:convergencew2}, relies on the contractivity of the Langevin SDE \eqref{eq:sde} in $w_{1,2}$, which can be deduced by using \cite[Theorem 2.2]{eberle2019quantitative}. The explicit statement of the contraction property in $w_{1,2}$, as well as the explicit contraction constants, is presented in the following lemma.

\begin{proposition}\label{prop:contractionw12} Let Assumptions \ref{asm:AI}, \ref{asm:AG}, \ref{asm:AF}, and \ref{asm:AC} hold. Moreover, let $\theta_0' \in L^2$, and let $(Z_t')_{t \geq 0}$ be the solution of SDE \eqref{eq:sde} with $Z'_0 := \theta'_0$, which is independent of $\mathcal{F}_{\infty} := \sigma(\bigcup_{t \geq 0} \mathcal{F}_t)$. Then, one obtains
\begin{equation}\label{eq:w12contraction}
w_{1,2}(\mathcal{L}(Z_t),\mathcal{L}(Z'_t)) \leq \hat{c} e^{-\dot{c} t} w_{1,2}(\mathcal{L}(\theta_0),\mathcal{L}(\theta_0')),
\end{equation}
where the explicit expressions for $\dot{c}, \hat{c}$ are given below.

The contraction constant $\dot{c}$ is given by:
\begin{equation}\label{eq:contractionc1}
\dot{c}:=\min\{\bar{\phi}, c_{V,1}(2), 4c_{V,2}(2) \epsilon c_{V,1}(2)\}/2,
\end{equation}
where $c_{V,1}(2) := a_h/2$, $c_{V,2}(2) := 3a_h\mathrm{v}_2(M_V(2))/2$ with $M_V(2):= (1/3+4b_h/(3a_h)+4d/(3a_h\beta) )^{1/2}$, the constant $\bar{\phi} $ is given by
\begin{equation}\label{eq:contractionc2}
\bar{\phi} := \left(\sqrt{8\pi/(\beta L_R)} \dot{c}_0  \exp \left( \left(\dot{c}_0 \sqrt{\beta L_R/8} + \sqrt{8/(\beta L_R)} \right)^2 \right) \right)^{-1},
\end{equation}
and $\epsilon >0$ is chosen such that
\begin{equation}\label{eq:contractionc3}
\epsilon  \leq 1 \wedge    \left(4 c_{V,2}(2) \sqrt{2 \beta\pi/  L_R }\int_0^{\dot{c}_1}\exp  \left( \left(s \sqrt{\beta L_R/8}+\sqrt{8/(\beta L_R)}\right)^2 \right) \,\rmd s \right)^{-1}
\end{equation}
with $\dot{c}_0 := 2(4c_{V,2}(2)(1+c_{V,1}(2))/c_{V,1}(2)-1)^{1/2}$ and $\dot{c}_1:=2(2 c_{V,2}(2)/c_{V,1}(2)-1)^{1/2}$.

Moreover, the constant $\hat{c}$ is given by:
\begin{equation}\label{eq:contractionhatcub}
\hat{c}: =2(1+ \dot{c}_0)\exp(\beta L_R \dot{c}_0^2/8+2\dot{c}_0)/\epsilon.
\end{equation}
\end{proposition}
\begin{proof} One can check that \cite[Theorem 2.2, Corollary 2.3]{eberle2019quantitative} hold under Assumptions \ref{asm:AI}, \ref{asm:AG}, \ref{asm:AF}, and \ref{asm:AC}. Indeed, due to Remark \ref{rem:Fhdissiposl}, \cite[Assumption 2.1]{eberle2019quantitative} holds with $\kappa = L_R$. Then, by Lemma \ref{lem:Lyapunovdriftc}, \cite[Assumption 2.2]{eberle2019quantitative} holds with $V=V_2$. Finally, \cite[Assumptions 2.4 and 2.5]{eberle2019quantitative} hold due to \eqref{eq:Lyapucontrasm}. Thus, by using \cite[Theorem 2.2, Corollary 2.3]{eberle2019quantitative} and by applying the arguments in the proof of \cite[Proposition~3.14]{nonconvex}, one obtains \eqref{eq:w12contraction}. To obtain the explicit constants $\dot{c}, \hat{c}$ in \eqref{eq:contractionc1}-\eqref{eq:contractionhatcub}, one may refer to the proof of \cite[Proposition 4.6]{lim2021nonasymptotic}.
\end{proof}
We proceed with the proof of Theorem \ref{thm:convergencew1}. To establish a non-asymptotic upper bound in Wasserstein-1 distance between the law of e-TH$\varepsilon$O POULA \eqref{eq:theopoula}-\eqref{eq:expressiontGF} and $\pi_{\beta}$ defined in \eqref{eq:pibetaexp}, we apply the following splitting method: for any $n \in \N_0$, and $t \in (nT, (n+1)T]$,
\begin{equation}\label{eq:convergencew1split}
W_1(\mathcal{L}(\bar{\theta}^{\lambda}_t),\pi_{\beta})\leq W_1(\mathcal{L}(\bar{\theta}^{\lambda}_t),\mathcal{L}(\bar{\zeta}^{\lambda, n}_t))+W_1(\mathcal{L}(\bar{\zeta}^{\lambda, n}_t),\mathcal{L}(Z_t^{\lambda}))+W_1(\mathcal{L}(Z_t^{\lambda}),\pi_{\beta}),
\end{equation}
where $\bar{\zeta}^{\lambda, n}_t := \zeta^{nT,\bar{\theta}^{\lambda}_{nT}, \lambda}_t$ with $\zeta^{nT,\bar{\theta}^{\lambda}_{nT}, \lambda}_t$ defined in \eqref{eq:auxzetaproc}, and $Z_t^{\lambda}$ is defined in \eqref{eq:tcsde}. We provide an upper bound for the first term on the RHS of \eqref{eq:convergencew1split} in the following lemma.

\begin{lemma}\label{lem:w1converp1} Let Assumptions \ref{asm:AI}, \ref{asm:AG}, \ref{asm:AF}, and \ref{asm:AC} hold.  Then, for any $0<\lambda\leq \lambda_{\max}$ with $\lambda_{\max}$ given in \eqref{eq:stepsizemax}, $n \in \N_0$, and $t \in (nT, (n+1)T]$, one obtains
\[
W_2(\mathcal{L}(\bar{\theta}^{\lambda}_t),\mathcal{L}(\bar{\zeta}^{\lambda, n}_t)) \leq \sqrt{\lambda}  \left(e^{- n  a_F \kappa^\sharp_2/2 } \bar{C}_0 \E\left[V_{4(2r+1)}(\theta_0)\right] +\bar{C}_1 \right)^{1/2},
\]
where the explicit expressions of $\kappa^\sharp_2, \bar{C}_0, \bar{C}_1$ are provided in \eqref{eq:w1converp1consts}.
\end{lemma}
\begin{proof} See Appendix \ref{proofw1converp1}.
\end{proof}
By the fact that $W_1 \leq w_{1,2}$ (see \cite[Lemma A.3]{lim2021nonasymptotic} for a detailed proof), and by applying Proposition~\ref{prop:contractionw12}, an upper estimate for the second term on the RHS of \eqref{eq:convergencew1split}  can be established.
\begin{lemma}\label{lem:w1converp2} Let Assumptions \ref{asm:AI}, \ref{asm:AG}, \ref{asm:AF}, and \ref{asm:AC} hold.  Then, for any $0<\lambda\leq \lambda_{\max}$ with $\lambda_{\max}$ given in \eqref{eq:stepsizemax}, $n \in \N_0$, and $t \in (nT, (n+1)T]$, one obtains
\[
W_1(\mathcal{L}(\bar{\zeta}_t^{\lambda,n}),\mathcal{L}(Z_t^\lambda)) \leq \sqrt{\lambda}\left(e^{-\min\{\dot{c},  a_F \kappa^\sharp_2, a_h\}n/4}\bar{C}_2\E\left[V_{4(2r+1)}(\theta_0)\right]+\bar{C}_3\right),
\]
where
\begin{align}\label{eq:w1converp2consts}
\begin{split}
\bar{C}_2
& :=\hat{c}e^{\min\{\dot{c},  a_F \kappa^\sharp_2, a_h\}/4}\left(1+ \frac{4}{\min\{\dot{c},   a_F \kappa^\sharp_2, a_h\}}\right)(\bar{C}_0+12),\\
\bar{C}_3
& := 2(\hat{c}/\dot{c})e^{ \dot{c}/2}(\bar{C}_1+15+12\mathring{c}_2+9\mathrm{v}_4(M_V(4)))
\end{split}
\end{align}
with $\dot{c}, \hat{c}$ given in Proposition \ref{prop:contractionw12}, $\bar{C}_0, \bar{C}_1$ given in \eqref{eq:w1converp1consts}, and $ a_F$, $\kappa^\sharp_2$, $\mathring{c}_{2}$ given in Lemma \ref{lem:2ndpthmmt}.
\end{lemma}
\begin{proof} We follow exactly the proof of \cite[Lemma 4.7]{lim2021nonasymptotic}. More precisely, to obtain the explicit constants, we apply Proposition \ref{prop:contractionw12}, Lemma \ref{lem:w1converp1}, \ref{lem:2ndpthmmt}, \ref{lem:zetaprocme} instead of \cite[Proposition 4.6, Lemma 4.5, 4.2, 4.4]{lim2021nonasymptotic}.
\end{proof}
One notes that $\pi_{\beta}$ defined in \eqref{eq:pibetaexp} is the invariant measure of \eqref{eq:tcsde}. Then, by using Proposition \ref{prop:contractionw12}, and by the fact that $W_1(\mu,\nu)\leq w_{1,2}(\mu,\nu)$, one can obtain an upper estimate for the third term on the RHS of \eqref{eq:convergencew1split}, i.e., for any  $0<\lambda\leq \lambda_{\max}$ with $\lambda_{\max}$ given in \eqref{eq:stepsizemax}, $n \in \N_0$, and $t \in (nT, (n+1)T]$,
\begin{equation}\label{eq:w1converp3}
W_1(\mathcal{L}(Z_t^{\lambda}),\pi_{\beta}) \leq w_{1,2}(\mathcal{L}(Z_t^{\lambda}),\pi_{\beta}) \leq  \hat{c} e^{-\dot{c}\lambda t} w_{1,2}(\mathcal{L}(\theta_0),\pi_{\beta}).
\end{equation}
\begin{proof}[\textbf{Proof of Theorem \ref{thm:convergencew1}}] Recall the definition of $w_{1,2}$ in \eqref{eq:semimetricw1p}. By applying Lemma \ref{lem:w1converp1}, \ref{lem:w1converp2} and \eqref{eq:w1converp3} to \eqref{eq:convergencew1split}, one obtains, for $t \in (nT, (n+1)T]$,
\begin{align}\label{eq:convergencew1ue1}
W_1(\mathcal{L}(\bar{\theta}^{\lambda}_t),\pi_{\beta})
&\leq \sqrt{\lambda}  \left(e^{- n  a_F \kappa^\sharp_2/2 } \bar{C}_0 \E\left[V_{4(2r+1)}(\theta_0)\right] +\bar{C}_1 \right)^{1/2} \nonumber\\
&\quad +\sqrt{\lambda}\left(e^{-\min\{\dot{c},  a_F \kappa^\sharp_2, a_h\}n/4}\bar{C}_2\E\left[V_{4(2r+1)}(\theta_0)\right]+\bar{C}_3\right)\nonumber\\
&\quad + \hat{c} e^{-\dot{c}\lambda t}\left(1+ \mathbb{E}[V_2(\theta_0)]+\int_{\mathbb{R}^d}V_2(\theta)\pi_{\beta}(d\theta)\right)\nonumber\\
&\leq C_1 e^{-C_0 (n+1)}(\E[|\theta_0|^{4(2r+1)}]+1) +C_2\sqrt{\lambda},
\end{align}
where
\begin{align}\label{eq:convergencew1consts}
\begin{split}
C_0&:=\min\{\dot{c},  a_F \kappa^\sharp_2, a_h\}/4,\\
C_1&:=2^{4r+1}e^{\min\{\dot{c},  a_F \kappa^\sharp_2, a_h\}/4}\left(\bar{C}_0^{1/2}+\bar{C}_2+ \hat{c} \left(2+\int_{\mathbb{R}^d}V_2(\theta)\pi_{\beta}(d\theta)\right)\right),\\
C_2&:=\bar{C}_1^{1/2}+\bar{C}_3
\end{split}
\end{align}
with $\dot{c}, \hat{c}$ given in Proposition \ref{prop:contractionw12}, $ a_F$, $\kappa^\sharp_2$ given in Lemma \ref{lem:2ndpthmmt}, $\bar{C}_0, \bar{C}_1$ given in \eqref{eq:w1converp1consts}, $\bar{C}_2, \bar{C}_3$ given in \eqref{eq:w1converp2consts}. One notes that \eqref{eq:convergencew1ue1} implies
\[
W_1(\mathcal{L}(\bar{\theta}^{\lambda}_{nT}),\pi_{\beta}) \leq C_1 e^{-C_0 n}(\E[|\theta_0|^{4(2r+1)}]+1) +C_2\sqrt{\lambda},
\]
which yields the desired result by replacing $nT$ with $n$ on the LHS and by replacing $n$ with $n/T \geq \lambda n$ on the RHS.
\end{proof}
\begin{proof}[\textbf{Proof of Corollary \ref{thm:convergencew2}}] For any $n \in \N_0$, and $t \geq nT$, recall that $\bar{\zeta}^{\lambda, n}_t := \zeta^{nT,\bar{\theta}^{\lambda}_{nT}, \lambda}_t$ with $\zeta^{nT,\bar{\theta}^{\lambda}_{nT}, \lambda}_t$ defined in \eqref{eq:auxzetaproc}, and that $Z_t^{\lambda}$ is defined in \eqref{eq:tcsde}. To obtain a non-asymptotic estimate between the law of e-TH$\varepsilon$O POULA \eqref{eq:theopoula}-\eqref{eq:expressiontGF} and $\pi_{\beta}$ (given in \eqref{eq:pibetaexp}) in Wasserstein-2 distance, we consider the following splitting approach: for any $n \in \N_0$, and $t \in (nT, (n+1)T]$,
\begin{equation}\label{eq:convergencew2split}
W_2(\mathcal{L}(\bar{\theta}^{\lambda}_t),\pi_{\beta})\leq W_2(\mathcal{L}(\bar{\theta}^{\lambda}_t),\mathcal{L}(\bar{\zeta}^{\lambda, n}_t))+W_2(\mathcal{L}(\bar{\zeta}^{\lambda, n}_t),\mathcal{L}(Z_t^{\lambda}))+W_2(\mathcal{L}(Z_t^{\lambda}),\pi_{\beta}).
\end{equation}
An explicit upper estimate for the first term on the RHS of \eqref{eq:convergencew2split} is provided in Lemma \ref{lem:w1converp1}. To obtain an upper bound for the second term on the RHS of \eqref{eq:convergencew2split}, one follows the same lines as in the proof of \cite[Lemma 4.7]{lim2021nonasymptotic} while applying $W_2 \leq \sqrt{2w_{1,2}}$ (see \cite[Lemma A.3]{lim2021nonasymptotic} for a detailed proof) instead of $W_1\leq w_{1,2}$, and applying Lemma \ref{lem:2ndpthmmt}, \ref{lem:zetaprocme} for the moment estimates of $\bar{\zeta}^{\lambda, n}_t$ and $\bar{\theta}^{\lambda}_t$ instead of \cite[Lemma 4.2, 4.4]{lim2021nonasymptotic}. Then, one obtains,
\begin{align}\label{eq:w2converp2}
W_2(\mathcal{L}(\bar{\zeta}_t^{\lambda,n}),\mathcal{L}(Z_t^\lambda)) \leq \lambda^{1/4}\left(e^{-\min\{\dot{c},  a_F \kappa^\sharp_2, a_h\}n/8}\bar{C}_4\left(\E\left[V_{4(2r+1)}(\theta_0)\right]\right)^{1/2}+\bar{C}_5\right),
\end{align}
where
\begin{align}\label{eq:w2converp2consts}
\begin{split}
\bar{C}_4
& :=\sqrt{\hat{c}}e^{\min\{\dot{c},  a_F \kappa^\sharp_2, a_h\}/8}\left(1+ \frac{8}{\min\{\dot{c},   a_F \kappa^\sharp_2, a_h\}}\right)(\bar{C}_0^{1/2}+2\sqrt{2}),\\
\bar{C}_5
& := 4(\sqrt{\hat{c}}/\dot{c})e^{ \dot{c}/4}(\bar{C}_1^{1/2}+1+2\sqrt{2}+2\sqrt{2\mathring{c}_2}+\sqrt{3\mathrm{v}_4(M_V(4))})
\end{split}
\end{align}
with $\dot{c}, \hat{c}$ given in Proposition \ref{prop:contractionw12}, $\bar{C}_0, \bar{C}_1$ given in \eqref{eq:w1converp1consts}, $ a_F$, $\kappa^\sharp_2$, $\mathring{c}_{2}$ given in Lemma \ref{lem:2ndpthmmt}, and $M_V(4)$ given in Lemma \ref{lem:Lyapunovdriftc}. For the last term on the RHS of \eqref{eq:convergencew2split}, an upper bound can be obtained by using $W_2 \leq \sqrt{2w_{1,2}}$ and Proposition \ref{prop:contractionw12} as follows:
\begin{align}\label{eq:w2converp3}
W_2(\mathcal{L}(Z_t^{\lambda}),\pi_{\beta}) \leq \sqrt{2}w_{1,2}^{1/2}(\mathcal{L}(Z_t^{\lambda}),\pi_{\beta}) \leq \sqrt{2\hat{c}} e^{-\dot{c}\lambda t/2} w_{1,2}^{1/2}(\mathcal{L}(\theta_0),\pi_{\beta}).
\end{align}
Applying Lemma \ref{lem:w1converp1}, \eqref{eq:w2converp2}, and \eqref{eq:w2converp3} to \eqref{eq:convergencew2split}, one obtains, for any $n \in \N_0$, and $t \in (nT, (n+1)T]$,
\begin{align}\label{eq:convergencew2ue1}
W_2(\mathcal{L}(\bar{\theta}^{\lambda}_t),\pi_{\beta})
&\leq \sqrt{\lambda}  \left(e^{- n  a_F \kappa^\sharp_2/2 } \bar{C}_0 \E\left[V_{4(2r+1)}(\theta_0)\right] +\bar{C}_1 \right)^{1/2} \nonumber\\
&\quad +\lambda^{1/4}\left(e^{-\min\{\dot{c},  a_F \kappa^\sharp_2, a_h\}n/8}\bar{C}_4\left(\E\left[V_{4(2r+1)}(\theta_0)\right]\right)^{1/2}+\bar{C}_5\right) \nonumber\\
&\quad + \sqrt{2\hat{c}} e^{-\dot{c}\lambda t/2}\left(1+ \mathbb{E}[V_2(\theta_0)]+\int_{\mathbb{R}^d}V_2(\theta)\pi_{\beta}(d\theta)\right)^{1/2}\nonumber\\
&\leq   C_4 e^{-C_3 (n+1)}(\E[|\theta_0|^{4(2r+1)}]+1)^{1/2} +C_5\lambda^{1/4},
\end{align}
where
\begin{align}\label{eq:convergencew2consts}
\begin{split}
C_3&:=\min\{\dot{c},  a_F \kappa^\sharp_2, a_h\}/8,\\
C_4&:=2^{2r+1/2}e^{\min\{\dot{c},  a_F \kappa^\sharp_2, a_h\}/8}\left(\bar{C}_0^{1/2}+\bar{C}_4+\sqrt{2 \hat{c}} \left(2+\int_{\mathbb{R}^d}V_2(\theta)\pi_{\beta}(d\theta)\right)^{1/2}\right),\\
C_5&:=\bar{C}_1^{1/2}+\bar{C}_5
\end{split}
\end{align}
with $\dot{c}, \hat{c}$ given in Proposition \ref{prop:contractionw12}, $ a_F$, $\kappa^\sharp_2$ given in Lemma \ref{lem:2ndpthmmt}, $\bar{C}_0, \bar{C}_1$ given in \eqref{eq:w1converp1consts}, and $\bar{C}_4, \bar{C}_5$ given in \eqref{eq:w2converp2consts}.
\end{proof}
By using the non-asymptotic estimate provided in Corollary \ref{thm:convergencew2}, one can obtain an upper estimate for the expected excess risk, i.e., $\E[u( \theta_n^{\lambda})] - u^*$, where $u^*:= \inf_{\theta \in \R^d} u(\theta)$ with $u$ given in \eqref{eq:obju}. We proceed with the following splitting:
\begin{equation}\label{eq:opeersplitting}
\E[u( \theta_n^{\lambda})] - u^* = \E[u( \theta_n^{\lambda})] - \E[u(Z_{\infty})]+\E[u(Z_{\infty})]- u^*,
\end{equation}
where $Z_{\infty}$ is an $\R^d$-valued random variable with $\mathcal{L}(Z_{\infty}) = \pi_{\beta}$.

In the following lemma, we provide an estimate for the first term on the RHS of \eqref{eq:opeersplitting}.
\begin{lemma}\label{lem:opeerp1} Let Assumptions \ref{asm:AI}, \ref{asm:AG}, \ref{asm:AF}, and \ref{asm:AC} hold.  Then, for any $0<\lambda\leq \lambda_{\max}$ with $\lambda_{\max}$ given in \eqref{eq:stepsizemax}, $n \in \N_0$, one obtains
\[
 \E[u( \theta_n^{\lambda})] - \E[u(Z_{\infty})] \leq C_7 e^{-C_6\lambda n}+C_8\lambda^{1/4},
\]
where
\begin{align}\label{eq:opeerp1consts}
\begin{split}
C_6&:=C_3 ,\\
C_7&:= 2^{2r}K_H\E[(1+|X_0|)^{\rho}]\left(C_4(2+\mathring{c}_{2r+1}^{1/2}+c_{Z_{\infty}, 4r+2}^{1/2})+C_5\right)  (\E[|\theta_0|^{4(2r+1)}]+1),\\
C_8&:=2^{2r}K_H\E[(1+|X_0|)^{\rho}]C_5(1+\mathring{c}_{2r+1}^{1/2}+c_{Z_{\infty}, 4r+2}^{1/2})
\end{split}
\end{align}
with $C_3, C_4, C_5$ given in \eqref{eq:convergencew2consts}, $\mathring{c}_{2r+1}$ given in Lemma \ref{lem:2ndpthmmt}, $c_{Z_{\infty}, 4r+2}$ denoting the $4r+2$-th moment of $\pi_{\beta}$.
\end{lemma}
\begin{proof} The proof follows the same arguments as in the proof of \cite[Lemma 4.8]{lim2021nonasymptotic}. However, to obtain explicit constants, we apply Remark \ref{rem:growthHlliph} for the growth condition of $h$ rather than \cite[Remark 2.2]{lim2021nonasymptotic}, Lemma \ref{lem:2ndpthmmt} for the moment estimate of $\theta_n^{\lambda}$ rather than \cite[Lemma 4.2]{lim2021nonasymptotic}, and apply Corollary \ref{thm:convergencew2} for the upper estimate of $W_2(\mathcal{L}(\bar{\theta}^{\lambda}_n),\pi_{\beta})$ rather than \cite[Corollary 2.9]{lim2021nonasymptotic}.
\end{proof}
\begin{lemma}\label{lem:opeerp2} Let Assumptions \ref{asm:AI}, \ref{asm:AG}, \ref{asm:AF}, and \ref{asm:AC} hold.  Then, for any $0<\lambda\leq \lambda_{\max}$ with $\lambda_{\max}$ given in \eqref{eq:stepsizemax}, $n \in \N_0$, and any $\beta>0$, one obtains
\[
\E[u(Z_{\infty})]- u^* \leq C_9/\beta,
\]
where
\begin{equation}\label{eq:opeerp2consts}
C_9\equiv  C_9(\beta): = \frac{d}{2}\log\left(\frac{L_h(1+4\max\{ \sqrt{b_h/a_h},\sqrt{2d/(\beta L_h)}\})^{2r}  e}{a_h}\left(\frac{\beta b_h}{d}+1\right)\right) +\log 2.
\end{equation}
In particular, we have that $\lim_{\beta \to \infty} C_9(\beta)/\beta =0$.
\end{lemma}
\begin{proof} We follow the ideas in the proof of \cite[Lemma 4.9]{lim2021nonasymptotic} while applying Remark \ref{rem:Fhdissiposl} for the dissipativity condition on $h$ rather than \cite[Remark 2.5]{lim2021nonasymptotic}, and applying Remark \ref{rem:growthHlliph} for the local Lipschitz condition of $h$ rather than \cite[Remark 2.2]{lim2021nonasymptotic}.
\end{proof}
\begin{proof}[\textbf{Proof of Theorem \ref{thm:opeer}}] Substituting the results in Lemma \ref{lem:opeerp1}, \ref{lem:opeerp2} into \eqref{eq:opeersplitting} yields the desired non-asymptotic error bound of the expected excess risk.
\end{proof}

\begin{proof}[\textbf{Proof of Corollary \ref{corollary:eerepsilon}}]For any $\delta>0$, if we first choose $\beta$ such that $C_9/\beta \leq \delta/3$, then choose $\lambda$ such that $ \lambda\leq \lambda_{\max}$ with $\lambda_{\max}$ given in \eqref{eq:stepsizemax} and $C_8\lambda^{1/4}\leq \delta/3$, and finally choose $C_7 e^{-C_6\lambda n} \leq \delta/3$, consequently, we have $\E[u( \theta_n^{\lambda})] - u^* \leq \delta$.

We note that $C_9/\beta \leq \delta/3$ is achieved if we choose
\[
\beta \geq \max\left\{1,  \frac{9d^2}{\delta^2}, \left(\frac{3d}{\delta}  \log\left(\frac{L_h(1+4(\sqrt{b_h/a_h}+\sqrt{2d/ L_h}))^{2r}  e}{a_h d}\left(b_h+1\right)\left(d+1\right)\right) +\frac{\log 64}{\delta}\right)\right\}.
\]
Indeed, for any $\beta \geq 1$, we have that
\begin{align*}
\frac{C_9}{\beta}
&\leq \frac{d}{2\beta}  \log\left(\frac{L_h(1+4(\sqrt{b_h/a_h}+\sqrt{2d/ L_h}))^{2r}  e}{a_h d}\left(b_h+1\right)\left(d+1\right)\left(\beta+1\right)\right)+\frac{\log 2}{\beta}\\
&\leq \frac{1}{\beta}\left( \frac{d}{2 }  \log\left(\frac{L_h(1+4(\sqrt{b_h/a_h}+\sqrt{2d/ L_h}))^{2r}  e}{a_h d}\left(b_h+1\right)\left(d+1\right)\right)+\log 2\right) + \frac{d}{2\sqrt{ \beta}}\\
&\leq \frac{\delta}{6}+\frac{\delta}{6} = \frac{\delta}{3}.
\end{align*}
where we use that $\log(1+\beta)/\beta \leq 1/\sqrt{1+\beta}\leq 1/\sqrt{\beta}$ holds for all $\beta>0$ in the second inequality. Furthermore, we have that $\lambda \leq \min\{ \lambda_{\max}, \delta^4/(81C_8^4)\}$, and $\lambda n \geq (1/C_6)\log (3C_7/\delta)$ which further implies that $n \geq \max\{ (1/C_6\lambda_{\max})\log (3C_7/\delta), (81C_8^4/C_6 \delta^4)\log (3C_7/\delta)\}$.
\end{proof}

\newpage
\appendix
\section{Proof of auxiliary results}
\subsection{Proof of auxiliary results in Section \ref{sec:assumption}}\label{sec:remarkproofs}
\begin{proof}[\textbf{Proof of statement in Remark \ref{rem:growthHlliph}}] By using Assumption \ref{asm:AF}, for any $\theta \in \R^d, x \in \R^m$, one obtains
\begin{align*}
|F(\theta, x)|= \left(\sum_{i = 1}^d |F^{(i)}(\theta,x) |^2\right)^{1/2}
& \leq \left(\sum_{i = 1}^d K_F^2(1+|x|)^{2\rho}(1+|\theta^{(i)}|)^2(1+|\theta|^{2r})^2\right)^{1/2} \\
&\leq 3\sqrt{2d} K_F (1+|x|)^{ \rho} (1 +|\theta|^{2r+1}),
\end{align*}
where 
we use $(a+b)^2\leq 2a^2+2b^2$ and $a^v \leq 1 +a^{2r+1}$, for $v \in\{ 1, 2r\}$, $a,b \geq 0$, $r \geq 1$. Then, by using Assumption \ref{asm:AG}, $2r \geq q \geq 1$, and $(1+a)^{2r+1} \leq 2^{2r}(1+a^{2r+1})$, $a \geq 0$, one obtains
\[
|G(\theta,x) | \leq K_G(1+|x|)^{\rho}(1+|\theta|)^{2r+1} \leq 2^{2r}K_G(1+|x|)^{\rho}(1+|\theta|^{2r+1}).
\]
Recall the expressin of $H$ given in \eqref{eq:expH}. Combining the results above yields the first inequality in Remark \ref{rem:growthHlliph}. Furthermore, one notes that the second inequality follows from the local Lipschitz continuity (in average) imposed on $F, G$, see Assumptions \ref{asm:AI}, \ref{asm:AG}, and \ref{asm:AF}.
\end{proof}
\subsection{Proof of auxiliary results in Section \ref{sec:numapp}}
\begin{proof}[\textbf{Proof of Proposition \ref{prop:optim_tl_reg_vasm}}]\label{sec:optim_tl_reg_proof}
For illustrative purposes, we consider the case where $K = 1$. Moreover, to ease the notation, we use $W_1^{\mathfrak{N}}$ instead of $W_1^{\mathfrak{N}(\widetilde \theta, W_0)}$ throughout the proof. By using \eqref{eq:dpp}, \eqref{eq:optim_bsvar}, \eqref{eq:th}, and the fact that $\{R_k\}_{k=0}^{K-1}$ are independent, \eqref{eq:optim_tl_reg} can be written explicitly as
\begin{align*}
V(0,2,W_0)
& = \min_{g_0 \in \mathcal{U}}\E\left[  \min_{g_1 \in \mathcal{U}}\E\left[\left. \left(W_1^{g_0}(\langle g_1(W_1^{g_0}),R_1\rangle+R_f)-\frac{\gamma}{2}\right)^2 \right|\mathcal{F}_1\right] \right]\\
&\approx \min_{g_0 \in \mathcal{U}}\Bigg(\E\left[  (W_1^{g_0})^2\left(\sum_{i=1}^p(\widetilde{g}_1^i(W_1^{g_0}))^2\overline{\overline{r}}_1^{ii}+\sum_{\substack{i\neq j,i,j=1}}^p\widetilde{g}_1^i(W_1^{g_0})\widetilde{g}_1^j(W_1^{g_0})\overline{\overline{r}}_1^{ij}\right.\right.\Bigg.\\
&\qquad \Bigg.\left.
\left.+2\sum_{i=1}^p\widetilde{g}_1^i(W_1^{g_0})\overline{r}_1^iR_f+R_f^2\right)
-W_1^{g_0}\left(\sum_{i = 1}^p\widetilde{g}_1^i(W_1^{g_0})\overline{r}^i_1+R_f\right)\gamma
 \right]\Bigg.
+\frac{\gamma^2}{4} \Bigg) \approx \min_{\widetilde{\theta} \in \R^{\widetilde{d}}} v(\widetilde{\theta}),
\end{align*}
where $v$ is given by
\begin{align}\label{eq:dpptlobjv}
\begin{split}
v(\widetilde{\theta})
&:=\E\left[  (W_1^{\mathfrak{N}})^2\left(\sum_{i=1}^p(\widetilde{g}_1^i(W_1^{\mathfrak{N}}))^2\overline{\overline{r}}_1^{ii}+\sum_{\substack{i\neq j,i,j=1}}^p\widetilde{g}_1^i(W_1^{\mathfrak{N}})\widetilde{g}_1^j(W_1^{\mathfrak{N}})\overline{\overline{r}}_1^{ij}+2\sum_{i=1}^p\widetilde{g}_1^i(W_1^{\mathfrak{N}})\overline{r}_1^iR_f\right.\right.\Bigg.\\
&\qquad \Bigg.\left.
\left.+R_f^2\right)
-W_1^{\mathfrak{N}}\left(\sum_{i = 1}^p\widetilde{g}_1^i(W_1^{\mathfrak{N}})\overline{r}^i_1+R_f\right)\gamma
 \right]\Bigg.
+\frac{\gamma^2}{4}  +\frac{\eta|\widetilde{\theta}|^{2(r+1)}}{2(r+1)},\\
\end{split}
\end{align}
and where we recall that $W_1^{\mathfrak{N}} = W_{0}(\langle\mathfrak{N}( \widetilde \theta, W_{0}), R_{0}\rangle +R_f)\in \R$ with $\mathfrak{N}$ given in \eqref{eq:slfn}, $\widetilde{g}_1$ denotes the approximation of the optimal $g_1$ by the neural network defined in \eqref{def:tlfn_sigmoid} with trained parameters $\theta^* = (\theta_0^*, \dots, \theta_{K-1}^*)$, $\overline{\overline{r}}_1\in \R^{p\times p}$ with $\overline{\overline{r}}_1^{ii}:=\E[(R_1^i)^2]$ and $\overline{\overline{r}}_1^{ij}:=\E[R_1^iR_1^j]$ for $i \neq j$, $\overline{r}_1 \in \R^p$ with $\overline{r}_1^i :=\E[R_1^i]$, $\gamma>0, r\geq1/2$, and $\eta>0$. In particular, $\bar{g}_1(y) := \tanh(K_3^*\sigma_2(K_2^*\sigma_2(K_1^*y+b_1^*)+b_2^*)+b_3^*)$, $y \in \R$, where $\sigma_2(z) = 1/(1+e^{-z})$, $z \in \R^{\nu}$, is the sigmoid activation function applied componentwise, and $\theta^* = ([K_1^*],[K_2^*],[K_3^*],b_1^*,b_2^*,b_3^*)$ are the trained parameters. 
For any $\widetilde{\theta} \in \R^{\widetilde{d}}, z\in \R, r_0 \in \R^p$, denote by $
w_1^{\mathfrak{N}} := z(\langle \mathfrak{N}( \widetilde \theta, z), r_0\rangle +R_f)$. Then, the stochastic gradient $H: \R^{\widetilde{d}} \times \R^m \rightarrow \R^p$ of $v$ defined in \eqref{eq:dpptlobjv} is of the form $H(\widetilde{\theta}, x) := G(\widetilde{\theta}, x)+ F(\widetilde{\theta}, x)$ for all $\widetilde{\theta} \in \R^{\widetilde{d}}$ and all $x \in \R^m$, $m=p+1$, with $x = (r_0,z)$, $r_0=(r_0^1, \dots, r_0^p) \in \R^p, z \in \R$, where the functions $F$ and $G$ are given by
\begin{align}\label{eq:fixedfgexp1}
\begin{split}
F(\widetilde{\theta},x) 	:= \eta \widetilde{\theta}|\widetilde{\theta}|^{2r}, \quad G(\widetilde{\theta},x) 	:=\left(G_{\widetilde{K}_1^{11}}(\widetilde{\theta},x), \dots, G_{\widetilde{K}_1^{p\nu}}(\widetilde{\theta},x),  G_{\widetilde{b}^1}(\widetilde{\theta},x), \dots, G_{\widetilde{b}^{\nu}}(\widetilde{\theta},x) \right),
\end{split}
\end{align}
where for $I = 1, \dots, p, J = 1, \dots, \nu$,
\begin{align}\label{eq:fixedfgexp2}
\begin{split}
G_{\widetilde{K}_1^{IJ}}(\widetilde{\theta},x) 	&:= 2w_1^{\mathfrak{N}}\partial_{\widetilde{K}_1^{IJ}}w_1^{\mathfrak{N}}\left(\sum_{i=1}^p(\widetilde{g}_1^i(w_1^{\mathfrak{N}}))^2\overline{\overline{r}}_1^{ii}+\sum_{\substack{i\neq j\\ i,j=1}}^p\widetilde{g}_1^i(w_1^{\mathfrak{N}})\widetilde{g}_1^j(w_1^{\mathfrak{N}})\overline{\overline{r}}_1^{ii} +2\sum_{i=1}^p\widetilde{g}_1^i(w_1^{\mathfrak{N}})\overline{r}_1^iR_f+R_f^2\right)\\
&\quad +(w_1^{\mathfrak{N}})^2\left(2\sum_{i=1}^p\widetilde{g}_1^i(w_1^{\mathfrak{N}})\partial_{\widetilde{K}_1^{IJ}}\widetilde{g}_1^i(w_1^{\mathfrak{N}})\overline{\overline{r}}_1^{ii}+\sum_{\substack{i\neq j\\ i,j=1}}^p\partial_{\widetilde{K}_1^{IJ}}\widetilde{g}_1^i(w_1^{\mathfrak{N}})\widetilde{g}_1^j(w_1^{\mathfrak{N}})\overline{\overline{r}}_1^{ii} \right.\\
&\qquad \left.+\sum_{\substack{i\neq j\\ i,j=1}}^p\widetilde{g}_1^i(w_1^{\mathfrak{N}})\partial_{\widetilde{K}_1^{IJ}}\widetilde{g}_1^j(w_1^{\mathfrak{N}})\overline{\overline{r}}_1^{ii}+2\sum_{i=1}^p\partial_{\widetilde{K}_1^{IJ}}\widetilde{g}_1^i(w_1^{\mathfrak{N}})\overline{r}_1^iR_f\right) \\
&\quad-\left(\partial_{\widetilde{K}_1^{IJ}} w_1^{\mathfrak{N}}\gamma\left(\sum_{i=1}^p\widetilde{g}_1^i(w_1^{\mathfrak{N}})\overline{r}_1^i+ R_f\right)+w_1^{\mathfrak{N}}\gamma\sum_{i=1}^p\partial_{\widetilde{K}_1^{IJ}}\widetilde{g}_1^i(w_1^{\mathfrak{N}})\overline{r}_1^i\right) ,\\
G_{\widetilde{b}_0^J}(\widetilde{\theta},x) 		&:=2w_1^{\mathfrak{N}}\partial_{\widetilde{b}_0^J}w_1^{\mathfrak{N}}\left(\sum_{i=1}^p(\widetilde{g}_1^i(w_1^{\mathfrak{N}}))^2\overline{\overline{r}}_1^{ii}+\sum_{\substack{i\neq j\\ i,j=1}}^p\widetilde{g}_1^i(w_1^{\mathfrak{N}})\widetilde{g}_1^j(w_1^{\mathfrak{N}})\overline{\overline{r}}_1^{ii} +2\sum_{i=1}^p\widetilde{g}_1^i(w_1^{\mathfrak{N}})\overline{r}_1^iR_f+R_f^2\right)\\
&\quad +(w_1^{\mathfrak{N}})^2\left(2\sum_{i=1}^p\widetilde{g}_1^i(w_1^{\mathfrak{N}})\partial_{\widetilde{b}_0^J}\widetilde{g}_1^i(w_1^{\mathfrak{N}})\overline{\overline{r}}_1^{ii}+\sum_{\substack{i\neq j\\ i,j=1}}^p\partial_{\widetilde{b}_0^J}\widetilde{g}_1^i(w_1^{\mathfrak{N}})\widetilde{g}_1^j(w_1^{\mathfrak{N}})\overline{\overline{r}}_1^{ii} \right.\\
&\qquad \left.+\sum_{\substack{i\neq j\\ i,j=1}}^p\widetilde{g}_1^i(w_1^{\mathfrak{N}})\partial_{\widetilde{b}_0^J}\widetilde{g}_1^j(w_1^{\mathfrak{N}})\overline{\overline{r}}_1^{ii}+2\sum_{i=1}^p\partial_{\widetilde{b}_0^J}\widetilde{g}_1^i(w_1^{\mathfrak{N}})\overline{r}_1^iR_f\right) \\
&\quad-\left(\partial_{\widetilde{b}_0^J} w_1^{\mathfrak{N}}\gamma\left(\sum_{i=1}^p\widetilde{g}_1^i(w_1^{\mathfrak{N}})\overline{r}_1^i+ R_f\right)+w_1^{\mathfrak{N}}\gamma\sum_{i=1}^p\partial_{\widetilde{b}_0^J}\widetilde{g}_1^i(w_1^{\mathfrak{N}})\overline{r}_1^i\right) ,\\
\end{split}
\end{align}
and where
\begin{align}\label{eq:fixedfgexp3}
\begin{split}
\partial_{\widetilde{K}_1^{IJ}}w_1^{\mathfrak{N}}&:= zr_0^I\sech^2\left(\sum_{j = 1}^\nu \widetilde{K}_1^{Ij}\sigma (c^jz+\widetilde{b}_0^j)\right)(c^Jz+\widetilde{b}_0^J)\1_{A_J}(z), \quad \partial_{\widetilde{K}_1^{IJ}}\widetilde{g}_1(w_1^{\mathfrak{N}}): = \widetilde{g}_1'(w_1^{\mathfrak{N}})\partial_{\widetilde{K}_1^{IJ}}w_1^{\mathfrak{N}}\\
\partial_{\widetilde{b}_0^J}w_1^{\mathfrak{N}}&:=z\sum_{i = 1}^p\sech^2\left(\sum_{j = 1}^\nu \widetilde{K}_1^{ij}\sigma (c^jz+\widetilde{b}_0^j)\right)r_0^i  \widetilde{K}_1^{iJ}\1_{A_J}(z), \quad \partial_{\widetilde{b}_0^J}\widetilde{g}_1(w_1^{\mathfrak{N}}): = \widetilde{g}_1'(w_1^{\mathfrak{N}})\partial_{\widetilde{b}_0^J}w_1^{\mathfrak{N}}
\end{split}
\end{align}
with
\begin{equation}\label{eq:fixedfgexp4}
A_J := \{z\in \R |c^{J } z+\widetilde{b}_0^J\geq 0\},
\end{equation}
and $\widetilde{g}_1'(w_1^{\mathfrak{N}})$ denoting the derivative of $\widetilde{g}_1(y)$ w.r.t. $y$ composed with $w_1^{\mathfrak{N}}$. Then, by using \cite[Proposition 3.1]{lim2021nonasymptotic}, one can show that \eqref{eq:optim_tl_reg} satisfies Assumptions \ref{asm:AI}-\ref{asm:AC}, and thus Theorem \ref{thm:convergencew1}, Corollary \ref{thm:convergencew2}, and Theorem \ref{thm:opeer} can be applied to the optimization problem \eqref{eq:optim_tl_reg}. Indeed, we first note that the conditions imposed on $x:=(r_0,z)\in \R^m$ in \cite[Proposition 3.1]{lim2021nonasymptotic} can be satisfied for a wide range of distributions. For example, $X:=(R_0,Z)$ with $R_0$ and $Z$ being independent, $R_0$ following a (three-parameter) log-normal distribution (see \cite{sangal19703}), and $Z$ following a uniform distribution is one of the valid choices which is used in our numerical experiments. Furthermore, the stochastic gradient of $v$ defined in \eqref{eq:dpptlobjv} are given explicitly in \eqref{eq:fixedfgexp1}-\eqref{eq:fixedfgexp4}, which has a  similar form as that specified in \cite[Proposition 3.1 (20)-(21)]{lim2021nonasymptotic}. By setting $q=2, r=1, \rho =7$, following the same arguments as in the proof of \cite[Proposition 3.1]{lim2021nonasymptotic} yields the desired result.

\end{proof}

\subsection{Proof of auxiliary results in Section \ref{sec:mtproofs}}
\begin{lemma}\label{lem:GlaFlabd}
Let Assumption \ref{asm:AF} hold. Then, for any $\theta \in \R^d$, $x \in \R^m$, $0<\lambda \leq 1$, $i = 1, \dots, d$, one obtains the following estimates for $G_{\lambda}^{(i)}$ and $F_{\lambda}^{(i)}$ given in \eqref{eq:expressiontGF}:
\begin{align}
|G_\lambda^{(i)}(\theta,x)|\leq 2\lambda^{-1/2}, \quad |F_\lambda^{(i)}(\theta,x)|\leq \lambda^{-1/2}(1+|x|)^{\rho}\left(K_F+K_F|\theta^{(i)}|\right) \label{eq:Flabd}.
\end{align}
\end{lemma}
\begin{proof} For any $\theta \in \R^d$, $x \in \R^m$, $0<\lambda \leq 1$, $i = 1, \dots, d$, by using the expression of $G_{\lambda}^{(i)}(\theta, x)$ given in \eqref{eq:expressiontGF}, one obtains
\begin{align*}
\left|G_\lambda^{(i)}(\theta,x)\right|
& =\left|\frac{G^{(i)}(\theta,x)}{1+\sqrt{\lambda}|G^{(i)}(\theta,x)|}\left(1+\frac{\sqrt{\lambda}}{\varepsilon+|G^{(i)}(\theta,x)|}\right)\right|
\leq 2\lambda^{-1/2}.
\end{align*}
Furthermore, by using the expression of $F_{\lambda}^{(i)}$ given in \eqref{eq:expressiontGF}, Assumption \ref{asm:AF}, and $0<\lambda \leq 1$, one obtains
\begin{align*}
\left|F_\lambda^{(i)}(\theta, x)\right|
 = \left|\frac{F^{(i)}(\theta,x)}{1+\sqrt{\lambda}|\theta|^{2r}}\right| 
\leq  \lambda^{-1/2}(1+|x|)^{\rho}\left(K_F+K_F|\theta^{(i)}|
\right) .
\end{align*}
\end{proof}
\begin{proof}[\textbf{Proof of Lemma \ref{lem:2ndpthmmt}-\ref{lem:2ndpthmmti}}]\label{proof2ndpthmmt}
The following inequality will be applied throughout the proof: for any $z\geq 1$, $l \in \N$, $a_i\geq 0$, $i= 1, \dots, l $,
\begin{equation}\label{eq:fundamentalineq}
\left(\sum_{i = 1}^l a_i\right)^z \leq l^{z-1} \sum_{i = 1}^l a_i^z .
\end{equation}
Recall the continuous-time interpolation of e-TH$\varepsilon$O POULA given in \eqref{eq:theopoulaproc}. Throughout the proof, let $0<\lambda\leq \lambda_{1, \max}$ with $\lambda_{1, \max}$ defined in \eqref{eq:stepsizemax}, $n \in \N_0$, and $t\in (n, n+1]$. We denote by
\begin{align}\label{eq:delxinotation}
\begin{split}
\Delta_{n,t}^{\lambda}
 := \bar{\theta}^{\lambda}_n - \lambda H_{\lambda}(\bar{\theta}^{\lambda}_n, X_{n+1})(t-n),\quad
\Xi_{n,t}^{\lambda}
 := \sqrt{2\lambda \beta^{-1}}(B_t^{\lambda} - B_n^{\lambda}).
\end{split}
\end{align}
Then, one observes that
\begin{equation}\label{eq:2ndmmtexp}
\E\left[\left.|\bar{\theta}^{\lambda}_t|^2\right|\bar{\theta}^{\lambda}_n \right]  = \E\left[\left.|\Delta_{n,t}^{\lambda}|^2\right|\bar{\theta}^{\lambda}_n \right] +2\lambda (t-n)d/\beta.
\end{equation}
To obtain an upper bound for the first term on the RHS of \eqref{eq:2ndmmtexp}, note first that by \eqref{eq:delxinotation},
\begin{align}\label{eq:deltasquaredexp}
|\Delta_{n,t}^{\lambda}|^2
= |\bar{\theta}^{\lambda}_n|^2- 2\lambda(t-n) \left\langle \bar{\theta}^{\lambda}_n, H_{\lambda}(\bar{\theta}^{\lambda}_n, X_{n+1}) \right\rangle   +\lambda^2(t-n)^2|H_{\lambda}(\bar{\theta}^{\lambda}_n, X_{n+1})|^2.
\end{align}
One further calculates, by using \eqref{eq:expressiontH}, \eqref{eq:expressiontGF}, $0<\lambda\leq \lambda_{1, \max}\leq 1$, that
\begin{align}
-\left\langle \bar{\theta}^{\lambda}_n, H_{\lambda}(\bar{\theta}^{\lambda}_n, X_{n+1}) \right\rangle 
& = -\sum_{i = 1}^d \bar{\theta}^{\lambda,(i)}_n \left(  F_\lambda^{(i)}(\bar{\theta}^{\lambda}_n, X_{n+1})  + G_\lambda^{(i)}(\bar{\theta}^{\lambda}_n, X_{n+1})  \right)\nonumber\\
&\leq -\frac{ \langle \bar{\theta}^{\lambda}_n, F(\bar{\theta}^{\lambda}_n, X_{n+1})\rangle  }{1+\sqrt{\lambda}|\bar{\theta}^{\lambda}_n|^{2r}}  +\sum_{i = 1}^d\frac{ |\bar{\theta}^{\lambda,(i)}_n ||G^{(i)}(\bar{\theta}^{\lambda}_n, X_{n+1}) |(1+\sqrt{\lambda}|\bar{\theta}^{\lambda}_n|^{2r}) }{(1+\sqrt{\lambda}|G^{(i)}(\bar{\theta}^{\lambda}_n, X_{n+1}) |)(1+\sqrt{\lambda}|\bar{\theta}^{\lambda}_n|^{2r})} +d|\bar{\theta}^{\lambda}_n|.\nonumber
\end{align}
Then, by using Assumption \ref{asm:AG}, one obtains
\begin{align}
-\left\langle \bar{\theta}^{\lambda}_n, H_{\lambda}(\bar{\theta}^{\lambda}_n, X_{n+1}) \right\rangle
&\leq -\frac{ \langle \bar{\theta}^{\lambda}_n, F(\bar{\theta}^{\lambda}_n, X_{n+1})\rangle  }{1+\sqrt{\lambda}|\bar{\theta}^{\lambda}_n|^{2r}}+ d|\bar{\theta}^{\lambda}_n|
 + \frac{ dK_G(1+|X_{n+1}|)^{\rho}(1+|\bar{\theta}^{\lambda}_n|)^{q+1}}{{1+\sqrt{\lambda}|\bar{\theta}^{\lambda}_n|^{2r}}} + \frac{ d|\bar{\theta}^{\lambda}_n|^{2r+1}}{1+\sqrt{\lambda}|\bar{\theta}^{\lambda}_n|^{2r}}, \nonumber
\end{align}
which, by using $1\leq q \leq 2r$, and \eqref{eq:fundamentalineq} (with $l \leftarrow 2$, $z \leftarrow 2r+1$), yields
\begin{align}
\begin{split}\label{eq:deltasqp2ub}
-\left\langle \bar{\theta}^{\lambda}_n, H_{\lambda}(\bar{\theta}^{\lambda}_n, X_{n+1}) \right\rangle
&\leq  -\frac{ \langle \bar{\theta}^{\lambda}_n, F(\bar{\theta}^{\lambda}_n, X_{n+1})\rangle  }{1+\sqrt{\lambda}|\bar{\theta}^{\lambda}_n|^{2r}} +d|\bar{\theta}^{\lambda}_n| \\
&\quad + 2^{2r}dK_G(1+|X_{n+1}|)^{\rho} +\frac{ d(2^{2r} K_G(1+|X_{n+1}|)^{\rho} +1)|\bar{\theta}^{\lambda}_n|^{2r+1}}{1+\sqrt{\lambda}|\bar{\theta}^{\lambda}_n|^{2r}}.
\end{split}
\end{align}
Furthermore, by using \eqref{eq:expressiontH}, Lemma \ref{lem:GlaFlabd}, and \eqref{eq:fundamentalineq}, one notes that  
\begin{align}
\begin{split}\label{eq:deltasqp3ub}
|H_{\lambda}(\bar{\theta}^{\lambda}_n, X_{n+1})|^2
&\leq d\lambda^{-1}(4+4(1+|X_{n+1}|)^{\rho}K_F)+ 4d\lambda^{-1}(1+|X_{n+1}|)^{\rho}K_F|\bar{\theta}^{\lambda}_n| \\
&\quad +3d\lambda^{-1}(1+|X_{n+1}|)^{2\rho} K_F^2 +\frac{3 \lambda^{-1}(1+|X_{n+1}|)^{2\rho}K_F^2|\bar{\theta}^{\lambda}_n|^2}{(1+\sqrt{\lambda}|\bar{\theta}^{\lambda}_n|^{2r})^2}\\
&\quad+\frac{3 (1+|X_{n+1}|)^{2\rho}K_F^2|\bar{\theta}^{\lambda}_n|^{4r+2}}{(1+\sqrt{\lambda}|\bar{\theta}^{\lambda}_n|^{2r})^2}.
\end{split}
\end{align}
Substituting \eqref{eq:deltasqp2ub}, \eqref{eq:deltasqp3ub} into \eqref{eq:deltasquaredexp} yields
\begin{align}
\begin{split}\label{eq:deltasquaredub1}
|\Delta_{n,t}^{\lambda}|^2
&\leq  |\bar{\theta}^{\lambda}_n|^2  -\frac{  2\lambda(t-n)  \langle \bar{\theta}^{\lambda}_n, F(\bar{\theta}^{\lambda}_n, X_{n+1})\rangle  }{1+\sqrt{\lambda}|\bar{\theta}^{\lambda}_n|^{2r}}   \\
&\quad + \lambda(t-n) d(1+|X_{n+1}|)^{2\rho} ( 2^{2r+1} K_G+4+4K_F+3K_F^2)  \\
&\quad + 4\lambda(t-n) d(1+|X_{n+1}|)^{ \rho} (1+K_F)|\bar{\theta}^{\lambda}_n|\\
&\quad +\frac{2\lambda(t-n)  d(2^{2r} K_G(1+|X_{n+1}|)^{\rho} +1)|\bar{\theta}^{\lambda}_n|^{2r+1}}{1+\sqrt{\lambda}|\bar{\theta}^{\lambda}_n|^{2r}}  \\
&\quad  +\frac{3\lambda(t-n)  (1+|X_{n+1}|)^{ 2\rho}K_F^2(1+|\bar{\theta}^{\lambda}_n|^{2r+1})}{ 1+\sqrt{\lambda}|\bar{\theta}^{\lambda}_n|^{2r} }\\
&\quad+\frac{3\lambda^2(t-n)^2 (1+|X_{n+1}|)^{2 \rho}K_F^2|\bar{\theta}^{\lambda}_n|^{4r+2}}{(1+\sqrt{\lambda}|\bar{\theta}^{\lambda}_n|^{2r})^2},
\end{split}
\end{align}
where the inequality holds due to $0< t-n \leq 1$ and $a^2 \leq 1+ a^{2r+1}$, for $a \geq 0$. 
Moreover, one observes that, for any $\theta \in \R^d$,
\begin{equation}\label{eq:theta2r1ineq}
|\theta| = \frac{ |\theta| (1+\sqrt{\lambda}|\theta|^{2r} )}{ 1+\sqrt{\lambda}|\theta|^{2r} } \leq  \frac{ |\theta|  +|\theta|^{2r+1}  }{ 1+\sqrt{\lambda}|\theta|^{2r} } \leq \frac{ 1  +2|\theta|^{2r+1}  }{ 1+\sqrt{\lambda}|\theta|^{2r} } \leq 1+ \frac{ 2|\theta|^{2r+1}  }{ 1+\sqrt{\lambda}|\theta|^{2r} }.
\end{equation}
Applying \eqref{eq:theta2r1ineq} to \eqref{eq:deltasquaredub1} and using $0< t-n \leq 1$ yield,
\begin{align}\label{eq:deltasquaredub2}
\begin{split}
|\Delta_{n,t}^{\lambda}|^2
&\leq  |\bar{\theta}^{\lambda}_n|^2-\frac{  2\lambda(t-n)  \langle \bar{\theta}^{\lambda}_n, F(\bar{\theta}^{\lambda}_n, X_{n+1})\rangle  }{1+\sqrt{\lambda}|\bar{\theta}^{\lambda}_n|^{2r}}    \\
&\quad +  \lambda(t-n) d(1+|X_{n+1}|)^{2\rho} ( 2^{2r+1} K_G+8+8K_F+6K_F^2) \\
&\quad  +\frac{\lambda(t-n)  d(1+|X_{n+1}|)^{2\rho}(2^{2r+1} K_G +10+8K_F+3K_F^2)|\bar{\theta}^{\lambda}_n|^{2r+1}}{1+\sqrt{\lambda}|\bar{\theta}^{\lambda}_n|^{2r}}   \\
&\quad+\frac{3\lambda^2(t-n) (1+|X_{n+1}|)^{2\rho}K_F^2|\bar{\theta}^{\lambda}_n|^{4r+2}}{(1+\sqrt{\lambda}|\bar{\theta}^{\lambda}_n|^{2r})^2}.
\end{split}
\end{align}
By taking conditional expectation on both sides, by using Remark \ref{rem:Fhdissiposl}, and by the fact that $X_{n+1}$ is independent of $\bar{\theta}^{\lambda}_n$, the above result yields,
\begin{align}
\E\left[\left.|\Delta_{n,t}^{\lambda}|^2\right|\bar{\theta}^{\lambda}_n \right]
&\leq |\bar{\theta}^{\lambda}_n|^2- \frac{2\lambda(t-n) a_F |\bar{\theta}^{\lambda}_n|^{2r+2}}{1+\sqrt{\lambda}|\bar{\theta}^{\lambda}_n|^{2r}} \nonumber\\
&\quad +\lambda(t-n) d\E[(1+|X_0|)^{2\rho}] ( 2b_F+2^{2r+1} K_G+8+8K_F+6K_F^2) \nonumber\\
&\quad + \frac{ \lambda(t-n)d\E[(1+|X_0|)^{2\rho}](2^{2r+1}K_G+10+8K_F+3K_F^2)|\bar{\theta}^{\lambda}_n|^{2r+1}}{{1+\sqrt{\lambda}|\bar{\theta}^{\lambda}_n|^{2r}} } \nonumber\\
&\quad + \frac{3\lambda^2 (t-n)\E[(1+|X_0|)^{2\rho}]K_F^2|\bar{\theta}^{\lambda}_n|^{4r+2}}{(1+\sqrt{\lambda}|\bar{\theta}^{\lambda}_n|^{2r})^2}  \nonumber\\
\begin{split}\label{eq:2ndmmtue1}
& =  |\bar{\theta}^{\lambda}_n|^2-\lambda(t-n) T_1(\bar{\theta}^{\lambda}_n)|\bar{\theta}^{\lambda}_n|^2- \lambda(t-n) T_2(\bar{\theta}^{\lambda}_n)\\
&\quad +\lambda(t-n) d\E[(1+|X_0|)^{2\rho}] ( 2b_F+2^{2r+1} K_G+8+8K_F+6K_F^2) ,
\end{split}
\end{align}
where for any $\theta \in \R^d\setminus \{(0,\dots, 0)_d\}$,
$
T_1(\theta)  := \frac{ a_F |\theta|^{2r+2}-d\E[(1+|X_0|)^{2\rho}](2^{2r+1}K_G+10+8K_F+3K_F^2)|\theta|^{2r+1}}{|\theta|^2(1+\sqrt{\lambda}|\theta|^{2r})} ,
$ 
and for any $\theta \in \R^d$,
$
T_2(\theta)  := \frac{  a_F |\theta|^{2r+2}}{1+\sqrt{\lambda}|\theta|^{2r}} - \frac{3\lambda \E[(1+|X_0|)^{2\rho}] K_F^2|\theta|^{4r+2}}{(1+\sqrt{\lambda}|\theta|^{2r})^2}.
$ 
One observes that $T_1(\theta) >\frac{  a_F |\theta|^{2r}}{2(1+\sqrt{\lambda}|\theta|^{2r})}$ implies $|\theta| > \frac{2d\E[(1+|X_0|)^{2\rho}](2^{2r+1}K_G+10+8K_F+3K_F^2)}{a_F }$.
Then, denote by $M_0 : = 2d\E[(1+|X_0|)^{2\rho}](2^{2r+1}K_G+10+8K_F+3K_F^2)/\min{\{1, a_F \}}$. The above calculation and the fact that $f(s) := s/(1+\sqrt{\lambda}s)$ is non-decreasing for all $s\geq 0$ imply that, for any $|\theta|>M_0$,
\begin{equation}\label{eq:2ndmmtleT1}
T_1(\theta)>\frac{ a_F |\theta|^{2r}}{2(1+\sqrt{\lambda}|\theta|^{2r})} >\frac{ a_F M_0^{2r}}{2(1+\sqrt{\lambda}M_0^{2r})}\geq a_F \kappa,
\end{equation}
where $\kappa: =M_0^{2r}/(2(1+ M_0^{2r}))$. In addition, for any $\theta \in \R^d$, one notes that
$
T_2(\theta) 
\geq ( \sqrt{\lambda}a_F |\theta|^{4r+2} - 3\lambda \E[(1+|X_0|)^{2\rho}]K_F^2|\theta|^{4r+2})/(1+\sqrt{\lambda}|\theta|^{2r})^2 \geq 0
$ 
since $0 <\lambda \leq \lambda_{1, \max}\leq a_F^2 /(9 (\E[(1+|X_0|)^{2\rho}])^2K_F^4 )$ by the definition of $\lambda_{1, \max}$ in \eqref{eq:stepsizemax}. 
Therefore, we obtain that
\begin{equation}\label{eq:2ndmmtleT2}
T_2(\theta)\geq 0.
\end{equation}
Denote by $\mathsf{S}_{n,M_0} := \{\omega \in \Omega: |\bar{\theta}^{\lambda}_n(\omega)| >M_0\}$. Then, by using \eqref{eq:2ndmmtue1}, \eqref{eq:2ndmmtleT1}, \eqref{eq:2ndmmtleT2}, we have that 
\begin{equation}\label{eq:2ndmmtue2}
\E\left[\left.|\Delta_{n,t}^{\lambda}|^2\1_{\mathsf{S}_{n,M_0} }\right|\bar{\theta}^{\lambda}_n \right]
 \leq  (1-\lambda(t-n)  a_F \kappa)|\bar{\theta}^{\lambda}_n|^2 \1_{\mathsf{S}_{n,M_0} }+\lambda(t-n)c_1\1_{\mathsf{S}_{n,M_0} },
\end{equation}
where $c_1 := d\E[(1+|X_0|)^{2\rho}] ( 2b_F+2^{2r+1} K_G+8+8K_F+6K_F^2)$. Moreover, by using \eqref{eq:2ndmmtue1}, the expression of $T_1(\theta)$ for $\theta \in \R^d\setminus \{(0,\dots, 0)_d\}$, and \eqref{eq:2ndmmtleT2}, one notes that 
\begin{align}\label{eq:2ndmmtue3}
\begin{split}
 \E\left[\left.|\Delta_{n,t}^{\lambda}|^2\1_{\mathsf{S}_{n,M_0}^{\mathsf{c}} }\right|\bar{\theta}^{\lambda}_n \right]
 &\leq  (1-\lambda(t-n)  a_F \kappa)|\bar{\theta}^{\lambda}_n|^2\1_{\mathsf{S}_{n,M_0}^{\mathsf{c}} } +\lambda(t-n)\left(c_1 +  a_F \kappa M_0^2\right.\\
&\quad \left.+d\E[(1+|X_0|)^{2\rho}](2^{2r+1}K_G+10+8K_F+3K_F^2)M_0^{2r+1}\right)\1_{\mathsf{S}_{n,M_0}^{\mathsf{c}} }.
 \end{split}
\end{align}
By using \eqref{eq:2ndmmtue2}, \eqref{eq:2ndmmtue3}, one obtains that 
\begin{align}\label{eq:2ndmmtue4}
\E\left[\left.|\Delta_{n,t}^{\lambda}|^2\right|\bar{\theta}^{\lambda}_n \right]
& = \E\left[\left.|\Delta_{n,t}^{\lambda}|^2(\1_{\mathsf{S}_{n,M_0} }+\1_{\mathsf{S}_{n,M_0}^{\mathsf{c}} })\right|\bar{\theta}^{\lambda}_n \right]
\leq  (1-\lambda(t-n)  a_F\kappa)|\bar{\theta}^{\lambda}_n|^2+\lambda(t-n)c_2,
\end{align}
where $c_2 := c_1 + a_F \kappa M_0^2+d\E[(1+|X_0|)^{2\rho}](2^{2r+1}K_G+10+8K_F+3K_F^2)M_0^{2r+1}$. Substituting \eqref{eq:2ndmmtue4} into \eqref{eq:2ndmmtexp} yields 
$\E\left[\left.|\bar{\theta}^{\lambda}_t|^2\right|\bar{\theta}^{\lambda}_n \right]  =(1-\lambda(t-n)  a_F \kappa)|\bar{\theta}^{\lambda}_n|^2+\lambda(t-n)c_0$, 
where
\begin{align}\label{eq:2ndmmtconstc0}
\begin{split}
\kappa &: =M_0^{2r}/(2(1+ M_0^{2r})),\\
M_0 &: =  2d\E[(1+|X_0|)^{2\rho}](2^{2r+1}K_G+10+8K_F+3K_F^2)/\min{\{1, a_F \}},\\
c_0& := 2d/\beta +d\E[(1+|X_0|)^{2\rho}] ( 2b_F+2^{2r+1} K_G+8+8K_F+6K_F^2)\\
&\quad +  a_F \kappa M_0^2+d\E[(1+|X_0|)^{2\rho}](2^{2r+1}K_G+10+8K_F+3K_F^2)M_0^{2r+1}.
\end{split}
\end{align}
By induction, one concludes that
\begin{align}\label{eq:2ndmmtinduction}
\begin{split}
\E\left[ |\bar{\theta}^{\lambda}_t|^2  \right]
& \leq (1-\lambda(t-n)  a_F \kappa)(1-\lambda  a_F \kappa)^n\E\left[ |\theta_0|^2 \right]+\mathring{c}_0,
\end{split}
\end{align}
where $\mathring{c}_0:=c_0(1+1/( a_F \kappa))$.
\end{proof}
\begin{proof}[\textbf{Proof of Lemma \ref{lem:2ndpthmmt}-\ref{lem:2ndpthmmtii}}]

Let $p \in [2, \infty) \cap \N$, $0 <\lambda \leq \lambda_{p, \max}$ with $\lambda_{p, \max}$ given in \eqref{eq:stepsizemax}, $n \in \N_0$, and $t\in (n, n+1]$. Recall the explicit expression of e-TH$\varepsilon$O POULA given in \eqref{eq:theopoulaproc}, and the definitions of $\Delta_{n,t}^{\lambda}$, $\Xi_{n,t}^{\lambda}$ given in \eqref{eq:delxinotation}. By using \cite[Lemma A.3]{nonconvex}, straightforward calculations yield
\begin{align}
\begin{split}
\E\left[\left.|\bar{\theta}^{\lambda}_t|^{2p}\right|\bar{\theta}^{\lambda}_n \right]
& \leq \E\left[\left.|\Delta_{n,t}^{\lambda}|^{2p}\right|\bar{\theta}^{\lambda}_n \right]  + 2p \E\left[\left.|\Delta_{n,t}^{\lambda}|^{2p-2} \langle \Delta_{n,t}^{\lambda}, \Xi_{n,t}^{\lambda} \rangle\right|\bar{\theta}^{\lambda}_n \right]  \\
&\quad + \sum_{k = 2}^{2p}\binom{2p}{k}\E\left[\left.|\Delta_{n,t}^{\lambda}|^{2p-k} |\Xi_{n,t}^{\lambda}|^k \right|\bar{\theta}^{\lambda}_n \right]
\end{split}\label{eq:2pthmmt0}\\
& =  \E\left[\left.|\Delta_{n,t}^{\lambda}|^{2p}\right|\bar{\theta}^{\lambda}_n \right]
+\sum_{k = 2}^{2p}\binom{2p}{k}\E\left[\left.|\Delta_{n,t}^{\lambda}|^{2p-k} |\Xi_{n,t}^{\lambda}|^k \right|\bar{\theta}^{\lambda}_n \right] \label{eq:2pthmmt1} ,
\end{align}
where the last equality holds due to the fact that the second term in \eqref{eq:2pthmmt0} is zero. The second term on the RHS of \eqref{eq:2pthmmt1} can be further estimated as
\begin{align}\label{eq:2pthmmt2}
&\sum_{k = 2}^{2p}\binom{2p}{k}\E\left[\left.|\Delta_{n,t}^{\lambda}|^{2p-k} |\Xi_{n,t}^{\lambda}|^k \right|\bar{\theta}^{\lambda}_n \right]\nonumber\\
&\leq 2^{2p-2}p(2p-1)\lambda (t-n)d \beta^{-1}\E\left[\left. |\Delta_{n,t}^{\lambda}|^{2p-2}   \right|\bar{\theta}^{\lambda}_n \right]
+2^{2p-4}(2p(2p-1))^{p+1}(d\beta^{-1}\lambda (t-n) )^p,
\end{align}
where 
the inequality holds due to \eqref{eq:fundamentalineq}, the fact that $\Delta_{n,t}^{\lambda}$ is independent of $\Xi_{n,t}^{\lambda}$, and \cite[Theorem 7.1]{mao2007stochastic}. Substituting \eqref{eq:2pthmmt2} into \eqref{eq:2pthmmt1} yields
\begin{align}\label{eq:2pthmmt3}
\begin{split}
\E\left[\left.|\bar{\theta}^{\lambda}_t|^{2p}\right|\bar{\theta}^{\lambda}_n \right]
& \leq  \E\left[\left.|\Delta_{n,t}^{\lambda}|^{2p}\right|\bar{\theta}^{\lambda}_n \right]
+2^{2p-2}p(2p-1)\lambda (t-n)d \beta^{-1}\E\left[\left. |\Delta_{n,t}^{\lambda}|^{2p-2}   \right|\bar{\theta}^{\lambda}_n \right]   \\
&\quad +2^{2p-4}(2p(2p-1))^{p+1}(d\beta^{-1}\lambda (t-n) )^p .
\end{split}
\end{align}
Then, we proceed with establishing an upper estimate for the first term on the RHS of \eqref{eq:2pthmmt3}. By using \eqref{eq:delxinotation}, \cite[Lemma A.3]{nonconvex}, and by using the same arguments as in \eqref{eq:2pthmmt0}, one obtains
\begin{align}\label{eq:2pthmmt4}
\begin{split}
\E\left[\left.|\Delta_{n,t}^{\lambda}|^{2p}\right|\bar{\theta}^{\lambda}_n \right]
&\leq |\bar{\theta}^{\lambda}_n|^{2p} - 2p \lambda(t-n) |\bar{\theta}^{\lambda}_n|^{2p-2}\E\left[\left. \langle \bar{\theta}^{\lambda}_n,  H_{\lambda}(\bar{\theta}^{\lambda}_n, X_{n+1}) \rangle\right|\bar{\theta}^{\lambda}_n \right]   \\
&\quad + \sum_{k = 2}^{2p}\binom{2p}{k}\lambda^k(t-n)^k|\bar{\theta}^{\lambda}_n|^{2p-k} \E\left[\left.| H_{\lambda}(\bar{\theta}^{\lambda}_n, X_{n+1})|^k \right|\bar{\theta}^{\lambda}_n \right].
\end{split}
\end{align}
By Lemma \ref{lem:GlaFlabd}, and by applying \eqref{eq:fundamentalineq} (with $l \leftarrow 2$, $z \leftarrow 2$), one observes that 
\begin{align}\label{eq:GlaFlafullnormbd}
\begin{split}
|G_{\lambda}(\bar{\theta}^{\lambda}_n, X_{n+1})|^2 &\leq 4d \lambda^{-1}, \quad
|F_{\lambda}(\bar{\theta}^{\lambda}_n, X_{n+1})|^2 \leq 2\lambda^{-1}(1+|X_{n+1}|)^{2\rho}K_F^2(d+|\bar{\theta}^{\lambda}_n |^2).
\end{split}
\end{align}
For any $k = 2, \dots, 2p$, by using \eqref{eq:expressiontH}, \eqref{eq:fundamentalineq}, 
and \eqref{eq:GlaFlafullnormbd}, one obtains that
\begin{align}\label{eq:2pthmmtHlakpower}
\begin{split}
| H_{\lambda}(\bar{\theta}^{\lambda}_n, X_{n+1})|^k
&\leq 2^{2k-1}d^{k/2}\lambda^{-k/2}+2^{2k-2}d^{k/2}\lambda^{-k/2}(1+|X_{n+1}|)^{\rho k}K_F^k\\
&\quad +2^{2k-2}\lambda^{-k/2}(1+|X_{n+1}|)^{\rho k}K_F^k|\bar{\theta}^{\lambda}_n|^k.
\end{split}
\end{align}
Substituting \eqref{eq:deltasqp2ub}, \eqref{eq:2pthmmtHlakpower} into \eqref{eq:2pthmmt4} yields
\begin{align}\label{eq:2pthmmt5}
\begin{split}
\E\left[\left.|\Delta_{n,t}^{\lambda}|^{2p}\right|\bar{\theta}^{\lambda}_n \right]
& \leq |\bar{\theta}^{\lambda}_n|^{2p}
-  \frac{ 2p \lambda(t-n) a_F|\bar{\theta}^{\lambda}_n|^{2r+2p} }{ 1+\sqrt{\lambda}|\bar{\theta}^{\lambda}_n|^{2r}}   + 2p d\lambda(t-n) |\bar{\theta}^{\lambda}_n|^{2p-1}   \\
&\quad  +  pd \lambda(t-n)(2b_F +2^{2r+1}K_G \E\left[ (1+|X_0|)^{\rho}  \right])  |\bar{\theta}^{\lambda}_n|^{2p-2}  \\
&\quad+\frac{ 2pd\lambda(t-n)(2^{2r} K_G\E\left[ (1+|X_0|)^{\rho}  \right] +1)|\bar{\theta}^{\lambda}_n|^{2r+2p-1}}
{1+\sqrt{\lambda}|\bar{\theta}^{\lambda}_n|^{2r}}\\
&\quad + \sum_{k = 2}^{2p}\binom{2p}{k}\lambda^k(t-n)^k|\bar{\theta}^{\lambda}_n|^{2p-k}2^{2k-2}d^{k/2}\lambda^{-k/2}\E[(1+|X_0|)^{\rho k}](2+K_F^k)\\
&\quad + \sum_{k = 2}^{2p}\binom{2p}{k}\lambda^k(t-n)^k|\bar{\theta}^{\lambda}_n|^{2p}2^{2k-2}\lambda^{-k/2}\E[(1+|X_0|)^{\rho k}]K_F^k,
\end{split}
\end{align}
where we use the independence of $\bar{\theta}^{\lambda}_n$ and $X_{n+1}$. For any $\theta \in \R^d$ 
and for $\nu = -2, -1, 2r-2$, one observes that $|\theta|^{2p+\nu} \leq 1+|\theta|^{2p+2r-1}$, then, by using the same arguments as in \eqref{eq:theta2r1ineq}, one obtains
\begin{equation}\label{eq:theta2p2rminus1ineq}
|\theta|^{2p-1} \leq 1+\frac{ 2|\theta|^{2p+2r-1}  }{ 1+\sqrt{\lambda}|\theta|^{2r} }, \quad |\theta|^{2p-2} \leq 2+\frac{ 2|\theta|^{2p+2r-1}  }{ 1+\sqrt{\lambda}|\theta|^{2r} }.
\end{equation}
By applying \eqref{eq:theta2p2rminus1ineq} to \eqref{eq:2pthmmt5}, and by using $|\theta|^{2p-k} \leq 1+|\theta|^{2p-1}$, for any $\theta \in \R^d$, 
$2 \leq k \leq 2p$, we have
\begin{align}\label{eq:2pthmmt6}
\begin{split}
\E\left[\left.|\Delta_{n,t}^{\lambda}|^{2p}\right|\bar{\theta}^{\lambda}_n \right]
& \leq |\bar{\theta}^{\lambda}_n|^{2p}
-  \frac{ 2p \lambda(t-n) a_F|\bar{\theta}^{\lambda}_n|^{2r+2p} }{ 1+\sqrt{\lambda}|\bar{\theta}^{\lambda}_n|^{2r}}    \\
&  +  pd \lambda(t-n)(2+4b_F+2^{2r+2}  K_G\E\left[(1+|X_0|)^{\rho}  \right])   \\
&+\frac{p d\lambda(t-n)(2^{2r+3} K_G\E\left[ (1+|X_0|)^{\rho}  \right]+4b_F  +6)|\bar{\theta}^{\lambda}_n|^{2r+2p-1}}
{1+\sqrt{\lambda}|\bar{\theta}^{\lambda}_n|^{2r}}\\
& +\binom{2p}{p}(2p-1) 2^{4p-2}\lambda (t-n)d^{p} (2+K_F)^{2p} \E[(1+|X_0|)^{2p\rho}](1+|\bar{\theta}^{\lambda}_n|^{2p-1})\\
& +\sum_{k = 2}^{2p}\binom{2p}{k}\lambda^{k/2} (t-n)2^{2k-2} K_F^k \E[(1+|X_0|)^{\rho k}]|\bar{\theta}^{\lambda}_n|^{2p}.
\end{split}
\end{align}
By using the first inequality in \eqref{eq:theta2p2rminus1ineq}, and by using $
|\theta|^{2p} =  \frac{|\theta|^{2p} + \sqrt{\lambda}|\theta|^{2p+2r}  }{ 1+\sqrt{\lambda}|\theta|^{2r} } \leq 1+\frac{  |\theta|^{2p+2r-1}  }{ 1+\sqrt{\lambda}|\theta|^{2r} }+ \frac{ \sqrt{\lambda}|\theta|^{2p+2r}  }{ 1+\sqrt{\lambda}|\theta|^{2r} }$, for any $\theta \in \R^d$, 
\eqref{eq:2pthmmt6} can be upper bounded as follows:
\begin{align}\label{eq:2pthmmt7}
\begin{split}
\E\left[\left.|\Delta_{n,t}^{\lambda}|^{2p}\right|\bar{\theta}^{\lambda}_n \right]
& \leq|\bar{\theta}^{\lambda}_n|^{2p}  -\lambda(t-n)T_3(\bar{\theta}^{\lambda}_n)|\bar{\theta}^{\lambda}_n|^{2p} - \lambda (t-n) T_4(\bar{\theta}^{\lambda}_n)\\
& \quad + \lambda(t-n)pd(2+4b_F+2^{2r+2}  K_G\E\left[(1+|X_0|)^{\rho}  \right])  \\
&\quad  +\lambda (t-n)d^{p}\binom{2p}{p}(2p-1) 2^{4p} (2+K_F)^{2p} \E[(1+|X_0|)^{2p\rho}],
\end{split}
\end{align}
where for any $\theta \in \R^d$, $T_3(\theta):= (a_F|\theta|^{2r}-p d(2^{2r+3} K_G\E\left[ (1+|X_0|)^{\rho}  \right]+4b_F  +6)|\theta|^{2r -1} -d^p\binom{2p}{p}(2p-1) 2^{4p} (2+K_F)^{2p} \E[(1+|X_0|)^{2p\rho}]|\theta|^{2r -1})/(1+\sqrt{\lambda}|\theta|^{2r})
$, and
\begin{equation}\label{eq:2pthmmtT4exp}
T_4(\theta) :=\sum_{k = 2}^{2p}\left(\frac{ a_F|\theta|^{2r+2p} }{ 1+\sqrt{\lambda}|\theta|^{2r}}
- \binom{2p}{k}\lambda^{(k-1)/2}2^{2k-2} K_F^k\E[(1+|X_0|)^{\rho k}]\frac{  |\theta|^{2p+2r}  }{ 1+\sqrt{\lambda}|\theta|^{2r} }\right).
\end{equation}
One obtains that $T_3(\theta)>\frac{  a_F|\theta|^{2r} }{ 2(1+\sqrt{\lambda}|\theta|^{2r})}$ implies $|\theta| >(2p d(2^{2r+3} K_G\E\left[ (1+|X_0|)^{\rho}  \right]+4b_F  +6)+d^p\binom{2p}{p}(2p-1) 2^{4p+1} (2+K_F)^{2p} \E[(1+|X_0|)^{2p\rho}])/a_F$.
Denote by $M_1(p) : =  (2p d(2^{2r+3} K_G\E\left[ (1+|X_0|)^{2\rho}  \right]+4b_F  +6)+ d^p\binom{2p}{p}(2p-1) 2^{4p+1} (2+K_F)^{2p} \E[(1+|X_0|)^{2p\rho}] )/\min{\{1, a_F \}}$. The above inequality and the fact that $f(s) := s/(1+\sqrt{\lambda}s)$ is non-decreasing for any $s \geq 0$ imply that, for any $|\theta |>M_1(p)$,
\begin{equation}\label{eq:2pthmmtleT3}
T_3(\theta)
> \frac{  a_F|\theta|^{2r} }{ 2(1+\sqrt{\lambda}|\theta|^{2r})}
\geq \frac{  a_F(M_1(p))^{2r} }{ 2(1+\sqrt{\lambda}(M_1(p))^{2r})} \geq 2a_F\bar{\kappa}(p),
\end{equation}
where $\bar{\kappa}(p): = (M_1(p))^{2r}/(4(1+(M_1(p))^{2r}))$. 
Furthermore, one notes that, for any $\theta \in \R^d$, $k = 2, \dots, 2p$, 
\begin{align}\label{eq:2pthmmtleT4i}
\begin{split}
\frac{ a_F|\theta|^{2r+2p} }{ 1+\sqrt{\lambda}|\theta|^{2r}}
- \binom{2p}{k}\lambda^{(k-1)/2}2^{2k-2} K_F^k\E[(1+|X_0|)^{\rho k}]\frac{  |\theta|^{2p+2r}  }{ 1+\sqrt{\lambda}|\theta|^{2r} }\geq 0 &\\
 \Leftrightarrow \quad \lambda \leq  \lambda_{p,k}: = \frac{ (a_F/K_F)^{2/(k-1)}}{16K_F^2(\binom{2p}{k}\E[(1+|X_0|)^{\rho k}])^{2/(k-1)} }.&
\end{split}
\end{align}
To see that $\lambda \leq \lambda_{p, \max} \leq \lambda_{p,k}$ is indeed satisfied, observes that, for any $2\leq k \leq 2p$,
$
(\binom{2p}{k}\E[(1+|X_0|)^{\rho k}])^{\frac{2}{k-1}}
\leq p^2(2p-1)^2(\E[(1+|X_0|)^{2p\rho }])^2$. 
By using this inequality and the definition of $\lambda_{p, \max}$ in \eqref{eq:stepsizemax}, it indeed holds that, 
\begin{align}\label{eq:2pthmmtleT4ii}
\lambda \leq \lambda_{p, \max} \leq \frac{\min\{(a_F/K_F)^2, (a_F/K_F)^{2/(2p-1)}\}}{16K_F^2 p^2(2p-1)^2(\E[(1+|X_0|)^{2p\rho }])^2} \leq \lambda_{p,k}, \quad k = 2, \dots, 2p.
\end{align}
Thus, by using \eqref{eq:2pthmmtleT4i}, \eqref{eq:2pthmmtleT4ii}, and the expression of $T_4$ in \eqref{eq:2pthmmtT4exp},  
one obtains, for any $\theta \in \R^d$, that
\begin{equation}\label{eq:2pthmmtleT4}
T_4(\theta)\geq 0.
\end{equation}
Denote by $\mathsf{S}_{n,M_1(p)} := \{\omega \in \Omega: |\bar{\theta}^{\lambda}_n(\omega)| >M_1(p)\}$. Substituting \eqref{eq:2pthmmtleT3}, \eqref{eq:2pthmmtleT4} into \eqref{eq:2pthmmt7} yields 
\begin{align}\label{eq:2pthmmt8}
\E\left[\left.|\Delta_{n,t}^{\lambda}|^{2p}\1_{\mathsf{S}_{n,M_1(p)} }\right|\bar{\theta}^{\lambda}_n \right]
& \leq  (1  -2\lambda(t-n)a_F\bar{\kappa}(p)) |\bar{\theta}^{\lambda}_n|^{2p} \1_{\mathsf{S}_{n,M_1(p)} } +\lambda(t-n)c_3(p)\1_{\mathsf{S}_{n,M_1(p)} },
\end{align}
where $c_3(p)= pd\E\left[(1+|X_0|)^{2\rho}  \right](2+4b_F+2^{2r+2}  K_G) + d^{p}\binom{2p}{p}(2p-1) 2^{4p} (2+K_F)^{2p} \E[(1+|X_0|)^{2p\rho}]$. Similarly, by using \eqref{eq:2pthmmt7}, \eqref{eq:2pthmmtleT4}, one obtains 
\begin{align}\label{eq:2pthmmt9}
\begin{split}
&\E\left[\left.|\Delta_{n,t}^{\lambda}|^{2p}\1_{\mathsf{S}_{n,M_1(p)}^{\mathsf{c}} }\right|\bar{\theta}^{\lambda}_n \right]  \\
& \leq  (1  -2\lambda(t-n)a_F\bar{\kappa}(p)) |\bar{\theta}^{\lambda}_n|^{2p}\1_{\mathsf{S}_{n,M_1(p)}^{\mathsf{c}} } +\lambda(t-n)(c_3(p)+2a_F\bar{\kappa}(p) (M_1(p))^{2p})\1_{\mathsf{S}_{n,M_1(p)}^{\mathsf{c}} }   \\
&\quad +\lambda(t-n)p d(2^{2r+3} K_G\E\left[ (1+|X_0|)^{\rho}  \right]+4b_F  +6)(M_1(p))^{2r+2p-1}\1_{\mathsf{S}_{n,M_1(p)}^{\mathsf{c}} }\\
&\quad +\lambda(t-n)d^p\binom{2p}{p}(2p-1) 2^{4p} (2+K_F)^{2p} \E[(1+|X_0|)^{2p\rho}](M_1(p))^{2r+2p-1}\1_{\mathsf{S}_{n,M_1(p)}^{\mathsf{c}} }.
\end{split}
\end{align}
Combining the results in \eqref{eq:2pthmmt8} and \eqref{eq:2pthmmt9} yields
\begin{align}\label{eq:2pthmmtdelta2p}
\E\left[\left.|\Delta_{n,t}^{\lambda}|^{2p} \right|\bar{\theta}^{\lambda}_n \right]
&\leq (1  -2\lambda(t-n)a_F\bar{\kappa}(p)) |\bar{\theta}^{\lambda}_n|^{2p}+\lambda(t-n)c_4(p),
\end{align}
where
$
c_4(p)
:= c_3(p)+2a_F\bar{\kappa}(p) (M_1(p))^{2p}
+p d\E\left[ (1+|X_0|)^{2\rho}  \right](2^{2r+3} K_G+4b_F  +6)(M_1(p))^{2r+2p-1}
+ d^p\binom{2p}{p}(2p-1) 2^{4p} (2+K_F)^{2p} \E[(1+|X_0|)^{2p\rho}](M_1(p))^{2r+2p-1}.
$
Define
\begin{align*}
M_1(1) &: =  2 d(2^{2r+3} K_G\E\left[ (1+|X_0|)^{2\rho}  \right]+4b_F  +6) /\min{\{1, a_F \}}\\
&\quad+ 64d    (2+K_F)^{2} \E[(1+|X_0|)^{2 \rho}]  /\min{\{1, a_F \}},\\
c_4(1) &:= d\E\left[(1+|X_0|)^{2\rho}  \right](2+4b_F+2^{2r+2}  K_G) + 32d (2+K_F)^{2} \E[(1+|X_0|)^{2 \rho}]\\
&\quad +2a_F(M_1(1))^{2r+2}/(4(1+(M_1(1))^{2r})) \\
&\quad +  d\E\left[ (1+|X_0|)^{2\rho}  \right](2^{2r+3} K_G+4b_F  +6)(M_1(1))^{2r+1}\\
&\quad +32 d (2+K_F)^{2 } \E[(1+|X_0|)^{2 \rho}](M_1(1))^{2r+1}.
\end{align*}
By using \eqref{eq:2ndmmtue4} and the fact that $0 <\lambda \leq \lambda_{1, \max}\leq a_F^2 /(9 (\E[(1+|X_0|)^{2\rho}])^2K_F^4 )$, one obtains
$
\E\left[\left.|\Delta_{n,t}^{\lambda}|^2 \right|\bar{\theta}^{\lambda}_n \right] \leq |\bar{\theta}^{\lambda}_n|^2+\lambda(t-n)c_2 \leq |\bar{\theta}^{\lambda}_n|^2+\lambda(t-n)c_4(1),
$ 
which, together with \eqref{eq:2pthmmtdelta2p} implies 
\begin{align}\label{eq:2pthmmtdelta2p2}
\E\left[\left.|\Delta_{n,t}^{\lambda}|^{2(p-1)} \right|\bar{\theta}^{\lambda}_n \right]
&\leq   |\bar{\theta}^{\lambda}_n|^{2(p-1)}+\lambda(t-n)c_4(p-1).
\end{align}
By substituting \eqref{eq:2pthmmtdelta2p}, \eqref{eq:2pthmmtdelta2p2} into \eqref{eq:2pthmmt3}, one obtains
\begin{align*}
\E\left[\left.|\bar{\theta}^{\lambda}_t|^{2p}\right|\bar{\theta}^{\lambda}_n \right]
& \leq  (1  -2\lambda(t-n)a_F\bar{\kappa}(p)) |\bar{\theta}^{\lambda}_n|^{2p}+\lambda(t-n)c_4(p)\\
&\quad +2^{2p-2}p(2p-1)\lambda (t-n)d \beta^{-1}\left(|\bar{\theta}^{\lambda}_n|^{2(p-1)}+\lambda(t-n)c_4(p-1) \right)\\
&\quad +2^{2p-4}(2p(2p-1))^{p+1}(d\beta^{-1}\lambda (t-n) )^p.
\end{align*}
One observes that for any $\theta \in \R^d$, $
-\lambda(t-n)a_F\bar{\kappa}(p)|\theta|^{2p}+2^{2p-2}p(2p-1)\lambda (t-n)d \beta^{-1}|\theta|^{2(p-1)}<0$
implies $  |\theta|>\left(2^{2p-2}p(2p-1)d \beta^{-1}/(a_F\bar{\kappa}(p))\right)^{1/2} =:M_2(p)$. 
Let $\mathsf{S}_{n,M_2(p)} := \{\omega \in \Omega: |\bar{\theta}^{\lambda}_n(\omega)| >M_2(p)\}$. By using the above inequality, and by using $0< \lambda \leq \lambda_{p, \max} \leq 1$, $0<t-n\leq 1$, one obtains 
\begin{align}\label{eq:2pthmmt10}
\begin{split}
\E\left[\left.|\bar{\theta}^{\lambda}_t|^{2p}\1_{\mathsf{S}_{n,M_2(p)} } \right|\bar{\theta}^{\lambda}_n \right]
& \leq  (1  -\lambda(t-n)a_F\bar{\kappa}(p)) |\bar{\theta}^{\lambda}_n|^{2p}\1_{\mathsf{S}_{n,M_2(p)} }+\lambda(t-n)c_4(p)\1_{\mathsf{S}_{n,M_2(p)} }\\
&\quad +\lambda (t-n) 2^{2p-2}p(2p-1) d \beta^{-1} c_4(p-1) \1_{\mathsf{S}_{n,M_2(p)} } \\
&\quad +\lambda (t-n) 2^{2p-4}(2p(2p-1))^{p+1}(d\beta^{-1})^p\1_{\mathsf{S}_{n,M_2(p)} }.
\end{split}
\end{align}
In addition, straightforward calculations yield
\begin{align}\label{eq:2pthmmt11}
\begin{split}
\E\left[\left.|\bar{\theta}^{\lambda}_t|^{2p}\1_{\mathsf{S}_{n,M_2(p)}^{\mathsf{c}} }\right|\bar{\theta}^{\lambda}_n \right]
& \leq  (1  -\lambda(t-n)a_F\bar{\kappa}(p)) |\bar{\theta}^{\lambda}_n|^{2p}\1_{\mathsf{S}_{n,M_2(p)}^{\mathsf{c}} }+\lambda(t-n)c_4(p)\1_{\mathsf{S}_{n,M_2(p)}^{\mathsf{c}} }\\
& +\lambda (t-n) 2^{2p-2}p(2p-1) d \beta^{-1} ((M_2(p))^{2(p-1)}+c_4(p-1) )\1_{\mathsf{S}_{n,M_2(p)}^{\mathsf{c}} }\\
& +\lambda (t-n) 2^{2p-4}(2p(2p-1))^{p+1}(d\beta^{-1})^p\1_{\mathsf{S}_{n,M_2(p)}^{\mathsf{c}} }.
\end{split}
\end{align}
Combining \eqref{eq:2pthmmt10} and \eqref{eq:2pthmmt11} yields
\begin{equation}
\E\left[\left.|\bar{\theta}^{\lambda}_t|^{2p}\right|\bar{\theta}^{\lambda}_n \right] \leq  (1  -\lambda(t-n)a_F\bar{\kappa}(p)) |\bar{\theta}^{\lambda}_n|^{2p}+\lambda(t-n)\bar{c}_0(p),
\end{equation}
where
\begin{align}\label{eq:2pthmmtconstbarc0}
\begin{split}
\bar{\kappa}(p)&: = (M_1(p))^{2r}/(4(1+(M_1(p))^{2r})),\\
M_1(p) &: = \Big(2p d(2^{2r+3} K_G\E\left[ (1+|X_0|)^{2\rho}  \right]+4b_F  +6)\Big. \\
&\quad \Big. + d^p\binom{2p}{p}(2p-1) 2^{4p+1} (2+K_F)^{2p} \E[(1+|X_0|)^{2p\rho}] \Big)/\min{\{1, a_F \}},\\
\bar{c}_0(p)&:=c_4(p)+2^{2p-2}p(2p-1) d \beta^{-1} ((M_2(p))^{2(p-1)}+c_4(p-1) )  +2^{2p-4}(2p(2p-1))^{p+1}(d\beta^{-1})^p,\\
c_4(p) &:= pd\E\left[(1+|X_0|)^{2\rho}  \right](2+4b_F+2^{2r+2}  K_G) +2a_F\bar{\kappa}(p) (M_1(p))^{2p}\\
&\quad +p d\E\left[ (1+|X_0|)^{2\rho}  \right](2^{2r+3} K_G+4b_F  +6)(M_1(p))^{2r+2p-1}\\
&\quad + d^p\binom{2p}{p}(2p-1) 2^{4p} (2+K_F)^{2p} \E[(1+|X_0|)^{2p\rho}](1+(M_1(p))^{2r+2p-1}),\\
M_2(p)& : = \left(2^{2p-2}p(2p-1)d \beta^{-1}/(a_F\bar{\kappa}(p))\right)^{1/2}.
\end{split}
\end{align}
Finally, by using the same arguments as in \eqref{eq:2ndmmtinduction}, and by using $\bar{\kappa}(p) \geq \bar{\kappa}(2)$, 
one obtains that
\begin{align*}
\E\left[ |\bar{\theta}^{\lambda}_t|^{2p}  \right]
&\leq  (1-\lambda(t-n)  a_F \kappa^{\sharp}_2 )(1- \lambda  a_F \kappa^{\sharp}_2)^n \E\left[|\theta_0|^{2p}\right]  +c_0^{\sharp}(p)(1+1/(  a_F \kappa^{\sharp}_2)),
\end{align*}
where $\kappa^{\sharp}_2 := \min\{\bar{\kappa}(2), \tilde{\kappa}(2)\}$, $c_0^{\sharp}(p):= \max\{\bar{c}_0(p), \tilde{c}_0(p)\}$ with $\bar{\kappa}(2)$, $\bar{c}_0(p)$ and $\tilde{\kappa}(2)$, $\tilde{c}_0(p)$ given in \eqref{eq:2pthmmtconstbarc0} and \eqref{eq:2pthmmtconsttildec0}, respectively.
\end{proof}

\begin{proof}[\textbf{Proof of Lemma \ref{lem:2ndpthmmt}-\ref{lem:2ndpthmmtiii}}] Let $p \in [2, \infty) \cap \N$, $0 <\lambda \leq \hat{\lambda}_{\max}$ with $\hat{\lambda}_{\max}$ given in \eqref{eq:stepsizemaxrelaxed}, $n \in \N_0$, and $t\in (n, n+1]$. Since $F(\theta, x) = \hat{F}(\theta)$, for any $\theta \in \R^d$, $x \in \R^m$, by Remark \ref{rem:Fhdissiposl}, one obtains,
\begin{equation}\label{eq:Fdissipsp}
\langle \theta,  \hat{F}(\theta)\rangle  = \langle \theta, \E[F(\theta, X_0)]\rangle \geq a_F|\theta|^{2r+2} - b_F,
\end{equation}
for any $\theta \in \R^d$, where $a_F := a/2$ and $b_F := (a/2+b)R_F^{\bar{r}+2}+dK_F^2\E[(1+|X_0|)^{2\rho}]/{2a}$ with $
R_F := \max\{(4b/a)^{1/(2r-\bar{r})}, 2^{1/(2r)}\}$. 
Moreover, by using the same arguments as in the proof of Lemma \ref{lem:GlaFlabd}, one obtains, for any $\theta \in \R^d$,
\begin{equation}\label{eq:Flabdsp}
|F_\lambda^{(i)}(\theta,x)|= \left|\frac{\hat{F}^{(i)}(\theta) }{1+\sqrt{\lambda}|\theta|^{2r}}\right| = \left|\frac{\E[F^{(i)}(\theta,X_0)]}{1+\sqrt{\lambda}|\theta|^{2r}}\right| \leq \lambda^{-1/2}\E[(1+|X_0|)^{\rho}]\left(K_F+K_F|\theta^{(i)}|\right).
\end{equation}
In order to obtain the 2p-th moment estimate of e-TH$\varepsilon$O POULA \eqref{eq:theopoulaproc} under a relaxed stepsize restriction, we follow the proof of Lemma \ref{lem:2ndpthmmt}-\ref{lem:2ndpthmmtii} up to \eqref{eq:2pthmmt3}, and establish upper bounds for $\E\left[\left. |\Delta_{n,t}^{\lambda}|^{2p}   \right|\bar{\theta}^{\lambda}_n \right]$ and $\E\left[\left. |\Delta_{n,t}^{\lambda}|^{2p-2}   \right|\bar{\theta}^{\lambda}_n \right]$ 
using the following method. By using the arguments in the proof of Lemma \ref{lem:2ndpthmmt}-\ref{lem:2ndpthmmti} up to \eqref{eq:deltasquaredub2} but with Remark \ref{rem:Fhdissiposl} replaced by \eqref{eq:Fdissipsp} and \eqref{eq:Flabd} in Lemma \ref{lem:GlaFlabd} replaced by \eqref{eq:Flabdsp}, one obtains
\begin{align}
\begin{split}\label{eq:2pthmmtFthetaub1}
|\Delta_{n,t}^{\lambda}|^2
&\leq  \left(1-  \frac{  \lambda(t-n)  a_F|\bar{\theta}^{\lambda}_n|^{2r} }{1+\sqrt{\lambda}|\bar{\theta}^{\lambda}_n|^{2r}} \right)|\bar{\theta}^{\lambda}_n|^2 - \lambda (t-n)(T_5(\bar{\theta}^{\lambda}_n) + d(1+|X_{n+1}|)^{2\rho}2^{2r+1} K_G) \\
&\quad+\frac{\lambda(t-n)  d(1+|X_{n+1}|)^{2\rho}(2^{2r+1} K_G +2)|\bar{\theta}^{\lambda}_n|^{2r+1}}{1+\sqrt{\lambda}|\bar{\theta}^{\lambda}_n|^{2r}}  \\
&\quad +  \lambda(t-n) d\E[(1+|X_0|)^{2\rho}] ( 2b_F+8+8K_F+6K_F^2) \\
&\quad  +\frac{\lambda(t-n)  d\E[(1+|X_0|)^{2\rho}](8+8K_F+3K_F^2)|\bar{\theta}^{\lambda}_n|^{2r+1}}{1+\sqrt{\lambda}|\bar{\theta}^{\lambda}_n|^{2r}},
\end{split}
\end{align}
where, for any $\theta \in \R^d$, $T_5(\theta) := \frac{  a_F|\theta|^{2r+2} }{1+\sqrt{\lambda}|\theta|^{2r}} - \frac{3\lambda  K_F^2\E[(1+|X_0|)^{2\rho}]|\theta|^{4r+2}}{(1+\sqrt{\lambda}|\theta|^{2r})^2}$. One observes that, for any $\theta \in \R^d$,
\begin{align}\label{eq:2pthmmtleT5}
\begin{split}
T_5(\theta) 
& \geq   \frac{ \sqrt{\lambda}a_F|\theta|^{4r+2} -3\lambda  K_F^2\E[(1+|X_0|)^{2\rho}]|\theta|^{4r+2}}{(1+\sqrt{\lambda}|\theta|^{2r})^2}  \geq 0,
\end{split}
\end{align}
since $0<\lambda\leq \hat{\lambda}_{\max}\leq a_F^2/(9K_F^4(\E[(1+|X_0|)^{2\rho}])^2)$. By using \eqref{eq:2pthmmtleT5}, 
\eqref{eq:2pthmmtFthetaub1} becomes
\begin{equation}\label{eqn:2pthmmtdeltasqT6T7}
|\Delta_{n,t}^{\lambda}|^2 \leq T_6(\bar{\theta}^{\lambda}_n) +T_7(\bar{\theta}^{\lambda}_n, X_{n+1}),
\end{equation}
where for any $\theta \in \R^d$,
$
T_6(\theta):=\left(1-  \frac{ \lambda(t-n)  a_F|\theta|^{2r} }{1+\sqrt{\lambda}|\theta|^{2r}} \right)|\theta|^2,
$ 
and for any $\theta \in \R^d$, $x \in \R^m$,
\begin{align*}
T_7(\theta, x)&:= \lambda(t-n) d(1+|x|)^{2\rho}2^{2r+1} K_G\\
&\quad+\frac{\lambda(t-n)  d(1+|x|)^{2\rho}(2^{2r+1} K_G +2)|\theta|^{2r+1}}{1+\sqrt{\lambda}|\theta|^{2r}}  \\
&\quad +  \lambda(t-n) d\E[(1+|X_0|)^{2\rho}] ( 2b_F+8+8K_F+6K_F^2) \\
&\quad  +\frac{\lambda(t-n)  d\E[(1+|X_0|)^{2\rho}](8+8K_F+3K_F^2)|\theta|^{2r+1}}{1+\sqrt{\lambda}|\theta|^{2r}}.
\end{align*}
One notes that, since $0<\lambda\leq \hat{\lambda}_{\max}\leq 1/(a_F^2)$, one has for any $\theta \in \R^d\setminus \{(0,\dots, 0)_d\}$ that $
(1-   \lambda(t-n)  a_F|\theta|^{2r} /(1+\sqrt{\lambda}|\theta|^{2r}) )\in(0,1).
$ 
Therefore, by using \eqref{eqn:2pthmmtdeltasqT6T7}, further calculations yield
\begin{align*}
&\E\left[\left. |\Delta_{n,t}^{\lambda}|^{2p}   \right|\bar{\theta}^{\lambda}_n \right]
 = \sum_{k = 0}^p \binom{p}{k} \E\left[\left. (T_6(\bar{\theta}^{\lambda}_n) )^{p-k} (T_7(\bar{\theta}^{\lambda}_n, X_{n+1}))^k \right|\bar{\theta}^{\lambda}_n \right]\\
&\leq \left(1-  \frac{ \lambda(t-n)   a_F|\bar{\theta}^{\lambda}_n|^{2r} }{1+\sqrt{\lambda}|\bar{\theta}^{\lambda}_n|^{2r}} \right)|\bar{\theta}^{\lambda}_n|^{2p}\\
&\quad +\sum_{k = 1}^p \binom{p}{k} |\bar{\theta}^{\lambda}_n|^{2(p-k)} 4^k\lambda^k (t-n)^k d^k\E[(1+|X_0|)^{2k\rho}]( 2^{2r+1} K_G+2b_F+8+8K_F+6K_F^2)^k \\
&\quad  +\sum_{k = 1}^p \binom{p}{k}  4^k\frac{\lambda^k(t-n)^k  d^k\E[(1+|X_0|)^{2k\rho}](2^{2r+1} K_G +10+8K_F+3K_F^2)^k|\bar{\theta}^{\lambda}_n|^{2p+2rk-k}}{(1+\sqrt{\lambda}|\bar{\theta}^{\lambda}_n|^{2r})^k}.
\end{align*}
where the inequality is obtained by using Jensen's inequality and \eqref{eq:fundamentalineq}. 
For any $\theta \in \R^d$, 
$1 \leq k \leq p$, it holds that $|\theta|^{2p-2k} \leq 1+|\theta|^{2p-1}$ and $|\theta|^{2p+2rk-k} \leq 1+|\theta|^{2p+2rk-1}$. By using the aforementioned inequalities together with \eqref{eq:theta2p2rminus1ineq}, one obtains 
\begin{align}\label{eq:2pthmmtFthetaub2}
\E\left[\left. |\Delta_{n,t}^{\lambda}|^{2p}   \right|\bar{\theta}^{\lambda}_n \right] 
&\leq\left(1-  \frac{ \lambda(t-n)   a_F|\bar{\theta}^{\lambda}_n|^{2r} }{2(1+\sqrt{\lambda}|\bar{\theta}^{\lambda}_n|^{2r})} \right)|\bar{\theta}^{\lambda}_n|^{2p}- \lambda(t-n)T_8(\bar{\theta}^{\lambda}_n)+\lambda(t-n)c_5(p),
\end{align}
where
$
c_5(p)
:=d^p \binom{p}{\lcrc{p/2}} p  4^{p+1}\E[(1+|X_0|)^{2p\rho}]( 2^{2r+1} K_G+2b_F+8+8K_F+6K_F^2)^p
+d^p\binom{p}{\lcrc{p/2}}p  4^p\E[(1+|X_0|)^{2p\rho}](2^{2r+1} K_G +10+8K_F+3K_F^2)^p,
$ 
and where for any $\theta \in \R^d$,
$
T_8(\theta):=\frac{    a_F|\theta|^{2r+2p} }{2(1+\sqrt{\lambda}|\theta|^{2r})}
 - \frac{d^p \binom{p}{\lcrc{p/2}} p  4^{p+1}\E[(1+|X_0|)^{2p\rho}]( 2^{2r+1} K_G+2b_F+8+8K_F+6K_F^2)^p|\theta|^{2p+2r-1}}{1+\sqrt{\lambda}|\theta|^{2r}}
 - d^p\binom{p}{\lcrc{p/2}}  4^p\E[(1+|X_0|)^{2p\rho}](2^{2r+1} K_G +10+8K_F+3K_F^2)^p\sum_{k = 1}^p \frac{\lambda^{k-1}  |\theta|^{2p+2rk-1}}{(1+\sqrt{\lambda}|\theta|^{2r})^k} .
$ 
One notes that,
\begin{align}
\begin{split}\label{eqn:2pthmmtT8lb1}
&\frac{     a_F|\theta|^{2r+2p} }{4(1+\sqrt{\lambda}|\theta|^{2r})}\\
& - \frac{d^p \binom{p}{\lcrc{p/2}} p  4^{p+1}\E[(1+|X_0|)^{2p\rho}]( 2^{2r+1} K_G+2b_F+8+8K_F+6K_F^2)^p|\theta|^{2p+2r-1}}{1+\sqrt{\lambda}|\theta|^{2r}}>0
\end{split}\\
&\Leftrightarrow |\theta|>d^p \binom{p}{\lcrc{p/2}} p  4^{p+2}\E[(1+|X_0|)^{2p\rho}]( 2^{2r+1} K_G+2b_F+8+8K_F+6K_F^2)^p/  a_F,  \nonumber
\end{align}
and moreover, for any $1 \leq k \leq p$,
\begin{align}
\begin{split}\label{eqn:2pthmmtT8lb2}
&\frac{    a_F\lambda^{(k-1)/2}|\theta|^{2rk+2p} }{4p(1+\sqrt{\lambda}|\theta|^{2r})^k}\\
& -   \frac{\lambda^{(k-1)/2}d^p\binom{p}{\lcrc{p/2}}  4^p\E[(1+|X_0|)^{2p\rho}](2^{2r+1} K_G +10+8K_F+3K_F^2)^p  |\theta|^{2p+2rk-1}}{(1+\sqrt{\lambda}|\theta|^{2r})^k}>0
\end{split}\\
&\Leftrightarrow |\theta|>  d^p\binom{p}{\lcrc{p/2}} p 4^{p+1}\E[(1+|X_0|)^{2p\rho}](2^{2r+1} K_G +10+8K_F+3K_F^2)^p/  a_F. \nonumber
\end{align}
Denote by $M_3(p):=d^p \binom{p}{\lcrc{p/2}} p  4^{p+2}\E[(1+|X_0|)^{2p\rho}]( 2^{2r+1} K_G+2b_F+10+8K_F+6K_F^2)^p/ \min\{1, a_F\}$. Then, by using \eqref{eqn:2pthmmtT8lb1} and \eqref{eqn:2pthmmtT8lb2}, one obtains, for any $|\theta|>M_3(p)$, that
\begin{align}\label{eq:2pthmmtleT8}
T_8(\theta)
& >0,
\end{align}
In addition, by using the fact that $f(s) := s/(1+\sqrt{\lambda}s)$ is non-decreasing for any $s \geq 0$, we have, for any $|\theta|>M_3(p)$, that 
\begin{equation}\label{eq:2pthmmtlecontraT}
\frac{    a_F|\theta|^{2r} }{2(1+\sqrt{\lambda}|\theta|^{2r})} \geq \frac{    a_F(M_3(p))^{2r} }{2(1+\sqrt{\lambda}(M_3(p))^{2r})} \geq 2 a_F \tilde{\kappa}(p),
\end{equation}
where $\tilde{\kappa}(p):= (M_3(p))^{2r} /(4(1+ (M_3(p))^{2r}))$. Denote by $\mathsf{S}_{n,M_3(p)} := \{\omega \in \Omega: |\bar{\theta}^{\lambda}_n(\omega)| >M_3(p)\}$. By using \eqref{eq:2pthmmtleT8} and \eqref{eq:2pthmmtlecontraT}, the RHS of \eqref{eq:2pthmmtFthetaub2} can be upper bounded by
\begin{align}\label{eq:2pthmmtFthetaub3}
\E\left[\left. |\Delta_{n,t}^{\lambda}|^{2p}\1_{\mathsf{S}_{n,M_3(p)} }   \right|\bar{\theta}^{\lambda}_n \right]
&\leq \left(1-   2\lambda(t-n)   a_F \tilde{\kappa}(p)   \right)|\bar{\theta}^{\lambda}_n|^{2p} \1_{\mathsf{S}_{n,M_3(p)} }+\lambda(t-n)c_5(p)\1_{\mathsf{S}_{n,M_3(p)} }.
\end{align}
Similarly, one obtains that
\begin{align}
\begin{split}\label{eq:2pthmmtFthetaub4}
\hspace{-0.5em}\E\left[\left. |\Delta_{n,t}^{\lambda}|^{2p}\1_{\mathsf{S}_{n,M_3(p)}^{\mathsf{c}} }   \right|\bar{\theta}^{\lambda}_n \right]
&\leq \left(1-   2\lambda(t-n)   a_F \tilde{\kappa}(p)   \right)|\bar{\theta}^{\lambda}_n|^{2p}  \1_{\mathsf{S}_{n,M_3(p)}^{\mathsf{c}} } +\lambda(t-n)\1_{\mathsf{S}_{n,M_3(p)}^{\mathsf{c}} }\Big(c_5(p)\Big. \\
& \quad+2 a_F \tilde{\kappa}(p)(M_3(p))^{2p}  + d^p \binom{p}{\lcrc{p/2}} p  4^{p+2}\E[(1+|X_0|)^{2p\rho}]\\
&\qquad \times\Big.( 2^{2r+1} K_G+2b_F+10+8K_F+6K_F^2)^p(M_3(p))^{2p+2rp-1}\Big).
\end{split}
\end{align}
Combining \eqref{eq:2pthmmtFthetaub3} and \eqref{eq:2pthmmtFthetaub4} yields, for any $p \in [2, \infty) \cap \N$,
\begin{equation}\label{eq:2pthmmtexpdelta2pub}
\E\left[\left. |\Delta_{n,t}^{\lambda}|^{2p}  \right|\bar{\theta}^{\lambda}_n \right] \leq \left(1-   2\lambda(t-n)   a_F \tilde{\kappa}(p)   \right)|\bar{\theta}^{\lambda}_n|^{2p} +\lambda(t-n)c_6(p),
\end{equation}
where $c_6(p): = c_5(p)+2a_F  \tilde{\kappa}(p)(M_3(p))^{2p}+d^p \binom{p}{\lcrc{p/2}} p  4^{p+2}\E[(1+|X_0|)^{2p\rho}] ( 2^{2r+1} K_G+2b_F+10+8K_F+6K_F^2)^p (M_3(p))^{2p+2rp-1} $. Define
$M_3(1):=64d  \E[(1+|X_0|)^{2 \rho}]( 2^{2r+1} K_G+2b_F+10+8K_F+6K_F^2) / \min\{1, a_F\}$,
$\tilde{\kappa}(1):= (M_3(1))^{2r} /(4(1+ (M_3(1))^{2r}))$,
$c_5(1):=16 d   \E[(1+|X_0|)^{2 \rho}]( 2^{2r+1} K_G+2b_F+8+8K_F+6K_F^2)
+4d \E[(1+|X_0|)^{2 \rho}](2^{2r+1} K_G +10+8K_F+3K_F^2)$, and $c_6(1): = c_5(1)+2a_F  \tilde{\kappa}(1)(M_3(1))^2
+ 64d   \E[(1+|X_0|)^{2 \rho}] ( 2^{2r+1} K_G+2b_F+10+8K_F+6K_F^2)  (M_3(1))^{ 2r+1}
$. 
By using \eqref{eq:2ndmmtue4}, one observes that 
\begin{equation}\label{eq:2pthmmtexpdelta2ub}
\E\left[\left. |\Delta_{n,t}^{\lambda}|^2  \right|\bar{\theta}^{\lambda}_n \right]
\leq   |\bar{\theta}^{\lambda}_n|^2+\lambda(t-n)c_2 \leq |\bar{\theta}^{\lambda}_n|^2+\lambda(t-n)c_6(1).
\end{equation}
Then, by using \eqref{eq:2pthmmtexpdelta2pub}, \eqref{eq:2pthmmtexpdelta2ub}, 
one obtains the following result:
\begin{equation}\label{eq:2pthmmtexpdelta2p2ub}
\E\left[\left. |\Delta_{n,t}^{\lambda}|^{2p-2}  \right|\bar{\theta}^{\lambda}_n \right] \leq  |\bar{\theta}^{\lambda}_n|^{2p-2} +\lambda(t-n)c_6(p-1).
\end{equation}
Substituting \eqref{eq:2pthmmtexpdelta2pub}, \eqref{eq:2pthmmtexpdelta2p2ub} into \eqref{eq:2pthmmt3} yields
\begin{align}\label{eq:2pthmmtFthetaub5}
\begin{split}
\E\left[\left.|\bar{\theta}^{\lambda}_t|^{2p}\right|\bar{\theta}^{\lambda}_n \right]
& \leq  \left(1-   2\lambda(t-n)   a_F \tilde{\kappa}(p)   \right)|\bar{\theta}^{\lambda}_n|^{2p} +\lambda(t-n)c_6(p)\\
&\quad +2^{2p-2}p(2p-1)\lambda (t-n)d \beta^{-1}(|\bar{\theta}^{\lambda}_n|^{2p-2} +\lambda(t-n)c_6(p-1))  \\
&\quad +2^{2p-4}(2p(2p-1))^{p+1}(d\beta^{-1}\lambda (t-n) )^p .
\end{split}
\end{align}
One notes that $
-   \lambda(t-n)   a_F \tilde{\kappa}(p)    |\bar{\theta}^{\lambda}_n|^{2p}  +2^{2p-2}p(2p-1)\lambda (t-n)d \beta^{-1}|\bar{\theta}^{\lambda}_n|^{2p-2}<0$ implies
$ |\theta|>(2^{2p-2}p(2p-1)d \beta^{-1}/( a_F \tilde{\kappa}(p) ))^{1/2}$. Denote by $M_4(p) : = (2^{2p-2}p(2p-1)d \beta^{-1}/( a_F \tilde{\kappa}(p) ))^{1/2}$ and $\mathsf{S}_{n,M_4(p)} := \{\omega \in \Omega: |\bar{\theta}^{\lambda}_n(\omega)| >M_4(p)\}$. By using \eqref{eq:2pthmmtFthetaub5} and the above inequalities, one obtains 
\begin{align}\label{eq:2pthmmtFthetaub6}
\begin{split}
\E\left[\left.|\bar{\theta}^{\lambda}_t|^{2p} \1_{\mathsf{S}_{n,M_4(p)} } \right|\bar{\theta}^{\lambda}_n \right]
&\leq \left(1-  \lambda(t-n)   a_F \tilde{\kappa}(p)   \right)|\bar{\theta}^{\lambda}_n|^{2p} \1_{\mathsf{S}_{n,M_4(p)} }+\lambda(t-n)c_6(p)  \1_{\mathsf{S}_{n,M_4(p)} } \\
&\quad +\lambda (t-n)2^{2p-2}p(2p-1)d \beta^{-1} c_6(p-1)   \1_{\mathsf{S}_{n,M_4(p)} } \\
&\quad +\lambda (t-n)2^{2p-4}(2p(2p-1))^{p+1}(d\beta^{-1} )^p  \1_{\mathsf{S}_{n,M_4(p)} } .
\end{split}
\end{align}
Furthermore, it holds that
\begin{align}\label{eq:2pthmmtFthetaub7}
\begin{split}
\E\left[\left.|\bar{\theta}^{\lambda}_t|^{2p} \1_{\mathsf{S}_{n,M_4(p)}^{\mathsf{c}} }   \right|\bar{\theta}^{\lambda}_n \right]
&\leq \left(1-  \lambda(t-n)   a_F \tilde{\kappa}(p)   \right)|\bar{\theta}^{\lambda}_n|^{2p} \1_{\mathsf{S}_{n,M_4(p)}^{\mathsf{c}} } +\lambda(t-n)c_6(p) \1_{\mathsf{S}_{n,M_4(p)}^{\mathsf{c}} }\\
&\quad+ \lambda (t-n)2^{2p-2}p(2p-1)d \beta^{-1} ((M_4(p))^{2p-2}+c_6(p-1))  \1_{\mathsf{S}_{n,M_4(p)}^{\mathsf{c}} } \\
&\quad +\lambda (t-n)2^{2p-4}(2p(2p-1))^{p+1}(d\beta^{-1} )^p  \1_{\mathsf{S}_{n,M_4(p)}^{\mathsf{c}} }  .
\end{split}
\end{align}
Combining the two estimates in \eqref{eq:2pthmmtFthetaub6} and \eqref{eq:2pthmmtFthetaub7} yields
\[
\E\left[\left.|\bar{\theta}^{\lambda}_t|^{2p} \right|\bar{\theta}^{\lambda}_n \right]
\leq \left(1-  \lambda(t-n)   a_F \tilde{\kappa}(p)   \right)|\bar{\theta}^{\lambda}_n|^{2p} +\lambda (t-n)\tilde{c}_0(p),
\]
where
\begin{align}\label{eq:2pthmmtconsttildec0}
\begin{split}
\tilde{\kappa}(p)&:= (M_3(p))^{2r} /(4(1+ (M_3(p))^{2r})),\\
M_3(p)&:=d^p \binom{p}{\lcrc{p/2}} p  4^{p+2}\E[(1+|X_0|)^{2p\rho}]( 2^{2r+1} K_G+2b_F+10+8K_F+6K_F^2)^p/ \min\{1, a_F\},\\
\tilde{c}_0(p)&:=c_6(p) +2^{2p-2}p(2p-1)d \beta^{-1} ((M_4(p))^{2p-2}+c_6(p-1))    +2^{2p-4}(2p(2p-1))^{p+1}(d\beta^{-1})^p,\\
c_6(p)&: =d^p \binom{p}{\lcrc{p/2}} p  4^{p+2}\E[(1+|X_0|)^{2p\rho}]( 2^{2r+1} K_G+2b_F+10+8K_F+6K_F^2)^p\\
&\quad+ 2a_F  \tilde{\kappa}(p)(M_3(p))^{2p} +d^p \binom{p}{\lcrc{p/2}} p  4^{p+2}\E[(1+|X_0|)^{2p\rho}] \\
&\qquad \times( 2^{2r+1} K_G+2b_F+10+8K_F+6K_F^2)^p (M_3(p))^{2p+2rp-1} ,\\
M_4(p)& : = (2^{2p-2}p(2p-1)d \beta^{-1}/( a_F \tilde{\kappa}(p) ))^{1/2}.
\end{split}
\end{align}
Therefore, by using similar arguments as in \eqref{eq:2ndmmtinduction}, and by using $\tilde{\kappa}(p) \geq \tilde{\kappa}(2)$, 
one obtains 
\begin{align*}
\E\left[ |\bar{\theta}^{\lambda}_t|^{2p}  \right]
&\leq  (1-\lambda(t-n)  a_F \kappa^{\sharp}_2 )(1- \lambda  a_F \kappa^{\sharp}_2)^n \E\left[|\theta_0|^{2p}\right]  +c_0^{\sharp}(p)(1+1/(  a_F \kappa^{\sharp}_2)),
\end{align*}
where $\kappa^{\sharp}_2 := \min\{\bar{\kappa}(2), \tilde{\kappa}(2)\}$, $c_0^{\sharp}(p):= \max\{\bar{c}_0(p), \tilde{c}_0(p)\}$ with $\bar{\kappa}(2)$, $\bar{c}_0(p)$ and $\tilde{\kappa}(2)$, $\tilde{c}_0(p)$ given in \eqref{eq:2pthmmtconstbarc0} and \eqref{eq:2pthmmtconsttildec0}, respectively.
\end{proof}
\begin{lemma}\label{lem:ose4thpower}  Let Assumptions \ref{asm:AI}, \ref{asm:AG}, \ref{asm:AF}, and \ref{asm:AC} hold.  Then, for any $0<\lambda\leq \lambda_{\max}$ with $\lambda_{\max}$ given in \eqref{eq:stepsizemax}, $t \geq 0$, one obtains
\[
\E\left[|\bar{\theta}^{\lambda}_t  - \bar{\theta}^{\lambda}_{\lfrf{t}}|^4\right] \leq \lambda^2\left(e^{-\lambda  a_F \kappa^\sharp_2\lfrf{t} }\bar{C}_{0,1}\E\left[|\theta_0|^{4(2r+1)}\right]+\bar{C}_{1,1}\right),
\]
where
\begin{align}\label{eq:ose4thpowerconsts}
\begin{split}
\bar{C}_{0,1}
& :=  2^{8r+6}d^4(1+K_F+K_G)^4\E\left[(1+|X_0|)^{4\rho}\right],\\
\bar{C}_{1,1}
&:= 2^{8r+6}d^4(1+K_F+K_G)^4\E\left[(1+|X_0|)^{4\rho}\right](1+\mathring{c}_{4r+2}) +32d(d+2) \beta^{-2},
\end{split}
\end{align}
with 
$\kappa^\sharp_2$ and $\mathring{c}_{4r+2}$ given in Lemma \ref{lem:2ndpthmmt}.
\end{lemma}
\begin{proof}The proof follows exactly the same ideas as in the proof of \cite[Lemma A.2]{lim2021nonasymptotic}. To obtain the explicit constants, recall the expression of $H_\lambda$ given in \eqref{eq:expressiontH}-\eqref{eq:expressiontGF}. By Assumptions \ref{asm:AG} and \ref{asm:AF}, one obtains
\begin{align}\label{eq:Hlaubconvergence}
|H_{\lambda}(\theta, x)|
&\leq \sum_{i = 1}^d|G^{(i)}(\theta, x)|+d+|F(\theta,x)|
\leq d(1+K_G+K_F)(1+|x|)^{\rho}(1+|\theta|)^{2r+1}
\end{align}
for any $\theta \in \R^d$, $x \in \R^m$. Then, one upper bounds $H_{\lambda}$ using \eqref{eq:Hlaubconvergence}, and replaces \cite[Remark 2.2 and Lemma 4.2]{lim2021nonasymptotic} with Remark \ref{rem:growthHlliph} and Lemma \ref{lem:2ndpthmmt}, respectively.
\end{proof}
\begin{proof}[\textbf{Proof of Lemma \ref{lem:w1converp1} }]\label{proofw1converp1} The proof follows the exact same ideas as in the proof of \cite[Lemma 4.5]{lim2021nonasymptotic}. To obtain the explicit constants, let $0 <\lambda \leq \lambda_{\max}$ with $\lambda_{\max}$ given in \eqref{eq:stepsizemax}, $n \in \N_0$, $t \in (nT, (n+1)T]$. By using the arguments in \cite[Eq. (165)]{lim2021nonasymptotic} and by Remark \ref{rem:growthHlliph}, one observes, for any $s \in (nT, (n+1)T]$, that
\begin{equation}\label{eq:convergencew2ub1}
\E\left[  |h(\bar{\theta}^{\lambda}_s)- h(\bar{\theta}^{\lambda}_{\lfrf{s}})|^2 \right]
\leq  3^{4r-(1/2)}L_h^2  \left(\E\left[ 1+ | \bar{\theta}^{\lambda}_s|^{8r}  +|\bar{\theta}^{\lambda}_{\lfrf{s}}|^{8r} \right]\right)^{1/2}\left(\E\left[  | \bar{\theta}^{\lambda}_s -\bar{\theta}^{\lambda}_{\lfrf{s}}|^4 \right]\right)^{1/2}.
\end{equation}
Moreover, by \eqref{eq:expH}, \eqref{eq:expressiontH}, \eqref{eq:expressiontGF}, and by using $(a+b)^2 \leq 2a^2+2b^2$ for $a, b \geq 0$, one obtains for any $s \in (nT, (n+1)T]$ that
\begin{align}\label{eq:convergencew2ub2}
&\E\left[ | H(\bar{\theta}^{\lambda}_{\lfrf{s}},X_{\lcrc{s}})- H_{\lambda}(\bar{\theta}^{\lambda}_{\lfrf{s}},X_{\lcrc{s}})|^2\right]\nonumber\\
& \leq 2\sum_{i = 1}^d \E\left[ | G^{(i)}(\bar{\theta}^{\lambda}_{\lfrf{s}},X_{\lcrc{s}})- G^{(i)}_{\lambda}(\bar{\theta}^{\lambda}_{\lfrf{s}},X_{\lcrc{s}})|^2\right] +2\sum_{i = 1}^d \E\left[ | F^{(i)}(\bar{\theta}^{\lambda}_{\lfrf{s}},X_{\lcrc{s}})- F^{(i)}_{\lambda}(\bar{\theta}^{\lambda}_{\lfrf{s}},X_{\lcrc{s}})|^2\right] \nonumber\\
&\leq 4d\lambda \left(K_G^4+K_F^2\right)\E\left[(1+|X_{\lcrc{s}}|)^{4\rho}(1+|\bar{\theta}^{\lambda}_{\lfrf{s}}|)^{8r+4} \right] +4d\lambda\nonumber\\
&\leq 2^{8r+5}d\lambda \left(K_G^4+K_F^2\right)\E\left[(1+|X_{\lcrc{s}}|)^{4\rho}(1+|\bar{\theta}^{\lambda}_{\lfrf{s}}|^{8r+4} )\right] +4d\lambda,
\end{align}
where the second last inequality holds due to Assumptions \ref{asm:AG} and \ref{asm:AF}, 
and where the last inequality holds due to \eqref{eq:fundamentalineq} (with $l \leftarrow 2$, $z \leftarrow 8r+4$). 
Denote by $\mathcal{H}_t := \mathcal{F}^{\lambda}_{\infty} \vee \mathcal{G}_{\lfrf{t}}\vee\sigma(\theta_0), t \geq 0$. Then, following the arguments in \cite[Lemma 4.5]{lim2021nonasymptotic} up to \cite[Eq. (167)]{lim2021nonasymptotic}, but by using \eqref{eq:convergencew2ub1} and \eqref{eq:convergencew2ub2} instead of \cite[Eq. (165)]{lim2021nonasymptotic} and \cite[Eq. (166)]{lim2021nonasymptotic}, one obtains that 
\begin{align}\label{eq:convergencew2ub3}
 \begin{split}
& \E\left[|\bar{\zeta}^{\lambda, n}_t- \bar{\theta}^{\lambda}_t  |^2\right]
 \leq 4\lambda L_R \int_{nT}^t\E\left[|\bar{\zeta}^{\lambda, n}_s -\bar{\theta}^{\lambda}_s|^2 \right]\, \rmd s +4d\lambda L_R^{-1} \\
&\quad + 3^{4r-(1/2)}\lambda L_h^2L_R^{-1} \int_{nT}^t\left(\E\left[ 1+ | \bar{\theta}^{\lambda}_s|^{8r}+|\bar{\theta}^{\lambda}_{\lfrf{s}}|^{8r} \right]\right)^{1/2}\left(\E\left[  | \bar{\theta}^{\lambda}_s -\bar{\theta}^{\lambda}_{\lfrf{s}}|^4 \right]\right)^{1/2}\, \rmd s \\
&\quad + 2^{8r+5}d\lambda^2 \left(K_G^4+K_F^2\right)  L_R^{-1} \int_{nT}^t E\left[(1+|X_{\lcrc{s}}|)^{4\rho}(1+|\bar{\theta}^{\lambda}_{\lfrf{s}}|^{8r+4} )\right] \, \rmd s\\
&\quad -2\lambda \int_{nT}^t\E\left[ \E\left[ \left.\left\langle \bar{\zeta}^{\lambda, n}_s -\bar{\theta}^{\lambda}_{\lfrf{s}} ,  h(\bar{\theta}^{\lambda}_{\lfrf{s}})- H(\bar{\theta}^{\lambda}_{\lfrf{s}},X_{\lcrc{s}})\right\rangle\right| \mathcal{H}_s \right]\right]\, \rmd s\\
&\quad -2\lambda^2 \int_{nT}^t\E\left[  \left\langle  \int_{\lfrf{s}}^s H_{\lambda}(\bar{\theta}^{\lambda}_{\lfrf{r}},X_{\lcrc{r}}) \, \rmd r ,  h(\bar{\theta}^{\lambda}_{\lfrf{s}})- H(\bar{\theta}^{\lambda}_{\lfrf{s}},X_{\lcrc{s}})\right\rangle \right]\, \rmd s\\
&\quad +2\lambda\sqrt{2\lambda\beta^{-1}} \int_{nT}^t\E\left[  \left\langle  \int_{\lfrf{s}}^s \rmd B^{\lambda}_r ,  h(\bar{\theta}^{\lambda}_{\lfrf{s}})- H(\bar{\theta}^{\lambda}_{\lfrf{s}},X_{\lcrc{s}})\right\rangle \right]\, \rmd s.
\end{split}
\end{align}
By Remark \ref{rem:growthHlliph} and \eqref{eq:Hlaubconvergence}, 
one obtains the following estimate for the sixth term on the RHS of \eqref{eq:convergencew2ub3}:
\begin{align}
& -2\lambda^2 \int_{nT}^t\E\left[ \left| \left\langle  \int_{\lfrf{s}}^s H_{\lambda}(\bar{\theta}^{\lambda}_{\lfrf{r}},X_{\lcrc{r}}) \, \rmd r ,  h(\bar{\theta}^{\lambda}_{\lfrf{s}})- H(\bar{\theta}^{\lambda}_{\lfrf{s}},X_{\lcrc{s}})\right\rangle \right|\right]\, \rmd s\nonumber\\
\begin{split}\label{eq:convergencew2ub4}
&\leq  2^{8r+3}\left(d^2(1+K_G+K_F)^2+4L_h^2+2K_H^2\right)\lambda^2 \int_{nT}^t\E\left[(1+|X_{\lcrc{s}}|)^{2\rho}(1+|\bar{\theta}^{\lambda}_{\lfrf{s}}|^{8r+4}  )\right]\,\rmd s\\
&\quad +4\lambda |h(0)|^2.
\end{split}
\end{align}
By Lemma \ref{lem:ose4thpower}, \eqref{eq:convergencew2ub4}, the fact that the fifth and seventh term of the RHS of \eqref{eq:convergencew2ub3} are zero, and that $X_{\lcrc{s}}$ and $\bar{\theta}^{\lambda}_{\lfrf{s}}$ are independent for any $s \geq 0$, one obtains
\begin{align*}
& \E\left[|\bar{\zeta}^{\lambda, n}_t- \bar{\theta}^{\lambda}_t  |^2\right] \\
& \leq 4\lambda L_R \int_{nT}^t\E\left[|\bar{\zeta}^{\lambda, n}_s -\bar{\theta}^{\lambda}_s|^2 \right]\, \rmd s +4d\lambda L_R^{-1}+4\lambda |h(0)|^2 \\
&\quad + 3^{4r-(1/2)}\lambda^2 L_h^2L_R^{-1} \int_{nT}^t\left(\E\left[ 1+ | \bar{\theta}^{\lambda}_s|^{8r}+|\bar{\theta}^{\lambda}_{\lfrf{s}}|^{8r} \right]\right)^{1/2}  \left(e^{-\lambda  a_F \kappa^\sharp_2\lfrf{t} }\bar{C}_{0,1}\E\left[|\theta_0|^{4(2r+1)}\right]+\bar{C}_{1,1} \right)^{1/2}\, \rmd s \\
&\quad + 2^{8r+5}d\lambda^2 \left(K_G^4+K_F^2\right)  L_R^{-1} \int_{nT}^t E\left[(1+|X_0|)^{4\rho}\right] E\left[(1+|\bar{\theta}^{\lambda}_{\lfrf{s}}|^{8r+4} )\right] \, \rmd s \\
&\quad + 2^{8r+3}\left(d^2(1+K_G+K_F)^2+4L_h^2+2K_H^2\right)\lambda^2\int_{nT}^t\E\left[ (1+ |X_0|)^{2\rho}\right] \E\left[ (1+|\bar{\theta}^{\lambda}_{\lfrf{s}}|^{8r+4} )\right] \, \rmd s.
\end{align*}
This implies, by using Lemma \ref{lem:2ndpthmmt} and by using $1- \nu \leq e^{-\nu}$ for any $\nu \in \R$, that
\begin{align*}
 \E\left[|\bar{\zeta}^{\lambda, n}_t- \bar{\theta}^{\lambda}_t  |^2\right]
& \leq 4\lambda L_R \int_{nT}^t\E\left[|\bar{\zeta}^{\lambda, n}_s -\bar{\theta}^{\lambda}_s|^2 \right]\, \rmd s +4d\lambda L_R^{-1}+4\lambda |h(0)|^2 \\
&\quad + 3^{4r-(1/2)}\lambda^2 L_h^2L_R^{-1} \int_{nT}^t\left( 1+ 2e^{-\lambda  a_F \kappa^\sharp_2\lfrf{s} }\E[|\theta_0|^{8r}]+2\mathring{c}_{4r} \right)^{1/2}  \\
&\qquad \times\left(e^{-\lambda  a_F \kappa^\sharp_2\lfrf{s} }\bar{C}_{0,1}\E\left[|\theta_0|^{4(2r+1)}\right]+\bar{C}_{1,1} \right)^{1/2}\, \rmd s \\
&\quad +\lambda^2\left( 2^{8r+5}d \left(K_G^4+K_F^2\right)  L_R^{-1}  + 2^{8r+3}\left(d^2(1+K_G+K_F)^2+4L_h^2+2K_H^2\right)\right) \\
&\qquad \times \E\left[(1+|X_0|)^{4\rho}\right]\int_{nT}^t\left( 1+ e^{-\lambda  a_F \kappa^\sharp_2\lfrf{s} }\E[|\theta_0|^{8r+4}]+ \mathring{c}_{4r+2} \right) \, \rmd s.
\end{align*}
By using $\lfrf{s}  \geq nT$ and $1/2 \leq \lambda T \leq 1$, the above inequality becomes
\begin{align*}
 \E\left[|\bar{\zeta}^{\lambda, n}_t- \bar{\theta}^{\lambda}_t  |^2\right]
&\leq 4\lambda L_R \int_{nT}^t\E\left[|\bar{\zeta}^{\lambda, n}_s -\bar{\theta}^{\lambda}_s|^2 \right]\, \rmd s +e^{-4L_R}\lambda \left(e^{- n  a_F \kappa^\sharp_2/2 } \bar{C}_0 \E\left[V_{4(2r+1)}(\theta_0)\right] +\bar{C}_1 \right),
\end{align*}
where
\begin{align}
\begin{split}\label{eq:w1converp1consts}
\kappa^{\sharp}_2& := \min\{\bar{\kappa}(2), \tilde{\kappa}(2)\},\\
\bar{C}_0&:=  e^{4L_R}\left( 2^{8r+5}d \left(K_G^4+K_F^2\right)  L_R^{-1}  +2^{8r+3}\left(d^2(1+K_G+K_F)^2+4L_h^2+2K_H^2\right)\right)\E\left[(1+|X_0|)^{4\rho}\right]\\
&\quad +e^{4L_R}3^{4r-(1/2)}  L_h^2L_R^{-1}\bar{C}_{0,1},\\
\bar{C}_1&:=e^{4L_R}\left( 2^{8r+5}d \left(K_G^4+K_F^2\right)  L_R^{-1}  +2^{8r+3}\left(d^2(1+K_G+K_F)^2+4L_h^2+2K_H^2\right)\right) \E\left[(1+|X_0|)^{4\rho}\right] \\ &\quad \times( \mathring{c}_{4r+2} +1)  +e^{4L_R}3^{4r-(1/2)}  L_h^2L_R^{-1}   \left( \bar{C}_{1,1} +2\mathring{c}_{4r} +1\right)+e^{4L_R}(4d  L_R^{-1}+4  |h(0)|^2),
\end{split}
\end{align}
with $\bar{C}_{0,1}$, $\bar{C}_{1,1}$ given in \eqref{eq:ose4thpowerconsts}, $\bar{\kappa}(2)$, $\tilde{\kappa}(2)$, $\mathring{c}_{4r}$, $\mathring{c}_{4r+2}$ given in Lemma \ref{lem:2ndpthmmt}. Finally, the desired result can be obtained by applying Gr\"{o}nwall's lemma, which implies that
\[
\E\left[|\bar{\zeta}^{\lambda, n}_t- \bar{\theta}^{\lambda}_t  |^2\right] \leq \lambda \left(e^{- n  a_F \kappa^\sharp_2/2 } \bar{C}_0 \E\left[V_{4(2r+1)}(\theta_0)\right] +\bar{C}_1 \right).
\]
\end{proof}
\bibliographystyle{plainnat}

\bibliography{references}
\end{document}